\documentclass[12pt,english,reqno]{amsart}

\usepackage[utf8]{inputenc}      % accents encodés en UTF-8
\usepackage[headings]{fullpage} 
 
\usepackage[english]{babel}
\usepackage[T1]{fontenc}         % polices accentuées
\usepackage{lmodern}             % ... et vectorielles

\usepackage{graphicx}
\usepackage{url}
\usepackage{amsmath}           %
\usepackage{amsthm}
\usepackage{amsfonts}            %
\usepackage{amssymb}             % 

\usepackage{mathrsfs}

\usepackage{tikz-cd}

\usepackage[hidelinks]{hyperref}

\newtheorem{mythm}{Theorem}[section]

\newtheorem{mylemma}[mythm]{Lemma}
\newtheorem*{mythmintro}{Theorem}

\newtheorem{myptn}[mythm]{Proposition}
\theoremstyle{definition}
\newtheorem{mydef}[mythm]{Definition}
\newtheorem{mynotation}[mythm]{Notation}

\newtheorem{myremark}[mythm]{Remark}

\DeclareMathOperator{\Sym}{\mathrm{Sym}}

\DeclareMathOperator{\Hom}{\mathrm{Hom}} % Set of homomorphisms 
\DeclareMathOperator{\Spec}{\mathrm{Spec}} % Spectrum of a ring 
\DeclareMathOperator{\Res}{\mathrm{Res}} % Weil restriction 

\DeclareMathOperator{\Pic}{\mathrm{Pic}} % Picard group
\DeclareMathOperator{\vol}{\mathrm{vol}}
\DeclareMathOperator{\pr}{\mathrm{pr}}
\DeclareMathOperator{\Supp}{\mathrm{Supp}}
\DeclareMathOperator{\res}{\mathrm{res}} % Residue 
\DeclareMathOperator{\rg}{\mathrm{rg}} 
\DeclareMathOperator{\ord}{\mathrm{ord}} 
\DeclareMathOperator{\ddiv}{\mathrm{div}}  

\newcommand{\KVar}{\mathrm{KVar}}
\newcommand{\KExpVar}{\mathrm{KExpVar}}
\newcommand{\ExpM}{\mathscr{E}\mathit{xp}\mathscr{M}}

\newcommand{\Cl}{\mathrm{Cl}}  
\newcommand{\aA}{\mathbf{a}}  
\newcommand{\m}{\mathsf{m}}  
\newcommand{\mface}{\mathsf{M}} %(choice of) maximal face(s)  
\newcommand{\dd}{\mathrm{d}}  
\newcommand{\ee}{\mathrm{e}}  
\newcommand{\mm}{{\mathbf{m}}}  
\newcommand{\nn}{{\mathbf{n}}}  
\newcommand{\pp}{{\mathbf{p}}}  
\newcommand{\CC}{\mathbf{C}}
\newcommand{\PP}{\mathbf{P}}
\newcommand{\ZZ}{\mathbf{Z}}
\newcommand{\NN}{\mathbf{N}}
\newcommand{\TT}{\mathbf{T}}
\newcommand{\LL}{\mathbf{L}}

\newcommand{\AAA}{\mathscr{A}}
\newcommand{\BBB}{\mathscr{B}}
\newcommand{\DDD}{\mathscr{D}}
\newcommand{\EEE}{\mathscr{E}}
\newcommand{\FFF}{\mathscr{F}}
\newcommand{\GGG}{\mathscr{G}}
\newcommand{\HHH}{\mathscr{H}}
\newcommand{\LLL}{\mathscr{L}}
\newcommand{\OOO}{\mathscr{O}}
\newcommand{\PPP}{\mathscr{P}}
\newcommand{\UU}{\mathbf{U}}
\newcommand{\UUU}{\mathscr{U}}
\newcommand{\XXX}{\mathscr{X}}
\newcommand{\YYY}{\mathscr{Y}} 
\newcommand{\ZZZ}{\mathscr{Z}} 
\newcommand{\BU}{\mathscr{B}^\mathscr{U}}

\begin{document}

\title[Geometric Batyrev-Manin-Peyre \& compactifications of additive groups]{Geometric Batyrev-Manin-Peyre \\for equivariant compactifications \\of additive groups}

\author[L. Faisant]{Loïs \textsc{Faisant}}
\date{\today}
\address{Institut Fourier\\UMR 5582\\
Universit\'e Grenoble Alpes\\
CS 40700\\ 38058 Grenoble CEDEX 09\\ France}
\email{lois.faisant@univ-grenoble-alpes.fr}


\makeatletter
\@namedef{subjclassname@2020}{
 \textup{2020} Mathematics Subject Classification}
\makeatother

%\subjclass{14H10, 14G10}

%%%%%%%%%%%%
%%% TITLE %%%%%
%%%%%%%%%%%%

%\def\smfbyname{} 
\maketitle

%%%%%%%%%%%%
%%% ABSTRACT %%%%%
%%%%%%%%%%%%

\begin{abstract}
Building on previous works by Bilu, Chambert-Loir and Loeser, we study 
the asymptotic behaviour of the moduli space of sections  of a given family over a smooth projective curve, 
assuming that the generic fiber is an equivariant compactification of a finite dimensional vector space. Working in a suitable Grothendieck ring of varieties,
we show that the class of these moduli spaces converges, modulo an adequate normalisation, to a non-zero effective element, 
when the class of the sections goes arbitrary far from the boundary of the dual of the effective cone.
The limit can be interpreted as a motivic Euler product in the sense of Bilu's thesis. This result provides a positive answer to a motivic version of the Batyrev-Manin-Peyre conjectures in this particular setting. 
\end{abstract}

%%%%%%%%%%%%
%%% TOC %%%%%%
%%%%%%%%%%%% 

%\setcounter{tocdepth}{2}
%\tableofcontents

%%%%%%%%%%%%%%%%
%%% SECTION %%%%
%%%%%%%%%%%%%%%%

\section*{Introduction}

The link between the structure of higher dimensional algebraic varieties and the rational curves on such varieties has been fruitfully exploited during the last decades \cite{kollar1996rational,debarre2001higher}.
In particular, this approach is relevant when one studies varieties admitting many rational curves, such as Fano varieties, by the work of Mori \cite{mori1979projective}. 
Covering families of curves are mobile and it is known that the cone of mobile curves is exactly the dual of the effective cone, by the theorem of Boucksom, Demailly, P\v{a}un and Peternell \cite{boucksom2012pseudo}.  
Starting from this remark, 
a new direction consists in studying the asymptotic behaviour of the moduli space of curves on a variety, 
when the class of the curves goes arbitrary far from the boundary of the dual of the effective cone. 
This approach finds another source of inspiration in 
the Batyrev-Manin-Peyre conjectures \cite{batyrev1990nombre,peyre1995hauteurs} which originally concern the estimation of the number of rational points of bounded height on an algebraic variety defined over a number field, when the bound tends to infinity. 
It is then natural to consider analogous questions over function fields, which amounts to studying the morphisms from a given projective curve $C$ to a given \textit{sufficiently nice} variety, or more generally the sections of a family over $C$. Roughly speaking, in this context the notion of height translates into the notion of degree of a curve,
and the height zeta function one may introduce in the arithmetic setting becomes a motivic height zeta function; its coefficients are the classes, in a Grothendieck ring of varieties, 
of the moduli spaces of sections of a given degree. The properties of such a generating series have been studied by Bourqui \cite{bourqui2009produit} in the case of toric varieties 
and by Chambert-Loir-Loeser \cite{chambert2016motivic} and Bilu \cite{bilu2018motivic} in the case of compactifications of vector spaces. 

In this paper we refine the study of the case of compactifications of vector spaces. More precisely, we study the asymptotic behaviour of the moduli space of sections of a family over a projective curve whose generic fibre is a {compactification} of a power of the additive group. 
One can attach to such a section a class in the dual of the Picard group of the generic fibre, called the \textit{multidegree} of the section, and consider the class of the moduli space of sections of a given multidegree in the Grothendieck ring of varieties. 
The main result of this article is the following:  if we normalise this last class by the class of the affine space of dimension the anticanonical degree of the curves, then it converges, when the multidegree tends to infinity \textit{in a proper way},  to an effective element of the Grothendieck ring of varieties. 
In comparison, previous results only provided information concerning the moduli space of sections with a given degree relatively to the generic log-anticanonical line bundle of the family. 
By writing \textit{in a proper way}, we mean that the multidegree goes arbitrary far from the boundary %of a cone spanned by a finite number of effective divisors on the family. 
of the dual of the effective cone.
% This cone can be understood as a translation of the effective cone of the generic fibre. 
Furthermore,  the limit of the normalised classes of moduli spaces can be written as a motivic Euler product in the sense of Bilu's thesis \cite{bilu2018motivic}. 
It has the expected form in comparison with the predictions of the Batyrev-Manin-Peyre conjectures in the arithmetic context \cite{peyre1995hauteurs} as well as in the geometric (or motivic) context \cite{bourqui2009produit,bourqui2010fonctions}.  
Our result also deals with the case of sections whose restriction to a fixed open subset of the curve has its image contained in a partial compactification of the vector space,  
a condition which requires the introduction of additional finer numerical invariants in order to distinguish the relevant components of the moduli space.  

Since we will use properties of a certain motivic height zeta function mentioned above, we adopt the framework of the last chapter of \cite{bilu2018motivic} which generalises the situation studied by Chambert-Loir and Loeser in \cite{chambert2016motivic}.
We consider a quasi-projective smooth connected curve
$C_0$ defined over an algebraically closed field $k$  of characteristic zero and $C$ its smooth projective compactification, of genus $g$ and field of rational functions $F=k(C)$.  
The complement of $C_0$ in $C$ is a finite scheme $S$. 
We take $G$ to be the additive group scheme $ \mathbf G_a^n $ for a given positive integer $n$ and we consider
a projective irreducible $k$-scheme $\XXX$ together with a non-constant morphism $\pi : \mathscr X \to C$  satisfying the following assumptions. 
We assume that the generic fibre $X=\XXX_F$ is a smooth equivariant compactification of $G_F$, meaning that  $X$ is a smooth projective scheme over $F$ containing a dense open subset isomorphic to $G_F$, such that the group law of $G_F$ extends to an action of $G_F$ on $X$.
Let $\UUU$ be a Zariski open subset of $\XXX$. Similarly, $U=\UUU_F$ is assumed to be stable under the action of $G_F$.  
We denote by $D$ the complement of $U$  in $X$.
Our goal is to study the asymptotic behaviour 
of sections of such a family $\pi : \XXX \to C$,
in a sense we make more precise in what follows. 

By \cite[Theorem 2.7]{tschinkel1999geometry}, the boundary $X\setminus G_F$ is a divisor whose irreducible components $(D_\alpha)_{\alpha \in \AAA}$ freely generate the Picard group of $X$, as well as its effective cone. % Moreover, up to scalar multiplication, there exists an unique $G_F$-invariant rational differential form on $X$
There exist integers ${ \rho_\alpha \geqslant 2 }$ such that an anticanonical divisor is given by $\sum_\alpha \rho_\alpha D_\alpha$; in particular, it is big. 
A log-anticanonical divisor with respect to $D$ is then $\sum_{\alpha \in \AAA} \rho_\alpha ' D_\alpha $ where ${ \rho_\alpha ' = \rho_\alpha  - 1 }$ if $D_\alpha$ is an irreducible component of $D$ and ${ \rho_\alpha ' = \rho_\alpha }$ otherwise. 
 
For every $\alpha $ in $ \AAA$, there is a line bundle $\mathscr L_\alpha$ on $\XXX$ extending $D_\alpha$.
Given a tuple of integers ${ \nn = (n_\alpha)_{\alpha \in \AAA} }$, the moduli space $M_\nn$ is the space parametrizing sections ${ \sigma : C \to \XXX }$ such that:
\begin{itemize}
	\item $\sigma$ maps the generic point $\eta_C$ of $C$ to a point of $G_F$;
	\item the image of $C_0$ by $\sigma$ is contained in $\UUU$;
	\item for all $\alpha $ in  $\AAA$, $\deg \left ( \sigma ^* \mathscr L_\alpha \right ) = n_\alpha $.
\end{itemize}
The existence of these moduli spaces follows from \cite[Proposition 2.2.2]{chambert2016motivic}. 
In order to avoid any local obstruction to the existence of such sections, we assume that local sections exists: we suppose that for every closed point $v$ of $C_0$ the intersection of $G(F_v) $ with $ \UUU ( \mathcal O_v )$ is non-empty, where $F_v$ is the completion of $F$ at $v$ and $\mathcal O_v $ is its ring of integers. 

Let $\mathscr M_k$ be the localisation of the Grothendieck ring of varieties over $k$ at the class $\LL$ of the affine line, that is $\mathscr M_k = \KVar_k [\LL^{-1}]$. 
Our focus will be on studying the asymptotic behaviour of the class $[M_\nn]$ in $\mathscr M_k$ endowed with the weight topology of \cite{bilu2018motivic}, following a question raised by Peyre \cite[Question 5.4]{peyre2021beyond}. 

Any section $\sigma : C \to \XXX$ gives rise to an element of $\Pic ( X )^\vee$. It is explicitly given by
\[
\nn_\sigma : \sum_{\alpha \in \AAA} \lambda_\alpha \left [ D_\alpha \right ] \mapsto \sum_{\alpha \in \AAA} \lambda_\alpha \deg ( \sigma^* \LLL_\alpha ) 
\]
and this defines a pairing 
\[
\langle \, \cdot \, , \, \cdot \, \rangle  : \Pic ( X ) \otimes \Pic ( X )^\vee \to \ZZ . 
\]
Let $\LLL_{\rho ' } = \otimes_{\alpha\in \AAA} \rho_\alpha ' \LLL_\alpha$ be the generic log-anticanonical line bundle. 
By the previous pairing, any section $\sigma$ parametrized by $M_\nn$ satisfies $\nn_\sigma = \nn$
and we will write $\langle \rho ' , \nn \rangle $ for  $ \langle \sum_{\alpha} \rho_\alpha ' [ D_\alpha ] , \nn_\sigma \rangle $. 
For now we restrain ourself to a particular case of our main result.  
\begin{mythmintro}
Let $\pi  : \XXX \to C $ be a family as above, with $\UUU = \XXX$. Then 
\[
[ M_{\nn } ]\LL ^{- \langle \rho , \nn \rangle }
\]
converges in the completion of $\mathscr M _ k  $ for the weight topology, when $\min ( n_\alpha )_{\alpha \in \AAA} \to \infty$. Furthermore,
the limit is a non-zero effective element
which can be interpreted as a motivic Euler product in the sense of Bilu's thesis \cite{bilu2018motivic}.
\end{mythmintro}
The condition $\min ( n_\alpha ) \to \infty$ actually means that the class of the sections in the dual of the effective cone goes arbitrary far from its boundaries in $\Pic ( \XXX )^\vee $
 (this class of curves is the \textit{multidegree} mentioned before). However, note that this condition is not strong enough to fully distinguish components of the moduli space of sections when $\UUU$ is not necessarily equal to $ \XXX$: one has to introduce additional invariants, corresponding to local intersection degrees above $S$,
 and take them to be arbitrarily large as well. 
We refer to Notation \ref{def:euler-product} for the precise definition of such a motivic Euler product and to Proposition \ref{ptn:final-result} and Theorem \ref{thm:final-result} for our most general result.  
For example, still assuming in this introduction $\UUU = \XXX$ and $C=C_0$, then the convergence can be understood as 
\[
[ M_\nn ]\LL ^{- \langle  \rho , \nn \rangle } 
\longrightarrow
\frac{\LL^{( 1 - g ) \dim (X)  }}{(1-\LL^{-1})^{\rg (\Pic (X))}}
 \prod_{v\in C} \mathfrak{c}_v  \qquad \text{when} \min_{\alpha \in \AAA} ( n_\alpha ) \to \infty
\]
where for all but a finite number of places $v$
\[
\mathfrak{c}_v  =  (1-\LL^{-1})^{\rg (\Pic (X))} \frac{[\XXX_v]}{\LL^{\dim (X)}} . 
\]
We give in Remarks \ref{remark:interpretation-local-term-volume} and \ref{remark:interpretation-local-term-volume-integral-points} a complete interpretation of those local factors in terms of motivic volume. 

This work is based on the study of the multivariate motivic height zeta function 
\[
Z ( \mathbf T ) = \sum_{\nn \in \mathbf Z^{\AAA} } [M_\nn] \mathbf  T^{\nn} 
\]
where
$\TT $ is the family of indeterminates $( T_\alpha )_{\alpha \in \AAA} $ and thus $\TT^{\nn} $ stands for the product $ \prod_{\alpha \in \AAA} T_\alpha ^{n_\alpha }$ for all $\nn \in \mathbf Z^{\AAA}$. It is proven in \cite[Chapter 6]{bilu2018motivic} that this function can be written as a sum over a space of characters.
From this decomposition, Bilu deduces a description of the first pole of the restriction of the motivic height zeta function to the line generated by the canonical line bundle. 
Here, we push her method one step further to be able to distinguish the components given by all possible multidegrees.  
In order to separate components arising from the integrality conditions, introducing new variables and refining several tools due to Bilu is found to be necessary. 

In the first section of this paper we start by collecting some definitions and classical results about Grothendieck rings of varieties. We shortly recall what are symmetric products of classes in the Grothendieck ring of varieties, leading to Bilu's definition of the motivic Euler product \cite{bilu2018motivic}. For completeness, we also add  a motivic Fourier transform together with a motivic Poisson formula to this list of reminders.  
The second section is devoted to a study of the convergence of such Euler products. 
We briefly recall the main properties of the weight on the Grothendieck ring of varieties, include a useful convergence criterion for Euler products and quickly study sums of families in this context. 
In the third part of this work we present a bunch of convergence results in the particular setting of compactifications of vector spaces.  
%We finally give a detailed proof of our result in the last section.
The fourth and last section is the heart of the proof of our result. For clarity's sake, we start with treating the case of the projective line and $\UUU = \XXX$, then generalise our argument to any smooth projective curve, and finally treat the general case $\UUU \subset \XXX$. 

%%% SUBSECTION %%%

%%%%%%%%%%%%%%%%
%%% SECTION %%%%
%%%%%%%%%%%%%%%%

\section{Preliminaries}
%%% SUBSECTION %%%

\subsection{Rings of varieties} \label{paragraph:ring-of-varieties}
Along this paper we work in Grothendieck rings of varieties, or in localisations and completions of such rings. 
Here we collect some definitions and properties. 
References for this paragraph are the paper of Denef and Loeser \cite{denef1999germs}, the second chapter of the book by Chambert-Loir, Nicaise and Sebag \cite{chambert2018motivic} and the articles of Chambert-Loir and Loeser \cite{chambert2016motivic} and Clucker and Loeser \cite{cluckers2010constructible}.  

Let $R$ be a Notherian scheme. 
The Grothendieck group of varieties $\KVar _R$ is defined as the abelian group generated by all $R$-varieties (that is, $R$-schemes of finite presentation), with relations 
\[
X - Y
\]
whenever $X$ and $Y $ are isomorphic $R$-varieties and 
\[
X - Y - U
\]
whenever $X$ is an $R$-variety, $Y $ is a closed subscheme of $X$ and $U$ is its open complement in $X$. 
The product $[X][Y]=[X\times_R Y]$ defines a ring structure on $\KVar _R $ with unit element the class of $R$ over itself with the identity structural map. 
The class of the affine line $\mathbf A^1 _R $  in $\KVar_R $ is denoted by $\LL$ and 
the localised Grothendieck ring of varieties $\KVar_R [\LL ^{-1}] $  by $ \mathscr M_R  $. Any constructible subset $X$ of a $R$-variety admits a class $[X]$ in such rings (see e.g. page 59 in \cite{chambert2018motivic}). In our case, a constructible subset is a finite union of locally closed subsets. 

The Grothendieck ring of varieties with exponentials $\KExpVar_R$ is defined in a similar way. Its generators are pairs of $R$-varieties $X$ together with morphisms $f : X \to \mathbf A^1 = \Spec ( \mathbf Z [T] )$. Relations are the isomorphism relation 
\[
(X , f ) - ( Y , f \circ u )
\]
whenever $X,Y$ are $R$-varieties, $f:X\to \mathbf A^1 $ a morphism and $u: Y \to X $ an isomorphism of $R$-varieties; the scissors relation
\[
(X, f ) - ( Y , f_{| Y} ) - ( U , f_{|U} ) 
\]
whenever $X$ is a $R$-variety, $Y$ a closed subscheme of $X$, $U$ its complement in $X$ and $f : X \to \mathbf A^1$ a morphism; and finally the relation 
\[
( X \times_\mathbf Z \mathbf A^1 , \pr_2 )
\]
whenever $X$ is a $R$-variety and $\pr_2$ is the second projection. If $X$ is a constructible set and $f : X \to \mathbf A^1 $ a piecewise morphism (given by the datum of morphisms $f_i : X_i \to \mathbf A ^1 $ on a partition $(X_i)_{1\leqslant i\leqslant m }$ of $X$ into locally closed subsets) then the class $[X,f]$ is well-defined. The class $[\mathbf A^1 _R , 0 ]$ is denoted $\LL$. 
Sending a class of a $R$-variety $X$ to the class $[X,0]$ defines a morphism of Abelian groups $\iota : \KVar_R \to \KExpVar_R$ sending $\LL$ to $\LL$. 
The product 
\[
[X,f][Y,g]=[X\times_R Y , f \circ \pr_1 +  g \circ \pr_2 ]
\]
defines a ring structure on $\KExpVar_R$, with unit element $[\Spec ( R ) , 0 ]$. Then the morphism $\iota$ is actually a morphism of rings. Localising at $\LL$, one gets a ring $\ExpM_R  $ together with a ring morphism $\iota : \mathscr M_R \to \ExpM_R  $. 

Any morphism $u : R \to S$ of $k$-varieties induces a morphism of group 
\[
u_! : \KExpVar_R \to \KExpVar_S 
\]
sending any effective element $[X,f]_R$ to $[X, f]_S$, where we view $X$ as an $S$-variety through $u$. If $u$ is a immersion, $u_!$ is a morphism of rings. The morphism $u$ induces as well a morphism of rings in the other direction
\[
u^* : \KExpVar_S \to \KExpVar_R 
\]
sending any effective element $[X,f]_S$ to $[X \times_S R , f \circ \pr_X ]_R $. If $T$ is another $S$-variety,
combining pull-backs and product, one obtains an exterior product
\[
\boxtimes : \KExpVar_{R} \times \KExpVar_{T} \overset{ \pr_R^* \pr_T^*}{\longrightarrow}   \KExpVar_{R \times_S T} . 
\] 

We conclude this subsection by introducing an analogue of the exponential sums of characters over finite field. Assume for a while that $k$ is a finite field and $\psi : k \to \CC^*$ is a non-trivial character. Then the exponential sum associated to a pair $[X,f]_k =[X,f]_{\Spec ( k ) } $ is 
\[
\sum_{x\in X(k)} \psi ( f (  x ) ) .
\]
Let $S$ be a $k$-variety, together with a morphism $u : S \to \mathbf A^1  $, and $g :  X\to S$ a variety over $S$ together with a morphism $f : X \to \mathbf A^1 $.
\begin{center}
\begin{tikzcd}
										& 	\mathbf A^1	& \mathbf A^1 \\
X \arrow[r,"g"] \arrow[dr] \arrow[ur,"f"] & S  \arrow[d]  \arrow[ur,"u"]		& 						\\
    &  \Spec ( k )  & 										& 
\end{tikzcd}
\end{center} 
We write $\theta = [ X , f ]_S$.
Then over a point $s$ we can introduce the sum over the fibre 
\[
\theta (s) = \sum_{\substack{ x\in X ( k ) \\ g ( x ) = s }} \psi ( f ( x ) ) .
\]
Using the additivity of $\psi$ we decompose fibre by fibre the sum 
\[
\sum_{s\in S (  k )} \theta ( s  ) \psi ( u (s) ) = \sum_{s\in S (  k )} \sum_{\substack{ x\in X ( k ) \\ g ( x ) = s }} \psi ( f ( x ) ) \psi (  u ( s) ) = \sum_{x\in X ( k )} \psi ( f (x) + u ( g ( x) ) )  
\]
which is the exponential sum associated to the pair $[X \times_S S , f + u \circ g ]_S =   [ X , f ]_S  [ S , u ]_S$ viewed in $\KExpVar_k$.
\begin{center}
\begin{tikzcd}
X \times_S S \arrow[r,"\sim"] \arrow[d] & X \arrow[r,"f"] \arrow[dl,"g"] & \mathbf A^1 \\
S \arrow[r,"u"] 										 & \mathbf A^1  						& 				
\end{tikzcd}
\end{center} 
Thus in general we define the \textit{sum over rational points }
\begin{equation}\label{def-equation:exponential-sum-notation}
\sum_{x\in S} \theta ( s )  \psi ( u ( s ) ) 
\end{equation}
for any $\theta \in \ExpM_S$ and $u:S \to \mathbf A^1$ 
as the image by $\KExpVar_S \to \KExpVar_k$ of the class $\theta \cdot [ S , u ]_S$
(we may sometimes omit the exponential factor $\psi ( u (s ))$ when writing the sum). This class is explicitly given by $[X, f + u \circ g ] $ when $\theta = [X , f ]_S$ and $g:X \to S$ is a $S$-variety. 
This notation easily extends to the relative setting when $v : S \to T $ is a morphism of varieties over $k $  : we write
\[
v_! \theta = \sum_{s \in S / T } \theta ( s ) \psi ( u ( s )) 
\]
where $v_! : \KExpVar_S \to \KExpVar_T$ is the morphism of groups induced by $v$. 

%%% SUBSECTION %%%
\subsection{Symmetric products} 
Let $R$ be a variety over a perfect field $k$, $X$ a variety over $R$ and $\XXX = (X_i)_{i  \in I}$ be a family of quasi-projective varieties above $X$, where $I$ is an arbitrary set. Let  $\pi = (n_i)_{i\in I} \in \NN^{(I)} $ be a family of non-negative integers with finite support, which we call a \textit{partition}. This terminology is justified by particular cases of set $I$: for example if $I=\NN^r \setminus \{ \mathbf 0 \} $ for some given $r\in \NN ^*$, then a partition of a non-zero tuple of non-negative integers $\nn \in I $ is a family $(n_\mathbf i)_{\mathbf i \in I}$ such that $\sum_{\mathbf i \in I } \mathbf i n_i = \nn$.  
%a positive integer $n$ and a partition $\pi$ of $n$, that is, $\pi = ( n_i )_{i\geqslant 1} \in  \nn ^{(\nn^* )}$ such that $\sum_i in_i = n$, 

Given a partition $\pi$, one can define an $R$-variety $S^\pi ( \XXX / R) $ called the ($\pi$-th) symmetric product of $\XXX$. 
Its geometric points can be understood as collections $(D_i )_{i\in I}$ of effective zero-cycles such that $D_i \in S^{n_i} X_i$ for all $i\in I$
and the union of all the images in $X$ of the supports of the $D_i $'s  consists of exactly $ \sum_{i\in I} n_i$ distinct geometric points  \cite[\S 3.1 \& \S 3.2]{bilu2018motivic}. 
This is formally defined as follows: if $\left (\prod_{i\in I} X^{n_i}\right )_*$ is the complement of the diagonal of $\prod_{i\in I} X^{n_i} $, 
in other words $\left (\prod_{i\in I} X^{n_i}\right )_*$ is the set of points with pairwise distinct coordinates, then
there is a Cartesian diagram
\begin{center}
\begin{tikzcd}
\left (\prod_{i\in I} X_i^{n_i}\right )_* \arrow[r , hook] \arrow[d] & \prod_{i\in I} X_i^{n_i}  \arrow[d] \\
\left (\prod_{i\in I} X^{n_i}\right )_* \arrow[r , hook]  &  \prod_{i\in I} X^{n_i} 
\end{tikzcd}
\end{center}
defining an open subset $\left (\prod_{i\in I} X_i^{n_i}\right )_*$ of points of $\prod_{i\in I} X_i^{n_i}  $ mapping to points of $\prod_{i\in I} X^{n_i} $ with pairwise distinct coordinates. Now one considers the natural action of the symmetric groups $\mathfrak{S}_{n_i}$ on each of the $X^{n_i}_i$'s and defines the symmetric product 
\[
S^\pi  ( \XXX / R ) = \left (\prod_{i\in I} X_i^{n_i}\right )_* / \prod_{i\in I} \mathfrak{S}_{n_i} 
\]
which is well-defined since $\XXX$ is a family of quasi-projective varieties over $X$. 
One can go one step further and define the symmetric product $S^\pi ( \XXX / R)$ of a family $(X_i)_{i\in I}$ of non-effective classes of $\KVar_{X} $ as well as symmetric products of varieties with exponentials, in a way compatible with the definition of symmetric product of $X$-varieties \cite[\S 3.5-3.7]{bilu2018motivic}. 
If $\nn \in \NN^r \setminus \{ 0 \}$, then $S^\nn ( \XXX / R ) $ is by definition the disjoint union of all the  $S^\pi ( \XXX / R ) $ for every partition $\pi $ of $\nn$. 
In case all the $X_i$ are equal to $X$, we obtain a decomposition of the classical symmetric power $S^n X$ of $X$ into locally closed subsets $S^\pi X $ \cite[\S 3.1.1]{bilu2018motivic}. 

Let $\XXX = ( \mathfrak{a}_i )_{i\in I} $ be a family of elements of $\KVar_X$ and $\YYY = ( \mathfrak{b}_i )_{i\in I} $, $\ZZZ = (  \mathfrak{c}_i )_{i\in I} $ the families obtained by restriction respectively to a closed subscheme $Y\subset X$ and to its complement $U$. 
Then
\[
S^\pi ( \XXX )  = \sum_{\pi ' \leqslant \pi } S^{\pi ' } ( \YYY ) \boxtimes S^{\pi - \pi ' } (  \ZZZ )  
\]
in $\KVar_{S^\pi X} $, where $\boxtimes$ is the exterior product morphism  
\[ 
{\KVar_{S^{\pi ' } Y} \times \KVar_{S^{\pi - \pi ' } U} \to \KVar_{S^{\pi ' } Y \times S^{\pi  - \pi ' } U} }
\]
composed with the morphism induced by $S^{\pi ' } Y \times S^{\pi  - \pi ' } U \to S^\pi X $ on the rings of varieties \cite[Corollary 3.5.2.5]{bilu2018motivic}.

If $\XXX$ and $\XXX ' $ are two families of quasiprojective varieties over $X$, it is clear from the definition that there is an isomorphism 
\begin{equation}\label{equation:iso-Spi(XXXtimesXXX)}
S^\pi ( \XXX \times_X \XXX  ' )  \overset{\sim}{\longrightarrow} S^\pi ( \XXX ) \times_{S^\pi X} S^\pi ( \XXX ' ) . 
\end{equation}
%%% SUBSECTION %%%

\subsection{Motivic Euler products}
In this paragraph, $R$ is a variety and $\mathrm A_R $ is one of the rings $\KVar_R$, $\KExpVar_R$, $ \mathscr M_R $ or $ \ExpM_R$ previously defined. 

Bilu introduced in the third chapter of her thesis \cite[\S 3.8]{bilu2018motivic} the following notation. 
\begin{mynotation}[Motivic Euler product]\label{def:euler-product}
Consider a $R$-variety $X$ and $\XXX = ( X_i )_{i\in I }  $ a family of elements of $\mathrm A_X$. Then the product
\[
\prod_{x\in X/R} \left ( 1 + \sum_{i\in I} X_{i,x}  T_i \right ) 
\]
is defined as a notation for the formal series
\[
\sum_{\pi \in \NN^{(I)}} [ S ^\pi ( \XXX / R ) ] \TT^\pi 
%= Z_{\XXX / R} ( \mathbf t  )  
\in \mathrm A_R [[\TT ]] 
\]
where $\TT = ( T_i )_{i\in I}$ is a family of indeterminates indexed by the set $I$ and $\TT^\pi $ stands for $ \prod_{i\in I} T_i^{n_i}$ for every partition $\pi \in \NN^{(I)}$. 
\end{mynotation}

The following proposition shows that this object behaves, in a sense, \textit{like a product}.
\begin{myptn}[\S  3.8.1 of \cite{bilu2018motivic}]
Let $R$ be a variety, $X$ be a $R$-variety and $\XXX = ( X_i )_{i\in I}$ be a family of elements of $ \mathrm A_X $.
\begin{itemize}
\item The Euler product notation is compatible with the $X=R$ case: 
\[
\prod_{x\in R/R} \left ( 1 + \sum_{i\in I} X_{i,x}  T_i \right ) = 1 + \sum_{i\in I} X_{i}  T_i .
\]
\item The Euler product notation is associative: if $X=U\cup Y $ with $Y$ a closed subscheme of $X$ and $U$ its complement, then 
\[
\prod_{x\in X/R} \left ( 1 + \sum_{i\in I} X_{i,x}  T_i \right )  = \left ( \prod_{u \in U/R} \left ( 1 + \sum_{i\in I} X_{i,u}  T_i \right ) \right )\left ( \prod_{y\in Y/R} \left ( 1 + \sum_{i\in I} X_{i,y}  T_i \right ) \right )
\]
when considering the Euler products of the restrictions $\mathscr Y =  (X_i \times_X Y )_{i\in I}$ and $\UUU =  {( X_i \times_X U )_{i\in I}}$. 
\item As a consequence of the previous identity, the Euler product notation is compatible with finite products: if $X$ is a finite disjoint union of varieties $Y_1$, ... , $Y_m$ isomorphic to $R$, and if $(Y_{i,j})_{i\in I}$ is the restriction of $\XXX $ to $Y_j$ for $j=1 , .... , m$, then 
\[
\prod_{x\in X/R} \left ( 1 + \sum_{i\in I} X_{i,x}  T_i \right ) = \prod_{j=1}^m \left ( 1 + \sum_{i\in I} Y_{i,j}  T_i \right ).
\]
 \end{itemize}
\end{myptn}
Another illustration of this behaviour is the following non-trivial property of multiplicativity, which is a particular case of  \cite[Proposition 3.9.2.4]{bilu2018motivic}. 
Note that in the remainder of the article, we will mostly restrict ourselves to the case $I= \NN^r \setminus \{ \mathbf 0 \}$ and $T_i = \TT^\mm$, where $T_1$, ..., $T_r$ is a finite family of indeterminates and $\TT^\mm = \prod_{j=1}^r T_j^{m_j}$ for every $r$-tuple $\mm = ( m_j)_{1\leqslant j \leqslant r }$ of non-negative integers. 

\begin{myptn}\label{proposition:euler-product:multiplicativity}
Let $R$ be a variety, $X$ a variety over $R$, 
$\AAA = ( A_{\mm})_{\mathbf{m}\in \NN^r } $ a family of effective elements of $\mathrm A_R$, 
and $\BBB = (B_{\mm})_{\mathbf{m}\in \NN^r }$ a family of elements of $\mathrm A_R$ such that $A_\mathbf 0=B_\mathbf 0=1$. Then 
\begin{align*}
& \prod_{x\in X / R } \left ( \left (\sum_{\mathbf{m}\in \NN^r } A_{\mm ,x} \mathbf T^{\mm}\right ) \left ( \sum_{\mathbf{m}\in \NN^r} B_{\mm,x} \mathbf T^{\mm} \right )\right  )  \\
& = \left ( \prod_{x\in X / R } \left ( \sum_{\mathbf{m}\in \NN^r } A_{\mm,x} \mathbf T^{\mm}\right ) \right ) 
\left ( \prod_{x\in X /R } \left ( \sum_{\mathbf{m}\in \NN^r } B_{\mm,x} \mathbf T^{\mm}\right ) \right )
\end{align*}
in $\mathrm A_R [[\mathbf T]]$.   
\end{myptn}

%%% SUBSECTION %%%

\subsection{Motivic Fourier transform and Poisson formula in families} 
In this subsection we recall the motivic analogue of a bunch of Fourier analysis tools used in the sixth chapter of  \cite{bilu2018motivic}, where an easy-to-handle expression of the motivic height zeta function is obtained by using a motivic Poisson formula. We will explicitly need this construction at the end of our proof (\S \ref{subsection:uniform-convergence}), where we will have to check an uniform convergence.
We follow \cite{chambert2016motivic} and the fifth chapter of \cite{bilu2018motivic}, which extends the scope of Hrushovski and Kazhdan's motivic Poisson formula  \cite{hrushovski2009motivic}.

\subsubsection{Local motivic Schwartz-Bruhat functions}
The first building block of a \textit{motivic} Poisson formula is a motivic analogue of classical Scwhart-Bruhat functions on the non-archimedian local field $K=k((t))$. One should keep in mind that this field can be thought about as the completion of the field of rational functions of a curve, for a valuation given by a closed point. The letter $t$ denotes an uniformiser of this completion; then the ring of integers of $K$ is $\mathcal O_K = k [[t]]$. 

Recall that a classical Schwartz-Bruhat function $\varphi $ on a locally compact, non-archimedian local field $L$ is a locally constant and compactly supported function on $L$. 
If $\varpi$ is an uniformiser for $L$ and $\mathcal O_L$ is the ring of integers of $L$, then there exist integers $M\leqslant N$ such that $\varphi $ is zero outside $\varpi^M \mathcal O_L$ and invariant modulo $\varpi^N \mathcal O_L$. The pair $(M,N)$ is called the level of $\varphi$ and the function $\varphi$ on $L$ can be seen as a function on the quotient $ \varpi^M \mathcal O_L / \varpi ^N \mathcal O_L $. 

The motivic analogue of such a function, as it has been introduced by Hrushovski and Kazhdan in \cite{hrushovski2009motivic}, takes values in the Grothendieck ring $\ExpM_k$. In order to deal with the fact that $k((t))$ is not locally compact for the topology induced by the valuation, one takes as a first definition the previous properties concerning invariance and compact support: a motivic Schwartz-Bruhat function of level $(M,N)$ is a function $t^M \mathcal O_K / t^N \mathcal O_K \to \ExpM_k$. 
Through the identification 
\[
\begin{array}{rll}
t^M \mathcal O_K / t^N \mathcal O_K  & \longrightarrow & \mathbf A_k^{M-N} \\
a_M t^M + a_{M+1}t^M + ... + a_{N-1} t^{N-1} & \longmapsto & ( a_M , ... , a_{N-1} ) 
\end{array}
\]
which endows $t^M \mathcal O_K / t^N \mathcal O_K$ with the structure of a $k$-variety,
such a motivic Schwartz-Bruhat function on $K$ is basically seen as an element of $\ExpM_{\mathbf A_k^{N-M}}$. In what follows we will freely use the notation 
\[
\mathbf A_k^{n ( M , N )} = \mathbf A_k^{n ( N-M ) } .
\]
for any $M \leqslant N $ and non-negative integer $n$. The previous identification naturally extends to $K^n$, leading to the following definition. 
\begin{mydef}Let $M\leqslant N $ and $n\geqslant 1$ be integers.
A (local) motivic Schwartz-Bruhat function of level $(M,N)$ on $K^n$ is an element of $\ExpM_{\mathbf A_k^{n(M,N)}}$. 
\end{mydef}
The \textit{extension-by-zero} morphism $ t^M \mathcal O_K / t^N \mathcal O_K  \to t^{M-1} \mathcal O_K / t^N \mathcal O_K  $ and the \textit{cutting} morphism $t^M \mathcal O_K / t^{N+1} \mathcal O_K  \to t^M \mathcal O_K / t^N \mathcal O_K  $ induce respectively a closed immersion $\mathbf A_k^{(M,N)} \to \mathbf A_k^{(M-1,N)} $ and a trivial fibration $ \mathbf A_k^{(M,N+1)} \to \mathbf A_k^{(M,N)} $ (with fibre $  \mathbf A_k^1$). Such morphisms in turn provide homomorphisms at the level of Grothendieck rings with exponentials, turning $(\ExpM_{\mathbf A^{n(M,N)}} )_{M\leqslant N}$ into a direct system, the direct limit of which is by definition the set of all motivic Schwartz-Bruhat functions on $K^n$. 

\paragraph*{Fourier kernel and transform.} Recall that we see $K=k((t))$ as the completion $F_v$ at a closed point $v$ of the field of fractions $F= k( C ) $ of a smooth projective curve over $k$, together with the choice of an uniformiser $t$. 
If one fixes a non-zero rational differential form $\omega \in \Omega_{K/k}$, then one obtains a non-zero $k$-linear map $r_v : K \to k$ defined by sending any element $a\in K$ to the residue at $v$ of the rational form $a\omega $:
\[
r_v ( a ) = \res_v ( a \omega ).
\]
There exists a smallest integer $\nu$ such that $r_v$ vanishes on $t^\nu \mathcal O_K$, given by the order of the pole of $\omega$ at $v$ (the \textit{conductor} of $r_v$). In particular, $r_v$ is invariant modulo $t ^N \mathcal O_K$ for all $N\geqslant \nu$ and can be seen as a linear function $r^{(M,N)} : \mathbf A_k^{(M,N)} \to \mathbf A_k^1$ for every $M\leqslant N$ such that $N\geqslant \nu$. 

The product $K\times K \to K$ restricts to a morphism $\mathbf A_k^{(M,N)} \times \mathbf A_k^{(M', N ' )} \to \mathbf A_k^{(M+M' , N'' )}$ for every $M\leqslant N $, $M' \leqslant N'  $ and $N'' = \min ( M' + N , M + N'  ) $. If $N,N',N'' \geqslant \nu$ it can be composed with $r_v^{(M+M',N'')}$ and one obtains an element   
\[
r_v : \mathbf A_k^{(M,N)} \times \mathbf A_k^{(M', N ' )} \overset{\cdot  \, \times \, \cdot }{\longrightarrow} \mathbf A_k^{(M+M' , N'' )} \overset{r_v^{(M+M',N'')}}{\longrightarrow} \mathbf A_k^1 
\]
of $\ExpM_{\mathbf A_k^{(M,N)} \times_k \mathbf A_k ^{(M',N')}} $, called the \textit{Fourier kernel on $K$} and written $\ee (xy)$. This notation is the analogue of the exponential factor $e^{2i\pi x y }$ of the integrand in classical Fourier analysis. 

More generally, the pairing $\langle x , y \rangle = \sum_{i=1}^n x_i y_i$ on $K^n$ provides a morphism 
\[
\mathbf A_k^{n(M,N)} \times \mathbf A_k^{n(M', N ' )} \overset{\langle \, \cdot  \, ,  \, \cdot \, \rangle }{\longrightarrow} \mathbf A_k^{(M+M' , N'' )} 
\]
which can be composed with $r_v^{(M+M' , N '' ) }$ to give the \textit{Fourier kernel $\ee( \langle x,  y \rangle ) $ on $K^n$}.

The Fourier transform of a motivic Schwart-Bruhat function $\varphi \in \ExpM_{\mathbf A_k^{n(M,N)}} $  of level $(M,N)$ on $K^n$ is defined as 
\begin{equation}\label{def:local-fourier-transform}
\mathscr F \varphi ( y ) = \int_{K^n} \varphi ( x)  \ee ( \langle x , y \rangle ) \dd x 
\end{equation}
which is a notation for the class of $\ExpM_{\mathbf A_k^{n(\nu - N , \nu - M)}}$ given by 
\[
\mathscr F \varphi  = \LL^{-Nn } \varphi  \times_{ \mathbf A_k^{n(M,N)}} [ \mathbf A_k^{n(M,N)} \times_k \mathbf A_k^{n ( \nu - N , \nu - M )} , r ]  
\]
where $r = r^{( M + \nu - N , N + \nu - M )} \circ \langle \, \cdot  \, ,  \, \cdot \, \rangle$.
The formal variable $y$ is thus\textit{ living in} $\mathbf A_k^{n ( \nu - N , \nu - M )}$. 
We can be more explicit when $\varphi = [ U , f ]$ with $g: U \to \mathbf A_k ^{n(M,N)}$ and $f:U \to \mathbf A^1$; in this case $\mathscr F \varphi$ is the class
\[
\mathscr F \varphi  = \LL^{-Nn} [ U  \times_k \mathbf A_k^{n ( \nu - N , \nu - M )}  , f \circ \pr_1 + r^{( M + \nu - N , N + \nu - M )} ( \langle g \circ \pr_1 , \pr_2 \rangle ) ]
\]
in $\ExpM_{\mathbf A_k^{n(\nu - N , \nu - M)}}$.
\begin{center}
\begin{tikzcd}
	U \times_k \mathbf A_k^{n ( \nu - N , \nu - M )}		\arrow[ddr,bend right]	\arrow[r,"\sim"]	& U \times_g \mathbf A_k^{n(M,N)} \times_k \mathbf A_k^{n ( \nu - N , \nu - M )}	\arrow[d]	\arrow[r]	& 	 U \arrow[r,"f"] \arrow[d,"g"]	 & \mathbf A ^1 \\
\mathbf A^1 &  \mathbf A_k^{n(M,N)} \times_k \mathbf A_k^{n ( \nu - N , \nu - M )} \arrow[r] \arrow[d] \arrow[l,"r",crossing over]  & 			\mathbf A_k^{n(M,N)} \arrow[d]	  & \\
 						& 				\mathbf A_k^{n ( \nu - N , \nu - M )}  \arrow[r]		   &  \Spec ( k )     &
\end{tikzcd}
\end{center} 
Notation \ref{def:local-fourier-transform} can be seen as a variant of the exponential sum notation of \S \ref{paragraph:ring-of-varieties}. 

\paragraph*{From local functions to global ones and summation over rational points.} 
Such a construction easily extends to finite products $\prod_{s\in S} F_s$ of completions of $F=k(C)$ at a finite number of closed points $s\in S$ of the curve $C$. 
The motivic Schwartz-Bruhat functions of level $(M_s , N_s )_{s\in S}$ are the elements of the ring 
\[
\ExpM_{\prod_{s\in S} \Res_{\kappa ( s)  / k} \mathbf A_{\kappa ( s ) } ^{n(M_s , N_s )} }
\]
where $\Res_{\kappa ( s)  / k} $ denotes the Weil restriction functor. The set of global motivic Schwartz-Bruhat functions on $K^n$ is defined as a direct limit over the set $S$ of closed points of $C$ and one can easily extend the previous Fourier kernel and transform to this setting; we will not be more explicit about this construction for now, since it will be generalised in the next paragraph. 

Nevertheless it might be enlightening for the reader to define the \textit{sum over rational points} of a global motivic Schwartz-Bruhat is this simple setting. Let $\varphi$ be a global Schwartz-Bruhat function of level $(M_s, N_s )_{s\in S}$ where $S\subset C$ is a finite set of closed points of $C$. 
We choose an uniformiser $t_s$ of $F_s$ for every $s\in S$. 
Consider the divisor $D = - \sum_{s\in S} M_s [s] $ on $C$ and remark that the embeddings $F\hookrightarrow F_s$ map the Riemann-Roch space 
\[
L ( D ) = \Gamma  ( C , \mathscr O_C ( D ) )  
\] 
(the space of global sections of the invertible sheaf $\mathscr O_C ( D )$) 
to $t_s^{M_s} \mathcal O_s $ for every $s\in S$. This mapping provides a morphism of varieties $\theta : \mathcal L ( D ) ^n  \to \Res_{\kappa ( s)  / k} \mathbf A_{\kappa ( s ) } ^{n(M_s, N_s)}$. 
The sum over rational points of $\mathbf A^n $ 
\[
\sum_{x\in k(C ) ^n} \varphi ( x )  = \sum_{x\in L ( D )  ^n } \theta^* \varphi
\]
is by definition the image in $\ExpM_k$ of the pull-back $\theta^* \varphi_S$, that is, the image of $\varphi  $ through the map
\[
\ExpM_{\prod_{s\in S} \Res_{\kappa ( s)  / k} \mathbf A_{\kappa ( s ) } ^{n(M_s , N_s )} } \overset{\theta ^* } {\longrightarrow} \ExpM_{\mathcal L ( D ) ^n } \longrightarrow \ExpM_k 
\]
which can be seen as another variant of the exponential sum notation of \S \ref{paragraph:ring-of-varieties}. 
In this context, one gets a motivic Poisson formula
\[
\sum_{x\in k ( C ) ^n }  \varphi (x ) = \LL^{(1-g)n} \sum_{y\in k(C) ^n} \mathscr F \varphi (y)  
\]
(see e.g. \cite[Theorem 1.3.10]{chambert2016motivic} for a proof). 

\subsubsection{Constructible families of Schwartz-Bruhat functions}
Until the end of this subsection we follow \cite[Chapter 5]{bilu2018motivic}.
We will say that a function $X\overset{f}{\to} \mathbf Z$ on a $k$-variety is \textit{constructible} if the inverse image of any integer  by $f$ is a constructible subset of $X$.

Again $C$ is a smooth projective connected curve over an algebraically closed field $k$ of characteristic zero. Let $M,N: C \to \mathbf Z $ be constructible functions such that $M\leqslant N$. Such functions are constant respectively equal to $M_0$ and $N_0$ on a dense open subset $U$ of the curve $C$. 
Then $\mathbf A_C ^{( M , N)}$ stands for the $C$-variety isomorphic to $U\times \mathbf A_k^{(M_0 , N_0)}$ over $U$ and with fibre over $u\notin U$ equal to $\mathbf A_k^{(M_u , N_u )}$. Let $n$ be a positive integer. 
One can as well define the $n$-th product over $C$
\[
\mathbf A_C^{n(M,N)} = \underbrace{\mathbf A_C ^{( M , N)} \times_C ... \times_C \mathbf A_C ^{( M , N)}}_{n\text{ times}} . 
\]
%We now apply these definitions to a family of functions $M$ and $N$.
Until the end of this section we fix two constructible functions 
$a , b: C \to \mathbf Z$ such that $a \leqslant 0 \leqslant b$ and we assume that there exists a dense open subset $U$ of $C$ such that $a_{|U} = b_{|U} = 0$. Let $\Sigma$ be the complement of $U$ in $C$, it is a finite set of points. 
Let 
$M=(M_\iota )$ and $N=(N_\iota ) $ be families of non-negatives integers indexed by $\NN^p$, such that $M_\mathbf 0 = N_\mathbf 0 = 0$. This data provides a family 
\[
\AAA_{\nn} ( a , b, M , N ) = S^{\nn}\left ( \left ( \mathbf A_C^{n ( a - M_\iota , b + N_\iota )  } \right )_{\iota \in \NN^p}  \right )
\]
for any $p$-tuple $\nn$ of non-negative integers, which is a variety over $S^\nn C$. 
Let us quickly describe the fibre of $  \AAA_{\nn} ( a , b, M , N ) \to S^\nn C $ over a schematic point $D\in S^{\nn}C$. The latter can be decomposed formally as
\[
D = \sum_{v\in C } \iota_v  v  =  \sum_{v\in U } \iota_v  v  +  \sum_{v \in \Sigma } \iota_v v = D_U + D_\Sigma 
\]
where the $\iota_v $ are $p$-tubles of non-negative integers (almost all equal to $\mathbf 0 \in \NN^p$), \[ 
D_U = \sum_{v\in U } \iota_v  v  \in S^\pi U
\] 
for some partition $\pi = ( n_\iota^U )_{\iota \in \NN^p} $ of some $\nn^U \leqslant \nn$ and ${D_\Sigma =\sum_{v \in \Sigma } \iota_v  v  }$. Then, the fibre over $D$ can be seen as the domain of definition of a Schwartz-Bruhat function $\Phi_D$, up to a finite extension of the residue field $\kappa ( D ) $ of $D$ (seen as a point of $S^\mathbf n C$) \cite[\S 5.3.2]{bilu2018motivic}: it is of the form 
\begin{equation} \label{equation:fibre-of-An-SnC}
\prod_{\iota \in \NN^p} \mathbf A_{\kappa ( D ) }^{ n_\iota^U \times n   ( M_\iota +  N_\iota )  }  \times_{\kappa ( D ) } \prod_{\substack{ v\in \Sigma \\ \mathbf i_v \neq \mathbf 0 }} \mathbf A_{\kappa ( D )} ^{n ( a_ v - M_{\mathbf i_v}, b_v  + N_{\mathbf i_v} )  } .
\end{equation}
Note that the first product is actually finite, since $( n_\iota^U )_{\iota \in \NN^p} $ has finite support as a partition. 

For this reason, elements of $\ExpM_{\AAA_{\nn} ( a , b, M , N ) }$ are called \textit{constructible families of Schwartz-Bruhat functions of level} $\nn$ \cite[\S 5.3.4.1]{bilu2018motivic}. 
Two special cases arise. On one hand, if all the integers $N_\iota $ are zero, the family is said to be \textit{uniformly smooth}. On the other hand, if all the integers $M_\iota $ are zero, the family is said to be \textit{uniformly compactly supported} \cite[\S 5.3.5.1]{bilu2018motivic}.

\subsubsection{Motivic Fourier transform in families}
\label{subsubsection:motivic-fourier-transform-in-families}
It is possible to define the Fourier transform of a family $\Phi \in \ExpM_{\AAA_{\nn} ( a , b, M , N )} $ of Schwartz-Bruhat functions, independently of the choice of $b $ and $N$. Let us fix a non-zero rational differential form $\omega \in \Omega_{k(C)/k}$ and define $\nu$ by $\nu_v = - \ord_v \omega $ for every closed point $v\in C$. 
This allows one to define a constructible Fourier kernel morphism 
\[
\mathbf A_C^{n(a - M_\iota , b + N_\iota )} \times_C \mathbf A_C^{n(\nu - b - N_\iota , \nu- a + M_\iota )} \to \mathbf A^1
\]
which by (\ref{equation:iso-Spi(XXXtimesXXX)}) induces morphisms on the symmetric products 
\[
r_{\nn}  : \AAA_{\nn} ( a , b, M , N ) \times_{S^{\nn}C} \AAA_{\nn} ( v - b , v - a , N , M ) \to \mathbf A^1 
\]
hence an element of $\ExpM_{\AAA_{\nn} ( a , b, M , N ) \times_{S^{\nn}C} \AAA_{\nn} ( v - b , v - a , N , M )}$ for any $\nn \in \NN^p$
\cite[\S 5.4.2]{bilu2018motivic}. 
Then one can show that the class $[S^{\nn}(( \mathbf A_C^{n(b , N_\iota)} ) _{\iota \in \NN^p} )]$ has an inverse in $\ExpM_{S^\nn C}$. We finally define a motivic Fourier transform 
\[
\mathscr F :   \ExpM_{\AAA_{\nn} ( a , b, M , N )} \longrightarrow  \ExpM_{\AAA_{\nn} ( \nu - b , \nu - a  , N , M )} 
\]
given for any $\Phi  = [V,f]\in \ExpM_{\AAA_{\nn} ( a , b, M , N )}$, with $g: V \to \AAA_{\nn} ( a , b , M , N )$ the structure morphism, by the class in  $\ExpM_{\AAA_{\nn} ( \nu - b , \nu - a, N , M )}$ 
\[
[V \times_{S^{\nn}C} \AAA_{\nn} ( \nu - b , \nu - a , N, M ) , f \circ \pr_1 + r_{\nn} ( g \circ \pr_1 \cdot \pr_2)] [S^{\nn}(( \mathbf A_C^{n(b , N_\iota)} ) _{\iota \in \NN^p} )]^{-1}
\]
where $\cdot $ denotes the product in the ring $\ExpM_{\AAA_{\nn} ( a , b, M , N ) \times_{S^{\nn}C} \AAA_{\nn} ( \nu - b, \nu - a, N, M )}$ \cite[\S 5.4.3]{bilu2018motivic}. 
This Fourier transform is independent of the choice of $ b  $ and $N$, and is compatible with symmetric products \cite[Proposition 5.4.4.2]{bilu2018motivic}.
It generalises the local and global Fourier transforms of a Schwartz-Bruhat function we presented in the first paragraph of this subsection. 

\subsubsection{Local Poisson formula}
% Recall that $U$ is contained in the zero locus of the constructible functions $a$ and $b$, $\Sigma $ is its complement in $C $ and $\nn$ a $p$-tuple of non-negative integers. 
% A schematic point $D$ of the $\nn$-th symmetric product  $S^{\nn}C$ is of the form $(D_U, D_\Sigma ) $ where $D_U \in S^\pi U$ for some partition $\pi ( n_\iota^U ) $ of some $\nn^U \leqslant \nn$ and $D_\Sigma $ a divisor with support in $\Sigma$ of the form  $\sum_{v\in \Sigma} i_v v $. 
% Then, as we already said, the fibre of $\AAA_{\nn} ( a , b, M , N ) \to S^{\nn}C $ above $D$ can be seen as the domain of definition of a Schwartz-Bruhat function $\Phi_D$, up to a finite extension of the residue field $\kappa ( D ) $ of $D$.  

Let $D_U'$ be a pull-back of $D_U \in S^\pi U$  via the quotient map
\[
\left (\prod_{\iota \in \NN^p \setminus \{ \mathbf 0 \}} U^{m_\iota} \right )_* \longrightarrow S^\pi U 
\]
defining $S^\pi U$, and $v_{\iota , j} $ the projection of $D_U '$ on the $j$-th copy of $U$ in $U^{m_\iota}$. A different choice of $D_U'$ only permutes the order of the $v_{\iota , j}  $'s. 
Let $E_D$ be the effective $\kappa ( D )$-zero-cycle on the curve $C_{\kappa ( D)}$ defined by 
\[
E_D = \sum_{\iota \in \NN^p \setminus \{0\} } M_\iota ( v_{\iota , 1 } + ... + v_{\iota , m_\iota}) - \sum_{v\in \Sigma } ( a_v - M_{\iota_v} ) v .
\]
The divisor $E_D$ does not depend on the choice of $D_U'$ since it is invariant under the action of $\prod_{\iota \in \mathbf N^p \setminus \{ 0 \}} \mathfrak{S}_{m_\iota} $ on the product $\prod_{\iota \in \mathbf N^p \setminus \{ 0 \}}  U^{m_\iota} $  \cite[Remark 5.5.1.1]{bilu2018motivic}. 
It can be rewritten as 
\[
\sum_{v\in C} \left ( M_{\iota_v} - a_v \right ) v 
\]
since $a$ is zero on $U$. 
One then defines a constructible morphism of $\kappa ( D ) $-varieties 
\[
\theta_D : L_{\kappa ( D )} ( E_D )^n  \longrightarrow \AAA_{\nn} ( a , b, M , N )_D 
\]
where $ L_{\kappa ( D )} ( E_D )$ is the Riemann-Roch space over $E_D$
 \[
  L_{\kappa ( D )} ( E_D ) % = \{   f \in \kappa ( D ) ( C ) \mid f = 0 \text{ or } \ddiv (f) +  E_D \geqslant 0 \}
  = \Gamma (  C_{\kappa ( D ) } , \mathscr O_{C_{\kappa ( D ) }} ( E_D ) ) .
 \]
Pointwise, the morphism $\theta_D $ sends a function $f \in   L_{\kappa ( D )} ( E_D ) $ to its $v$-adic expansions in the ranges given by the exponents defining $\AAA_{\nn} ( a , b, M , N )_D$ in \ref{equation:fibre-of-An-SnC}.  For example, its image in 
\[
% \prod_{v\in \Sigma} \mathbf A_{\kappa ( D ) }^{(\alpha_v - M_{\iota_v} , \beta_v + N_{\iota_v})} = 
\prod_{v\in \Sigma}  \left ( \mathfrak{ m }_v^{a_v - M_{\iota_v}} \mathcal O_v /  \mathfrak{ m }_v^{b_v + N_{\iota_v}} \mathcal O_v \right )
\]
may be understood as the coefficients of its $v$-adic expansion in the range $a_v - M_{\iota _v} , ... , b_v + N_v - 1 $ (see \cite[page 164]{bilu2018motivic} for an explicit definition).

The \textit{summation over rational points} of $\Phi_D$ is by definition the class in $\ExpM_{\kappa ( D)}$ of
\begin{equation}\label{def:summation-rational-points}
\sum_{x\in \kappa ( D ) ( C ) ^n } \Phi_D ( x ) = \sum_{x\in L_{\kappa ( D ) } ( E_D) ^n}  ( \theta_D^* \Phi_D ) (x) . 
\end{equation} 
Here the sum of the right side is the exponential sum notation (\ref{def-equation:exponential-sum-notation}) corresponding to the morphism
\[
\ExpM_{L_{\kappa ( D ) } (E_D)^n} \longrightarrow \ExpM_{\kappa ( D) } 
\]
induced by the projection on $\kappa ( D) $. 

For a proof of the following, see \cite[Theorem 1.3.10]{chambert2016motivic} and \cite[Lemma 5.5.1.4]{bilu2018motivic}. 
\begin{myptn}[Poisson formula] Let $D $ be a schematic point of $S^{\nn}C$ and $\Phi_D $ an element of the fibre $  \ExpM_{\AAA_{\nn} ( a , b, M , N )_D }$. Then 
\[
\sum_{x\in \kappa ( D)  ( C ) ^n } \Phi_D ( x )  = \LL^{(1-g)n} \sum_{x\in \kappa ( D ) ( C ) ^n } \mathscr F \Phi_D ( x ) .
\]
\end{myptn}

\subsubsection{Poisson formula in families}
\label{paragraph:poisson-families} 
If the family $\Phi$ is uniformly compactly supported, the zero cycle $E_D$ is actually equal to the $k$-zero-cycle $D_a = - \sum_v a_v v $ for any schematic point $D \in S^\nn C $. Its Riemann-Roch space over $k$ is by definition 
\[
L(D_a ) = \{ f \in k ( C ) \mid f = 0 \text{ or } \ddiv (f) +  D_a \geqslant 0 \}.
\]
By flat base change, there is a $\kappa ( D ) $-linear canonical isomorphism $L_{\kappa ( D ) } ( D_a ) \simeq L(D_a ) \otimes_k \kappa ( D )$. 
Then one can prove the existence of a constructible morphism
\[
\theta_{\nn} : L ( D_a ) \times S^{\nn}C \longrightarrow \AAA_{\nn} ( a , b, 0 , N ) 
\]
over $S^{\nn}C $, whose restriction to the fibres above a schematic point $D \in S^{\nn} C$ induces $\theta_D$ \cite[\S 5.5.2.1]{bilu2018motivic}.
Given an uniformly compactly supported family $\Phi \in \ExpM_{\AAA_{\nn} ( a , b, 0 , N ) } $, 
the image in $\ExpM_{S^{\nn}C} $ of the pullback $\theta_{\nn} ^* \Phi$
is called the \textit{uniform summation over rational points}. 
It is denoted by 
\[
\left ( \sum_{x\in \kappa (D) ( C ) ^n } \Phi_D ( x ) \right )_{D\in S^{\nn}C} .
\]
This is an example of what Bilu calls an \textit{uniformly summable family}, that is, a constructible family of functions $\Phi \in\AAA_{\nn} ( a , b, M , N )  $ such that there exists a function $\Sigma \in \ExpM_{S^{\nn}C} $ on $S^{\nn}C$ whose pullback $D^* \Sigma$ in $\ExpM_{\kappa ( D )}$ is the sum over rational points
$\sum_{x\in \kappa (D) ( C ) ^n} \Phi_D ( x ) $, for any $D\in S^{\nn}( C ) $. 

Now one may remark that there is a commutative diagram of group morphisms
\begin{center}
\begin{tikzcd}
\ExpM_{L(D_\alpha)^n \times S^{\nn} C}  \arrow[r] \arrow[d] & \ExpM_{S^{\nn} C} \arrow[d]\\
\ExpM_{L(D_\alpha)^n }  \arrow[r] & \ExpM_k
\end{tikzcd}
\end{center}
which means that it is possible to permute the sums as follows: 
\[
\sum_{D\in S^{\nn} C } \sum_{x\in \kappa ( D ) ( C ) ^n } \Phi_D ( x )  = \sum_{x \in k(C)^n } \sum_{D\in S^{\nn} C } \Phi_D ( x ) .
\]
\begin{myptn}[Poisson formula in families]
 If $\Phi $ is an uniformly smooth family of Schwartz-Bruhat functions, then its Fourier transform $\mathscr F \Phi $ is  uniformly compactly supported and one has  
\[
\sum_{D\in S^{\nn}C } \sum_{x\in \kappa ( D ) ( C ) ^n } \Phi_D ( x ) = \LL^{(1-g)n} \sum_{y\in k ( C ) ^n} \sum_{D\in S^{\nn}C } \mathscr F \Phi_D ( y )  .
\]
in $\ExpM_k$.
\end{myptn}

\subsection{Arc schemes and integrals of motivic residual functions}
In this paragraph we collect a few facts about jet schemes, arc schemes and a particular case of motivic measure on such spaces. A concise reference for this paragraph could be \cite{blickle2005short} or \cite{craw2004introduction}. A more exhaustive one is the third chapter of \cite{chambert2018motivic}. 

\subsubsection{Jet schemes, arc schemes and integration} The field $k$ is an algebraically closed field of characteristic zero and $\mathscr Y$ is a flat scheme of finite type over $k[[t]]$, which we assume to be equidimensional of relative dimension $n$. A simple example is given by $\mathscr Y = Y \times_{\Spec k } \Spec ( k [[t]] ) $ for some $k$-variety $Y$. 
For any non-negative integer $m$, the jet-scheme $\mathscr L_m ( \mathscr Y ) $ of order $m$ of $\mathscr Y$ is the $k$-variety representing the functor 
\[
A \mapsto \Hom_k ( \Spec ( A[[t]] / (t^{m+1} )) , \mathscr Y ) 
\]
on the category of $k$-algebras.
There are canonical affine projection morphisms 
\[ 
p^{m+1}_m : \mathscr L_{m+1} ( \mathscr Y ) \rightarrow \mathscr L_m ( \mathscr Y )
\] and one can consider the proscheme over $k$ 
\[
\mathscr L ( \mathscr Y ) = \underset{\leftarrow}{\lim}_m \mathscr L_m ( \mathscr Y ) 
\] of arcs on $\mathscr Y $. This projective limit carries canonical projections $p_m : \mathscr L ( \mathscr Y ) \to \mathscr L_m ( \mathscr Y ) $.

The projections $p^{m+1}_m $ induce ring morphisms 
\[
\left (  p^{m+1}_m \right )^* : \ExpM_{\mathscr L_m ( \mathscr Y ) } \to \ExpM_{\mathscr L_{m+1} ( \mathscr Y ) }
\]between the corresponding localised Grothendieck rings with exponential. The ring of motivic residual functions on $\mathscr L ( \mathscr Y ) $ is defined as the direct limit $ \underset{\rightarrow}{\lim}_m \ExpM_{\mathscr L_m ( \mathscr Y )}  $. 
\begin{mydef}
Let $h$ be a motivic residual function, which we assume to be of the form $[H,f]_{\mathscr L_m ( \mathscr Y ) }$ with $H$ a variety over $\mathscr L_m ( \mathscr Y)$ for some $m$ and $f : H \to \mathbf A^1 $ a morphism. 
The integral of $h$ over $\mathscr L ( \mathscr Y ) $ is the element of $\ExpM_k$ 
\[
\int_{\mathscr L ( \mathscr Y ) } h ( x )  \dd x  = \LL^{-(m+1)n} [H , f]_k .
\]
\end{mydef}
Since the projection $\mathscr L_{m'} ( \mathscr Y ) \to \mathscr L_m ( \mathscr Y ) $ is a locally trivial fibration with fibres isomorphic to $\mathbf A^{(m'-m)n}$ for any $m' \geqslant m$ \cite[Proposition 3.7.5]{chambert2018motivic}, this definition does not depend on the choice of $m$. 
\subsubsection{Motivic volumes}\label{subsubsection:motivic-volume} An useful particular case of such integrals is given by the characteristic function of a constructible subset $W=p_m^{-1} ( W_m ) $ of $\LLL ( \YYY ) $. The integral of $\mathbf 1_W$ over $\LLL ( \YYY ) $ is by definition the \textit{volume}  $\vol ( W ) $ of $W$. There are three particular volumes of interest for our purpose. The first one is the volume of the whole arc space $\LLL ( \YYY ) = p_0^{-1} ( \YYY_k )$, which is 
\[
\vol ( \LLL ( \YYY )) = \LL^{-n} [ \YYY_k , 0] .
\]
Then the volume of the subspace $W = p_0^{-1} ( \{ 0 \} ) = \LLL ( \mathbf A^1 , 0 ) $ of arcs in $\mathbf A^1$ with origin at $0$ is 
\[
\vol ( \LLL ( \mathbf A^1 , 0 ) )= \LL^{-1} [ \{ 0 \} , 0] = \LL^{-1} .
\]
Finally, an arc  on $\mathbf A^1$ of order zero at $\{ 0 \}$ in an arc whose image in $\mathbf A^1$ do not belongs to $\{ 0 \}$, thus the corresponding volume is $1-\LL^{-1}$. 
More generaly, the set of arcs of order $m\in \NN$ at $\{ 0 \}$ has volume $(1-\LL^{-1})\LL^{-m}$. 

\subsubsection{Weights and volumes}\label{inequ-weight-integral-volume} The weight function of \cite{bilu2018motivic} we will introduce in the next section satisfies the following property 
 \cite[Remark 6.3.1.3]{bilu2018motivic}: for any constructible subset $W$ of $\LLL  ( \YYY ) $ and motivic residual function $h$ we have the inequality
\[
w \left ( \int_{\LLL ( \YYY )} \mathbf 1_W ( x )   h ( x)  d x \right ) \leqslant w ( \vol ( W ) ) . 
\]

%%%%%%%%%%%%%%%
%%% SECTION %%%
%%%%%%%%%%%%%%%

\section{Convergence of Motivic Euler products}
The aim of this section is to recall the main properties of the weight function on the Grothendieck ring $\ExpM_X$ introduced by Bilu \cite[Chapter 4]{bilu2018motivic} and to give an effective criterion, Proposition \ref{proposition:linear-conv-criterion}, 
ensuring the weight-linear convergence, in the corresponding completed ring, of multivariate motivic Euler products. We also include Lemma \ref{lemma:weight-convergence-product-stability} and Lemma \ref{lemma:negligible-terms}, two lemmas
dealing respectively with multiplicativity of weight-linear convergence and what one might call negligible convolution products. 

%%% SUBSECTION %%%
\subsection{Weight filtration}
We refer to Lemmas 4.5.1.3, 4.6.2.1, 4.6.3.1 and 4.6.3.4 of \cite{bilu2018motivic}  for a proof of the following. 
  
\begin{myptn}\label{proposition:properties-of-the-weight}
 Let $X$ be a complex variety. There exists a weight function 
\[
w_X : \ExpM_X \to \mathbf Z \cup \{ - \infty \} 
\] satisfying the following properties.
\begin{enumerate}
\item $w_X ( 0 ) = - \infty  $
\item $w_X ( \mathfrak a + \mathfrak a '  ) \leqslant \max ( w_X ( \mathfrak a ), w_X (   \mathfrak a ' ) )$ with equality if  $w_X ( \mathfrak a ) \neq w_X (   \mathfrak a ' ) $, for any $\mathfrak{a}, \mathfrak{a}' \in \ExpM_X$. 
\item If $\AAA = ( \mathfrak a_i ) _{i\in I} $ is a family of $\ExpM_X$ indexed by a set $I$ and $\pi = ( n_i )_{i\in I } \in \NN^{(I)}$ then
\[
w_{S^\pi X} ( S^\pi \AAA ) \leqslant \sum_{i\in I} n_i w_X ( \mathfrak a _i ) .
\]
\item If $Y \to X$ is a variety over $X$ then 
\[
w_X  ( Y ) = 2 \dim_X ( Y ) + \dim ( X ) . 
\] 
\item If $p : Y \to X$ and $q : Z \to X $ are smooth morphisms with fibres of dimension $d\geqslant 0$ and $Y,Z$ irreducible, then
\[
w_X (  [Y \overset{p}{\to} X ] - [Z \overset{q}{\to} X ]) \leqslant 2 d + \dim X - 1 . 
\]
\end{enumerate}
\end{myptn}
If $X=\Spec ( \mathbf C )$, the weight function $w_X$ will be simply written $w$.  
It induces a filtration $(W_{\leqslant n} \ExpM _X)_{n\in \mathbf Z}$ on $\ExpM_X$ given by 
\[
W_{\leqslant n} \ExpM _X =  \{ \mathfrak a \in \ExpM_X \mid w_X ( \mathfrak a ) \leqslant n \}
\]
for all $n\in \mathbf Z$.  
\begin{mydef}
The completion of $\ExpM _X$ with respect to the weight topology is
\[
\widehat{\ExpM _X} = \underset{\underset{n}{\leftarrow}}{\lim} \left ( \ExpM _X / W_{\leqslant n} \ExpM _X \right ). 
\]
\end{mydef}

%%% SUBSECTION %%%
\subsection{Convergence criteria} 

\begin{mydef}
Let $F(T)=\sum_{i\geqslant 0} X_i T^i $ be a formal power series with coefficients in $\ExpM _X $. The radius of convergence of $F$ is defined by 
\[
\sigma_F = \limsup_{i\geqslant 1} \frac{w_X (X_i)}{2i}.
\]	
We say that $F$ \textit{converges for} $|T|<\LL^{-r}$ if $r\geqslant \sigma_F$.
\end{mydef}

This terminology can be explained by the fact that if $F$ converges for $|T|<\LL^{-r}$
then for all $\mathfrak a \in \ExpM _\mathbf C$ with weight 
$w(\mathfrak a ) < -2r$
the evaluation $F(\mathfrak a)$ exists as an element of $\widehat{\ExpM}_X$. Indeed, by \cite[Lemma 4.5.1.3]{bilu2018motivic} one has
\[ 
w_X ( X_i \boxtimes \mathfrak a^i  ) \leqslant i (2 r + w ( \mathfrak{a} ) ) 
\]
for all $i$ sufficiently large, so $F(\mathfrak a ) = \sum_{i\geqslant 0} X_i \boxtimes \mathfrak{ a } ^i $ can be thought as an element of $\widehat{\ExpM}_X$ by considering partial sums with respect to the weight (here $\boxtimes$ denotes the exterior product defined in \S \ref{paragraph:ring-of-varieties}).

In what follows we will say that a formal series $\sum_{i\geqslant 0} \mathfrak{c}_i T^i $ with terms in $\ExpM_X$  converges \textit{weight-linearly} at $\mathfrak a \in \ExpM_\mathbf C$ if there exists $\delta > 0$ such that $w_X ( \mathfrak{c}_i \mathfrak a^i ) < -  \delta i $ for $i$ sufficiently large. 
We endow $\NN^r$ with its poset structure 
\[
\mm ' \leqslant \mm  \text{ if and only if }   m_\alpha ' \leqslant m_\alpha  \text{ for all } \alpha
\]
and 
extend this definition to families indexed by $\NN^r$ by considering indices such that
\[
\langle   \rho , \mm \rangle = \sum_{\alpha =1 }^r \rho_\alpha m_\alpha   = i  
\]
for a fixed $\rho \in \mathbf N^r \setminus \{ \mathbf 0 \}$. More precisely, a formal series $\sum_{\mm \in \NN^r} \mathfrak c _{\mm} \TT^\mm $ is said to converge \textit{weight-linearly with respect to $\rho$}, which we may sometimes abbreviate by saying \textit{$\rho$-weight-linearly},  at $\mathfrak a = ( \mathfrak a_1 , ... , \mathfrak a_r ) \in \left ( \ExpM_\CC \right ) ^r $ if there exists a real number $\delta > 0$  such that for $i$ sufficiently large one has
\[
w_X 
 (  \mathfrak c_{\mm} \mathfrak a ^\mm ) < - \delta i . 
\]
for every $\mm \in \NN^r $ such that $\langle \rho, \mm \rangle = i$. 
This property only depends on the support 
\[
\Supp ( \rho ) = \{ 1 \leqslant \alpha \leqslant r \mid \rho_\alpha > 0 \}
\]
of $\rho$. Indeed, if $\rho_1 , \rho_2 \in \NN^r$ have same support, then 
\[
\min_{\alpha \in \Supp ( \rho_1 ) } ( \rho_{1,\alpha} / \rho_{2,\alpha} ) \langle \rho_1 , \mm \rangle \leqslant \langle \rho_2 , \mm \rangle  \leqslant \max_{\alpha \in \Supp ( \rho_1 ) } ( \rho_{1\alpha} / \rho_{2,\alpha} ) \langle \rho_1 , \mm \rangle
\] 
for all $\mm \in \NN^r $.

In particular,
for such a $\rho$-weight-linearly convergent series we can licitly consider the sum
\[
\sum_{\mm \in \NN^r } \mathfrak c_{\mm} \mathfrak a ^\mm
\]
in $\widehat{\ExpM_X}$
as soon as 
$\sum_{\mm \in \NN^r } \mathfrak c_{\mm} \TT^\mm$ belongs to the subring of formal series with indeterminates $T_\alpha $ for $\alpha \in \Supp ( \rho )$.  

Then the value of the sum does not depend on the order of summation, 
in particular the partial sum \[
\sum_{\mm ' \leqslant \mm } \mathfrak c _{\mm'  } \mathfrak a ^{\mm ' }
\]
has a limit in $\widehat{\ExpM_X} $ 
when $\langle \rho , \mm \rangle $ tends to infinity. 

In the fourth section of this paper, we will have to deal with weight-linearly convergent series with respect to multiple $\rho$'s having different supports, each of these series depending only on the indeterminates $T_\alpha $ for $\alpha \in \Supp ( \rho )$. 
In this situation, simultaneous convergence for the various $\rho$'s will be obtained by asking $\min_{1\leqslant \alpha \leqslant r} ( m_\alpha )$ to be arbitrarily large.  

\begin{mylemma}\label{lemma:weight-convergence-product-stability}Fix $\rho \in \mathbf Z_{\geqslant 1}^r$ and
let $P(\TT ) = \sum_{\mm \in \NN^r} \mathfrak p_{\mm} \TT^\mm $ and $Q( \TT ) = \sum_{\mm \in \NN^r} \mathfrak q_{\mm} \TT^\mm$ be two formal series  
with coefficients in $\ExpM_X $. 
Assume that they converge $\rho$-weight-linearly at $\TT_\alpha = \LL^{-\rho_\alpha}$, which means that
 $\sum_{\mm \in \NN^r} \mathfrak p _{\mm} \LL^{-\langle \rho , \mm \rangle }$ and $\sum_{\mm \in \NN^r} \mathfrak q_{\mm} \LL^{-\langle \rho , \mm \rangle }$ are $\rho$-weight-linearly convergent. 
Then the product $R ( \TT ) = P ( \TT ) Q (\TT ) $ converges $\rho$-weight-linearly at $\TT_\alpha = \LL^{-\rho_\alpha}$.
\begin{proof}
The  $\mm$-th coefficient of $R(\TT)$ is given by 
\[
\mathfrak r_\mathbf m = \sum_{\mm' \leqslant \mm } \mathfrak p_{\mm ' } \mathfrak q_{\mm - \mm ' } .
\]
Let $\delta , \delta ' > 0$ and $i_0, i_0' \in \NN$ be such that 
\begin{align*}
w_X \left ( \mathfrak p _{\mm} \LL^{-\langle \rho , \mm \rangle } \right ) &< -\delta \langle \rho , \mm \rangle \text{ for any $\mm$ such that $\langle \rho , \mm \rangle \geqslant i_0$; } \\
w_X \left ( \mathfrak q _{\mm} \LL^{-\langle \rho , \mm \rangle } \right ) &< -\delta ' \langle \rho , \mm \rangle \text{ for any $\mm$ such that $\langle \rho , \mm \rangle \geqslant i_0 '$.}
\end{align*} 
Then by Proposition \ref{proposition:properties-of-the-weight}
\[
w_X ( \mathfrak r _{\mm} \LL^{-\langle \rho , \mm \rangle } )  \leqslant \max_{\mm ' \leqslant \mm} \left (  w_X \left ( \mathfrak p _{\mm ' } \LL^{-\langle \rho , \mm ' \rangle } \right )   +  w_X \left ( \mathfrak q _{\mm - \mm ' } \LL^{-\langle \rho , \mm - \mm '  \rangle } \right )  \right )  .
\]
There is only a finite number of indices $\mm ' $ such that $\langle \rho , \mm ' \rangle < i_0$. Fix such an $\mm' $. Then for $\mm$ such that $\langle \rho , \mm \rangle$ is sufficiently large 
\[
w_X \left ( \mathfrak p _{\mm ' } \LL^{-\langle \rho , \mm ' \rangle } \right )   +  w_X \left ( \mathfrak q _{\mm - \mm ' } \LL^{-\langle \rho , \mm - \mm '  \rangle } \right ) < - \frac{\delta ' }{2} \langle \rho , \mm   \rangle .
\]
The case when $\langle \rho , \mm - \mm ' \rangle < i_0 '$ is symmetric, and the last case $\langle \rho , \mm ' \rangle \geqslant i_0$ and $\langle \rho , \mm - \mm ' \rangle \geqslant i_0 '$ is left to the reader. This shows the existence of $j_0\in \NN$ such that
\[
w_X ( \mathfrak r _{\mm} \LL^{-\langle \rho , \mm \rangle } ) <-  \frac{\min ( \delta , \delta ' )}{2} \langle \rho  , \mm \rangle 
\]
for any $\mm$ such that $\langle \mm , \rho \rangle \geqslant j_0$. 
Thus $R ( \TT ) $ converges $\rho$-weight-linearly at $\TT_\alpha = \LL^{-\rho_\alpha}$.
\end{proof}
\end{mylemma}

\begin{mylemma}\label{lemma:negligible-terms}
Let $\sum_{\mm \in \NN^r } \mathfrak c_{\mm} \TT^\mm$ be a 
%$\rho$-absolutely 
convergent sum at $\mathfrak a $ with coefficients in $\ExpM_X $ and $\varepsilon \in \{0,1\}^r \setminus \{ \mathbf 0 \}$.
For any $\mm = ( m_i)_{i=1}^r \in \NN^r $ we write $\langle \varepsilon , \mm \rangle = \sum_{i=1}^r \varepsilon_i m_i $. 
Then 
\[
\sum_{\mm ' \leqslant \mm } \mathfrak c_{\mm'} \mathfrak a^{\mm ' } \LL^{-\langle \varepsilon , \mm - \mm' \rangle } 
\]
tends to zero in $\widehat{\ExpM_X} $ when $\min_{1\leqslant i \leqslant r} ( m_i ) \to \infty $. 

If moreover $\sum_{\mm \in \NN^r } \mathfrak c_{\mm} \TT^\mm $ is weight linearly-convergent at $\mathfrak a$ with respect to some given ${\rho \in \mathbf Z_{\geqslant 1}^r} $, then there exist $\delta > 0$ and $i_0 \in \NN$ such that 
\[
w_X \left ( \sum_{\mm ' \leqslant \mm } \mathfrak c_{\mm'} \mathfrak a ^{\mm  '} \LL^{-\langle \varepsilon , \mm - \mm' \rangle }  \right ) < - \delta \langle \varepsilon , \mm \rangle
\]
whenever $\langle \varepsilon , \mm \rangle \geqslant i_0$.
\begin{proof}In the proof we write $\mathfrak c_{\mm}$ for $\mathfrak c_{\mm} \mathfrak a ^{\mm  }$.
Fix an $A \in \mathbf Z$. There exists $\mm_0 \in \NN^r$ such that 
\[
w_X ( \mathfrak c_{\mm} ) < A  \text{ for all $\mm$ such that $\mm \nleqslant \mm_0$.}
\]
Set $\omega = \max ( w_X ( \mathfrak c_{\mm} \mid \mm \leqslant \mm_0 )) $. On one hand if $\mm ' \leqslant \mm_0 $ one has
\[
w_X \left ( \mathfrak c_{\mm '} \LL^{-\langle \varepsilon , \mm - \mm ' \rangle} \right )  \leqslant  \omega - \langle \varepsilon , \mm - \mm ' \rangle \leqslant \omega - \langle \varepsilon , \mm - \mm_0 \rangle < A
\]
if $\min ( m_i ) $ is large enough. 
On the other hand $\mm ' \nleqslant \mm_0$ implies $w_X \left ( \mathfrak c_{\mm '} \LL^{-\langle \varepsilon , \mm - \mm ' \rangle} \right ) < A$. Thus applying Proposition \ref{proposition:properties-of-the-weight} one gets 
\[
w_X \left (  \sum_{\mm ' \leqslant \mm } \mathfrak c_{\mm '} \LL^{-\langle \varepsilon , \mm - \mm ' \rangle}   \right ) < A
\]
for any $\mm $ sufficiently large. This proves the lemma.

Now we assume that 
there exist $i_0 \in \NN$ and $\delta >0 $ such that 
\[
w_X ( \mathfrak c_{\mm} ) < - \delta   \langle \mm , \rho \rangle 
\]
whenever $\langle \mm , \rho \rangle \geqslant i_0$. 
Remark that since we assumed that $\rho \in \ZZ_{\geqslant 1}^r$ and $\varepsilon \in \{ 0 , 1 \}^r$, we have 
\[
\langle \rho , \mm \rangle \geqslant   \langle \varepsilon , \mm \rangle 
\]
for all $\mm \in \NN^r$. 
On one hand, if $\langle \mm ' , \rho \rangle < i_0$, then 
\[
w_X \left ( \mathfrak c_{\mm '} \LL^{-\langle \varepsilon , \mm - \mm ' \rangle} \right )   \leqslant \omega '  - \langle \varepsilon , \mm  \rangle + i_0 < - \frac{1}{2} \langle \varepsilon , \mm  \rangle 
\]
whenever $\langle \varepsilon , \mm \rangle > 2 ( \omega ' + i_0 ) $ and $\omega ' = \max ( w_X ( \mathfrak c_{\mm} \mid \langle \mm ' , \rho \rangle < i_0 )) $. 
On the other hand, if $\langle \mm ' , \rho \rangle \geqslant i_0$ then $w_X ( \mathfrak c_{\mm ' } ) < - \delta   \langle \mm ' , \rho \rangle $. In particular, we have this last inequality if $\langle \varepsilon , \mm ' \rangle \geqslant i_0$. 
Finally by Proposition \ref{proposition:properties-of-the-weight}
\[
w_X \left (  \sum_{\mm ' \leqslant \mm } \mathfrak c_{\mm '} \LL^{-\langle \varepsilon , \mm - \mm ' \rangle}   \right ) < - \min \left ( \frac{1}{2}  \langle \varepsilon , \mm \rangle  , \delta \langle \rho , \mm \rangle  \right )  
\leqslant - \min \left ( \frac{1}{2} , \delta \right ) \langle \varepsilon , \mm \rangle
\]
for all $\mm $ such that $\langle \varepsilon , \mm \rangle \geqslant 2 ( \omega ' + i_0 ) + 1 $. 
\end{proof}
\end{mylemma}

Now we state and prove a weight-linear convergence criterion for Euler products. This result and its proof are a slight extension of Proposition 4.7.2.1 in  \cite{bilu2018motivic} to the framework of multivariate series.
\begin{myptn}\label{proposition:linear-conv-criterion}
Fix a complex variety $X$, an integer $r\geqslant 1$ and an $r$-tuple $\rho \in \mathbf Z_{\geqslant 1}^r$. 
Let $\XXX$ be a family $(X_\mm )_{\mm \in \NN^r \setminus \{ \mathbf 0 \}}$ of elements of $\ExpM _X$. 
Assume that 
%\[
%F(\TT)=1+\sum_{\mm\in \NN^r \setminus \{ \mathbf 0 \} } X_{\mm} \TT^{\mm} \in \ExpM _X [[\TT]]
%\] 
%is such that 
there exist an integer $M\geqslant 0$ 
and real numbers $\varepsilon > 0$, $\alpha < 1 $ and $\beta $ such that 
\begin{itemize}
\item $w_X (X_{\mm} ) \leqslant \left ( \langle \rho , \mm \rangle - \frac{1}{2} - \varepsilon \right ) w(X)$ whenever $1 \leqslant \langle \rho , \mm \rangle \leqslant  M$;
\item $w_X ( X_{\mm}) \leqslant \left ( \alpha \langle \rho , \mm \rangle + \beta - \frac{1}{2} \right ) w(X)$	whenever $\langle \rho , \mm \rangle > M$.
\end{itemize}
Let $\mathfrak b_{\mm} $ be the coefficient of multidegree $\mm \in \NN^r $ of the Euler product 
\[
\prod_{x\in X} \left ( 1+\sum_{\mm \in \NN^r \setminus \{\mathbf 0 \} } X_{\mm , x} \TT^{\mm} \right ) \in \ExpM _\mathbf C [[\TT]] .
\]
%where $F_x ( \TT ) = 1+\sum_{m\in \NN^r \setminus \{\mathbf 0 \} } X_{\mm , x} \TT^{\mm}$ for all $x\in X $. 
Then there exists $\delta \in \left ]0,\frac{1}{2}\right ]$  and $\delta ' > 0 $ such that 
\[
w \left ( \mathfrak b_{\mm} \mathfrak a_1^{\rho_1 m_1}  \cdots \mathfrak a_r^{\rho_r m_r}  \right ) \leqslant - \delta ' \langle \rho , \mm \rangle 
\]
for every
$\mathbf{m}\in \NN^r \setminus \{ \mathbf 0 \}$ and
$\mathfrak a_1 , ... , \mathfrak a_r \in \ExpM_\mathbf C$ such that $w( \mathfrak a_i ) <  - w ( X ) \left ( 1 - \delta + \frac{\beta }{M+1} \right )$ for all $1\leqslant i \leqslant r$.  
In particular, the Euler product  \[
\prod_{x\in X} \left ( 1+\sum_{\mm\in \NN^r \setminus \{\mathbf 0 \} } X_{\mm , x} T^{\langle \rho , \mm \rangle } \right ) \in \ExpM _\mathbf C [[T]]
\]
converges for $|T|<\LL^{-\frac{w(X)}{2}\left (  1 - \delta + \frac{\beta}{M+1}\right )}$
and takes non-zero values for $|T| \leqslant \LL ^ {-\frac{w(X)}{2}\left (  1 - \eta + \frac{\beta}{M+1}\right )} $ for any $0\leqslant \eta <\delta $. 
\begin{proof}
By definition of the motivic Euler product notation, for every $\mm \in \NN^r $ the coefficient $\mathfrak b_\mm$ is given by the sum
\[\mathfrak b_{\mm} = \sum_{\pi \text{ partition of }\mm } [ S^\pi \XXX ] .\]
Since $w ( \mathfrak b_{\mm} ) $ is smaller or equal to $\max (  \{ w ( S^\pi \XXX ) \mid \pi \text{ partition of } \mm \} ) $, it is enough to find a uniform bound for the weight of $S^\pi \XXX$. 
So we fix $\pi = ( n_\mathbf i )_{\mathbf i\in  \NN^r \setminus \{ \mathbf 0 \}}$ a partition of $\mm $, 
that is such that $\sum_{\mathbf i\in  \NN^r \setminus \{ \mathbf 0 \} } n_\mathbf i \mathbf i = \mm $.
Then 
\begin{align*}
w( S^\pi \XXX )  & \leqslant w_{S^\pi X } ( S^\pi \XXX ) + \dim ( S^\pi X ) \\
& \leqslant \sum_{\mathbf i\in  \NN^r \setminus \{ \mathbf 0 \}} n_\mathbf i w_X ( X_\mathbf i ) + \frac{1}{2} \sum_{\mathbf i\in  \NN^r \setminus \{ \mathbf 0 \}} n_\mathbf i  w ( X ) \\
& \leqslant \sum_{j\geqslant 1} \sum_{ \substack{  \mathbf i\in  \NN^r \setminus \{ \mathbf 0 \} \\ \langle \rho , \mathbf i \rangle =  j } } n_\mathbf i w_X ( X_\mathbf i )  + \frac{1}{2} \sum_{j\geqslant 1} \sum_{\substack { \mathbf i\in  \NN^r \setminus \{ \mathbf 0 \} \\\langle \rho , \mathbf i \rangle =  j}} n_\mathbf i w ( X ) \\
&\leqslant  \sum_{j = 1}^M \left (  \sum_{ \substack{  \mathbf i\in  \NN^r \setminus \{ \mathbf 0 \} \\ \langle \rho , \mathbf i \rangle =  j } } n_\mathbf{i} ( j - \varepsilon ) w ( X )\right ) + \sum_{j \geqslant M+1} \left (  \sum_{ \substack{  \mathbf i\in  \NN^r \setminus \{ \mathbf 0 \} \\ \langle \rho , \mathbf i \rangle =  j } } ( \alpha j   + \beta ) n_\mathbf i  w ( X ) \right ) \\
& \leqslant  \sum_{j = 1}^M \left (  \sum_{ \substack{  \mathbf i\in  \NN^r \setminus \{ \mathbf 0 \} \\ \langle \rho , \mathbf i \rangle =  j } } n_\mathbf{i} \left ( 1 - \frac{\varepsilon }{M} \right ) j  w ( X )\right ) + \sum_{j \geqslant M+1} \left (  \sum_{ \substack{  \mathbf i\in  \NN^r \setminus \{ \mathbf 0 \} \\ \langle \rho , \mathbf i \rangle =  j } } \left ( \alpha   + \frac{\beta}{M+1} \right ) j  n_\mathbf i  w ( X ) \right ) \\
& = \left ( 1 - \frac{\varepsilon}{M} \right )  \sum_{j = 1}^M \left (  \sum_{ \substack{  \mathbf i\in  \NN^r \setminus \{ \mathbf 0 \} \\ \langle \rho , \mathbf i \rangle =  j } } n_\mathbf{i} j  w ( X )\right ) + \left ( \alpha + \frac{\beta}{M+1} \right )  \sum_{j \geqslant M+1} \left (  \sum_{ \substack{  \mathbf i\in  \NN^r \setminus \{ \mathbf 0 \} \\ \langle \rho , \mathbf i \rangle =  j } } n_\mathbf{i} j  w ( X )\right ) \\
& \leqslant \left ( 1 - \delta + \frac{\beta}{M+1} \right ) \langle \rho , \mm \rangle w ( X ) 
\end{align*}
where $\delta \in \left ] 0 , \frac{1}{2} \right ]$ is given by $1-\delta = \max \left ( 1 - \frac{\varepsilon}{M} , \alpha \right )$ if $M\neq 0 $ and $\delta = 1- \alpha $ otherwise. This proves the first part of the proposition. 

To conclude the proof, remark that if $w ( \mathfrak a_k ) <  - w ( X ) \left ( 1 - \eta_k + \frac{\beta }{M+1} \right )$ for some $0\leqslant \eta_k  < \delta$, for all integer $k$ in $\{ 1 , ... , r\}$, then the weight of $S^\pi \XXX \mathfrak a_1^{\rho_1 m_1}  \cdots \mathfrak a_r^{\rho_r m_r}$ is negative for any $\mm $ non-zero and any partition of $\mm$. This argument shows that the product takes a non-zero value at $\mathfrak a$ : it is equal to $1$ plus some terms of negative weight. 
\end{proof}
\end{myptn}

%%%%%%%%%%%%%%%%
%%% SECTION %%%%
%%%%%%%%%%%%%%%%

\section{Compactifications of Additive Groups} \label{section:compactification-of-additive-groups}

\subsection{Back to our setting} \label{subsection:back-to-our-setting}
We recall in this paragraph a few facts and notation coming from the setting adopted in \cite{bilu2018motivic} and \cite{chambert2016motivic}.

From now on, we consider the situation presented in the introduction: briefly,
our data consists in a Zariski open subset $\UUU$ of a projective irreducible scheme $\XXX$ 
over an algebraically closed field $k$ of characteristic zero,
together with a non-constant proper morphism 
\[ \pi : \XXX \to C  \] 
whose generic fibre $X = \XXX_F$ is a smooth equivariant compactification of $G_F = \mathbf G_{a,F}^n$. Let $(D_\alpha)_{\alpha \in \AAA}$ 
be the family of irreducible components of the boundary divisor $X\setminus G_F$. 
Their linear classes freely generate the Picard group of $X$, as well as its effective cone.
An anticanonical divisor of $X$ is given by 
\[
\sum_\alpha \rho_\alpha D_\alpha
\]
for a certain family ${ \rho_\alpha \geqslant 2 }$ of integers. In particular, it is big.

Let $\mathscr D_\alpha$ be the Zariski closure of $D_\alpha$ in $\XXX$, for all $\alpha \in \AAA$. 
Using resolution of singularities in characteristic zero, one can assume that $\XXX$ is a \textit{good} model, that is to say, 
$\XXX$ is smooth over $k$ and the sum of the non-smooth fibres of $\XXX$ 
and of the $\mathscr D_\alpha$ is a divisor with strict normal crossings. 
We may assume that $\XXX \setminus \UUU$ is a divisor with strict normal crossings as well. These assumptions will not change our counting problem (see \S 3.2 and Lemma 3.4.1 of \cite{chambert2016motivic}). 

The restriction $U=\UUU_F$ is assumed to be a partial compactification of $G_F$. 
The boundary $D=X\setminus U$ is a divisor 
which can be written $D=\sum_{\alpha \in \AAA_D} D_\alpha$ for a  subset $\AAA_D $ of $ \AAA$. 
A log-anticanonical divisor with respect to $D$ is then $\sum_{\alpha \in \AAA} \rho_\alpha ' D_\alpha $ where ${ \rho_\alpha ' = \rho_\alpha  - 1 }$ if $D_\alpha$ is an irreducible component of $D$ and ${ \rho_\alpha ' = \rho_\alpha }$ otherwise. 
We set 
\[ 
\AAA_U = \AAA \setminus \AAA_D . 
\]

We fix as well a dense open subset $C_0 \subset C$ and assume that $\UUU ( \mathcal O_v ) \neq \varnothing $ whenever $v\in C_0 ( k )$.
For any $k$-point $v $ of $C$, let $\BBB_v $ be the set of irreducible components of $\pi^{-1} ( v)$.
Given an irreducible component $\beta \in \BBB_v$,
let us denote by $E_\beta$ the corresponding component  
and $\mu_\beta$ its multiplicity in the special fibre of $\XXX$ at $v$ (that is, the length of the $\mathcal  O_v$-module of regular functions on $E_\beta$ in $\XXX_v = \XXX \times_{\mathcal O_v} \Spec (k  ) $).
Let $\BBB = \cup_{v\in C(k)} \BBB_v$ be the union of all the $\BBB_v$'s 
and $\BBB_1$ the subset of components of multiplicity one. 
Over any place $v$ this subset restricts to a subset $\BBB_{1,v} = \BBB_1  \cap \mathscr  B_v $ of $\BBB_v$.
 If $v\in C_0 (k)$, we write $\BU_v$ for the subset of $\BBB_{1,v}$ 
of vertical components intersecting $ \UUU$. This definition makes sense since we assumed that $\XXX \setminus \UUU$ is a divisor: it is the union of the $\DDD_\alpha$ for $\alpha \in \AAA_D$ together with a finite number of vertical divisors. 
It is convenient to set $\BU_s = \BBB_{1, s} $ if $s\in S$.

For every $\alpha $ in $ \AAA$, we fix a line bundle $\mathscr L_\alpha$ on $\XXX$ extending $D_\alpha$.
There is a finitely supported family of integers $(e_\alpha^\beta)$ indexed by $\AAA$ and $\BBB $ such that 
\[
\mathscr L_\alpha  = \mathscr D_\alpha + \sum_{\beta \in \BBB } e_\alpha^\beta E_\beta 
\]
for any $\alpha $ in $\AAA$, as well as a finitely supported family of integers $(\rho^\beta ) $ indexed by $\BBB$ such that 
\[
- \ddiv ( \omega_X ) = \sum_{\alpha \in \AAA} \rho_\alpha \DDD_\alpha + \sum_{\beta \in \BBB} \rho^\beta E_\beta 
\]
with $\omega_X$ the unique (up to a multiplicative constant) $G_F$-invariant rational differential form on $X$, which is understood here as a rational section of the relative canonical bundle of $\XXX$ over $C$. 
By Lemmas 3.3.3 and 3.3.4 in \cite{chambert2016motivic}, there exists an open dense subset $C_1 \subset C_0$ such that $|\BBB_v |=1$ for every $v\in C_1 ( k)$ (so that $\BBB_v  = \BBB_{1,v} = \BBB_v^\UUU $ above $C_1$, since we assumed $\UUU ( \mathcal O_v ) \neq \varnothing $). Furthermore one can assume $e_\alpha^{\beta_v}= \rho^{\beta_v} = 0 $ for all $v\in C_1 ( k ) $ and $\mathscr D_\alpha \times_C C_1 \to C_1 $ smooth for all $\alpha \in \AAA_U $. This open subset $C_1$ can be understood as the set of places of \textit{good reduction}. 

Given any closed point $v\in C(k)$, for every subset $A$ of $\AAA$ and every irreducible component $\beta$ of multiplicity one above $v$, let $\Delta_v ( A, \beta ) $ be the set of points of the special fibre $\XXX_v $ of $\XXX\to C $ over $v$ belonging exclusively to $\mathscr D_\alpha $, for  every $\alpha \in A$, and to $E_\beta$: 
\[
\Delta_v ( A , \beta ) = \left ( \bigcap_{\alpha \in A } \DDD_{\alpha , v } \cap E_\beta \right )  \setminus \left ( \bigcup_{\substack{\alpha \notin A \\ \beta ' \neq \beta} } \DDD_{\alpha , v } \cup  E_{\beta '}  \right ) .
\]
 We denote by $\Omega_v( A , \beta ) $ the preimage of $\Delta_v ( A , \beta ) $ in the arc space $\mathscr L ( \XXX_{\mathcal O_v} ) $ through the projection $\mathscr L ( \XXX_{\mathcal O_v} )  \to \mathscr L_0 ( \XXX_{\mathcal O_v} )  = \XXX_v$.
Lemma 5.2.6 in \cite{chambert2016motivic} (recalled as Lemma 6.3.3.2 in \cite{bilu2018motivic}) gives the existence of an isomorphism 
\[
 \begin{array}{rll} \Theta : 
\Delta_v ( A , \beta ) \times \LLL ( \mathbf A^1 , 0 ) ^A \times \LLL ( \mathbf A^1 , 0 )^{n-|A|} & \longrightarrow & \Omega_v ( A , \beta )\\
x = ( x_\Delta , ( x_\alpha )_{\alpha \in A } , ( y_\alpha ) ) & \longmapsto	 & \Theta ( x)  
\end{array}
\]
which preserves the motivic measures and such that $\ord_0 ( x_\alpha ) = \ord_{\mathscr D_\alpha } ( \Theta ( x ) ) $ whenever $\alpha\in A$ and $ \ord_{\mathscr D_\alpha } ( \Theta ( x ) )  = 0$ otherwise. 
For further details, we refer to \cite[\S 2.4 \& \S 6.2]{chambert2016motivic}.

To conclude this section,
one may note that our problem is defined by a finite number of polynomial equations over $k$. Fixing an embedding of $k$ into $\CC$, we can assume that everything is defined over the field of complex numbers (using the fact that the definition of the moduli spaces of sections is functorial). Moreover, the assumption $\UUU ( \OOO_v ) \neq \varnothing $ for all $v\in C_0$ implies that for such $v$,
at least one of the $\Delta_v ( A , \beta ) $ has a $k$-point, thus has a $\CC$-point
\cite[\S 6.4.4]{bilu2018motivic}. 

%%% SUBSECTION %%%

\subsection{Expression of the Multivariate Zeta series}
For the sake of completeness, we start by recalling the method leading to the decomposition of the motivic Zeta function, as it is developed in \cite{chambert2016motivic} and \cite{bilu2018motivic}. A motivation for doing that is giving the reader a precise meaning of the summation symbols involved here: our asymptotic study will require to permute a sum over rational points and a limit in the completed ring $\widehat{\ExpM_k}$.

Given a $k$-point $v\in C(k)$, it is possible to define local intersection degrees $(g,D_\alpha )_v $ and $(g,E_\beta )_v$ for all $g\in G ( F_v ) $, $\alpha \in \AAA$ and $\beta \in \BBB_v$, in the following manner \cite[\S 3.3]{chambert2016motivic}.
Let ${g : \Spec ( F_v ) \to G_F} $ be such a point. Since $\XXX \to C $ is proper, by the valuative criterion of properness this map extends to a map ${ \tilde{ g } : \Spec ( \mathcal O_v ) \to \XXX}$. 
\begin{center}
\begin{tikzcd}
\Spec ( F_v ) \arrow[r,"g"] \arrow[d] & G_F \arrow[r] & \XXX \arrow[d," \pi "] \\
\Spec ( \mathcal O_v ) \arrow[rr] \arrow[urr, dashed, " \tilde{g} "] &  & C 
\end{tikzcd}
\end{center}
The non-negative local intersection number $( g , \DDD_\alpha )_v $ is given by the effective Cartier divisor on $\Spec ( \mathcal O_v )$ obtained by pulling-back $\DDD_\alpha$:
\[
\tilde{g}^* \DDD_\alpha = ( g , \DDD_\alpha )_v  [v ] . 
\]
The number $( g , E_\beta )_v$ is defined by pulling-back $E_\beta$ on $\Spec ( \mathcal O_v ) $. 
Such invariants satisfy the following two properties: 
\begin{itemize}
\item they are compatible with the global degree of the section $\sigma_g : C \to \XXX$ canonically extending $g\in G(F)$, with respect to $\mathscr D_\alpha$ for any $\alpha$:
\[
\deg_C ( \sigma_g^* ( \mathscr D_\alpha ) )  = \sum_{v\in C(k)} ( g , \mathscr D _\alpha )_v \, ; 
\] 
\item the $F_v$-point $g $ intersects exactly one vertical component of multiplicity one: there exists an unique $\beta \in \BBB_v$ such that 
\[ 
( g , E_\beta )_v = 1 \text{ with } \mu_\beta  = 1 
\]
and 
\[ 
\beta ' \neq \beta  \text{ implies } ( g  , E_{\beta ' } )_ v = 0.
\]
\end{itemize}
Then one defines the sets 
\[
G ( \mm, \beta  )_v = \{ g \in G ( F_v ) \mid ( g , E_\beta )_v = 1 \text{ and } ( g , \mathscr D_\alpha )_v = m_{\alpha} \text{ for all } \alpha \in \AAA \}
\]
for all $\mm \in \NN^\AAA$ and $\beta \in \BBB_v$. These sets provide a decomposition of $G(F_v)$ into disjoint bounded definable subsets \cite[Lemma 3.3.2]{chambert2016motivic}. 
A pair $(\mm  , \beta ) \in \NN^\AAA \times \BBB_v$ is said to be $v$-integral if 
\begin{itemize}
\item either $v\in C_0$, $\beta \in \BBB_{v}^\UUU$ and $m_{\alpha  }= 0$ for every $\alpha \in \AAA_D$; 
\item or $v\in C\setminus C_0$. 
\end{itemize}
One then introduces the corresponding sets 
\[
H ( \mm  , \beta  )_v = G ( \mm  , \beta  )_v \text{ if and only if } ( \mm  , \beta )\text{ is $v$-integral, and $\varnothing$ otherwise.}
\]
for any place $v\in C(k)$ and any pair $(\mm , \beta)$.  
These definitions provide adelic sets 
\[
H ( \m , \beta ) = \prod_{v\in C ( k ) } H ( \m _v , \beta_v  )_v \subset G ( \m , \beta ) = \prod_{v\in C ( k ) } G ( \m _v , \beta_v  )_v \subset G ( \mathbb A_{k(C)} ) 
\]
for any $\m = ( \m_v ) \in S^\nn C $ and $\beta = ( \beta_v ) \in \prod_v \BBB_v $. 
By propositions 6.2.3.3 and 6.2.4.2 of \cite{bilu2018motivic}, there exist 
\begin{itemize}
\item almost zero functions $s , s' : C \to \mathbf Z $ (playing the roles of $a$ and $b$ in \S \ref{paragraph:poisson-families});
\item an unbounded family $N = (N_{\mm})_{\mm \in \NN^\AAA}$ such that $N_\mathbf 0 = 0$;
\item for all $\mm \in \NN^\AAA $ and $\beta \in \prod_v \BBB_v$, a constructible subset $H_{\mathbf m , \beta } \subset \mathbf A_C^{n(s'-N_\mathbf m , s )} $; 
\end{itemize}
such that for every $v\in C ( k)$, the fibre of $H_{\mathbf m , \beta }$ at $v$ is $H( \mm , \beta_v ) _v$. 
By taking symmetric products, this data provides constructible subsets
\[
S^{\nn} (( H_{\mathbf m , \beta })_{\mm \in \NN^\AAA} ) \subset \AAA_{\nn} ( s '  , s  , N , 0 ) 
\]for any $\nn \in \NN^\AAA$,
which themselves define uniformly smooth constructible families of Schwartz-Bruhat functions (of any level $\nn $), denoted by
\[
\left ( (( \mathbf 1 _{H(\m , \beta)} ))_{\m \in S^\nn C} \right ) \in \ExpM_{\AAA_\nn ( s' ,  s , N , 0 ) } .
\]
Such functions correspond to the characteristic functions of $H(\m , \beta)$, for fixed $\beta$ and varying ${\m \in S^\nn C }$.
Therefore, by taking Fourier transforms, one gets uniformly compactly supported family 
\[
(\mathscr F ( \mathbf 1 _{H(\m , \beta)} ))_{\m \in S^\nn C} \in \ExpM_{\AAA_\nn ( \nu - s , \nu - s ' , 0 , N ) } 
\]
where $\nu$ is the conductor of a rational differential form $\omega \in \Omega_{k(C)/k}$. 

The moduli space $M_{\nn}$ of sections ${ \sigma : C \to \XXX }$ such that:
\begin{itemize}
	\item $\sigma$ maps the generic point $\eta_C$ of $C$ to a point of $G_F$;
	\item the image of $C_0$ by $\sigma$ is contained in $\UUU$;
	\item for all $\alpha $ in  $\AAA$, $\deg \left ( \sigma ^* \mathscr L_\alpha \right ) = n_\alpha $,
\end{itemize}
can be rewritten as a disjoint union of constructible subsets $M_\nn^\beta$, 
where for any $\beta \in \prod_{v\in C(k)} \BU_v$,
the subset $M_{\nn, \beta}$ is the set of sections $\sigma \in M_{\nn}$
such that $( \sigma ( \eta_C) , E_{\beta_v} )_v = 1 $ for all $v \in C ( k )$. 
In what follows, 
\[
\mathbf e^{\beta_v}  = ( e_\alpha^{\beta_v} )_{\alpha \in \AAA}   \qquad 
\mathbf e^{\beta}  = \sum_v \mathbf e^{\beta_v} \qquad
\nn^\beta =  \nn  - \mathbf e^{\beta} \in \NN^\AAA 
\]
hence elements of $M_\nn^\beta$ are sections $\sigma \in M_\nn $ such that $\left ( \deg ( \sigma^* \mathscr D_\alpha )\right )_{\alpha \in \AAA} = \nn^\beta $.
Since $\BU $ is the set of vertical components of multiplicity one lying above $S$ or contained in $\UUU$, one has $\mathbf e^{\beta_v} \in \NN^{\AAA_U}$ whenever $v\in C_0 ( k )$. Moreover, if $v\in C_1 ( k)$ then the condition $( \sigma ( \eta_C) , E_{\beta_v} )_v = 1$ is automatically satisfied since $| \BBB_v | = 1$. Thus the partition of $M_\nn$ we are describing here is actually finite, since it only depends on the local intersection numbers with respect to a finite number of vertical divisors. 

By \cite[Lemma 6.2.6.1]{bilu2018motivic}, if $M_\nn^\beta$ is non-empty there is a morphism $M_\nn^\beta \to S^{\nn^\beta} C$ of constructible sets sending a section $\sigma_g$ to the tuple of zero-cycles $\sum_{v\in C (k  )} (( g , \DDD_\alpha ) _v )_{\alpha \in \AAA} [v] $.  
The \textit{coarse height motivic zeta function} $Z(\TT ) $ can be rewritten 
\begin{equation}
Z (\TT ) =  \sum_{\beta \in \prod_{v}  \BU_v} \sum_{\nn^\beta \in \NN^\AAA}   [ M_{\nn^\beta + \mathbf e^\beta}^\beta ] \TT^{  \nn^\beta + \mathbf e^\beta }  = \sum_{\beta \in \prod_{v} \BU_v} \TT^{\mathbf e^\beta }  Z^\beta ( \TT )  . 
\end{equation}
where $  Z^\beta ( \TT )  = \sum_{\nn^\beta \in \NN^\AAA}   [ M_{\nn^\beta + \mathbf e^\beta } ^\beta ] \TT^{  \nn^\beta } $ for every $\beta \in \prod_{v}  \BU_v$ (note again that this last product is finite).

When one considers $S$-integral points and use this function, it appears that it is not precise enough: poles of multiple order appear, which means that we miss relevant invariants which correspond to distinct components of the moduli space of sections. 
In order to deal with it, remark that a section $\sigma : C \to \XXX$ intersects the divisors $\DDD_\alpha$ for $\alpha \in \AAA_D$
 only above the finite set of closed points $S = C \setminus C_0$ 
 and \textit{each point} $s\in S$ gives an \textit{invariant } 
 \[
 \left ((\sigma ( \eta_C ) , \DDD_\alpha )_s\right )_{\alpha \in \AAA_D} .
 \]
This leads us to define for any $(\nn , \beta )$ and $\m_S = {( \m_s )_{s\in S} \in \NN^{\AAA_D \times S}}$ 
the subset $M_{\nn , \m_S}^\beta$ of sections $ \sigma \in M_\nn^\beta  $ 
such that $ (  ( \sigma ( \eta_C) , \mathscr D_\alpha ) )_s )_{s\in S} = \m_S$. 
This gives a decomposition of $M_\nn^\beta$ into a finite disjoint union of definable subsets. 
If $\nn^\beta = (\nn^\beta_U , \nn^\beta_D) $ is an element of $\NN^\AAA =  \NN^{\AAA_U} \times \NN^{\AAA_D}$, 
this  subset can be identified with the fibre of $M_\nn^\beta \to S^{\nn^\beta} C$ 
over the points $\m\in S^{\nn^\beta} C = S^{\nn^\beta_U} C \times S^{\nn^\beta_D }C $ 
whose image by the projection $ S^{\nn^\beta} C  \to S^{\nn^\beta_D }C $ has support in $S=C\setminus C_0$ 
and is given by $\m_S$. 

\begin{mydef}\label{def:refined-zeta-function}Let $\UU = \left ( \UU_0 , ( \UU_s )_{s\in S} \right )$ be a family of indeterminates indexed by the set
\[
\AAA_U  \sqcup ( \AAA_D \times S )  
\] and let 
\[ 
\mathbf e^\beta = \left  ( \sum_{v\in C_0} \mathbf e^{\beta_v} + \sum_{s\in S} \mathbf e^{\beta_s}_{\AAA_U} , \left (  \mathbf e^{\beta_s}_{\AAA_D} \right )_{s\in S} \right ) 
\]
for all $\beta\in \prod_{v} \BU_v$. 

We define the \textit{refined motivic height zeta function} $\ZZZ ( \UU )$ by
\begin{equation}\label{equation:Z(T)-rewritten-with-M-n-beta}
\ZZZ ( \UU )  = \sum_{\beta \in \prod_{v} \BU_v} \UU^{\mathbf e^\beta }  \ZZZ^\beta ( \UU ) 
\end{equation}
where 
\[
\ZZZ^\beta ( \UU )  =   \sum_{ \substack{  \nn_{\AAA_U}^\beta \in \NN^{\AAA_U}  \\  \m_S \in \NN^{\AAA_D \times S}  }  } [ M_{\nn^\beta + \mathbf e^\beta  , \m_S}^\beta ] \UU^{ \nn_U    + \m_S } 
\]
for every $\beta \in \prod_{v} \BU_v$.  
\end{mydef}

\begin{myremark}\label{remark:refined-corse-specialisation}
By definition, the specialisation of $\ZZZ ( \UU ) $ 
obtained by replacing respectively ${\UU_0 = (U_\alpha )_{\alpha \in \AAA_U}}$ by $(T_\alpha )_{\alpha \in \AAA_U}$ 
and ${\UU_s = (U_{\alpha , s} )_{\alpha \in \AAA_D} }$ by $(T_\alpha)_{\alpha \in \AAA_D}$, for every $s\in  S $,
is the coarse zeta function $Z ( \TT ) $.   
Furthermore, if $C_0 = C$ and $\UUU = \XXX$ then this refined function $\ZZZ$ coincides with the coarse one $Z$, since in that case $S=\varnothing$ and $\AAA = \AAA_U$. A similar remark can be done for $\ZZZ^\beta $ and $Z^\beta$. 
\end{myremark}

The map sending a section $\sigma : C \to X$ to $\sigma ( \eta_C ) \in X( F) $ induces an exact correspondence between 
\begin{center}
sections $\sigma \in M_{\nn}^\beta $ such that $\sum_{v\in C(k)} (( g , \mathscr D _\alpha )_v )_{\alpha \in \AAA } [v] = \m  \in S^{\nn^\beta} C $
\end{center}
and 
\begin{center}
elements of $G ( F) \cap H ( \m , \beta ) $. 
\end{center}
Remark that if  $H ( \m , \beta ) $ is non-empty 
then the image of $\m$ by the projection $ S^{\nn^\beta} C  \to S^{\nn^\beta_D} C $ has support in $S$. 

By definition of the summation over rational points (\ref{def:summation-rational-points}), one has the equalities in $\mathscr M_{\kappa ( \m ) }$							
\[
\sum_{x\in \kappa ( \m ) (C) } \mathbf 1_{G(F) \cap H(\m,\beta )} ( x ) = [G(F) \cap H ( \m , \beta ) ] = \left ( M_\nn^\beta \right )_{\m} 
\]
where $\left ( M_\nn^\beta \right )_{\m} $ is the class of the fibre of $M_\nn^\beta \to S^{\nn^\beta }C$ over the schematic point $\m \in S^{\nn^\beta} C $.
Applying the motivic Poisson formula for families (see \S \ref{paragraph:poisson-families}) and following Bilu's computations, we obtain for the coarse subspaces 
\begin{align*} 
\left  [M_\nn^\beta \right ] & = 
\sum_{\m \in S^{\nn ^\beta } (  C ) } [ G(F) \cap H ( \m , \beta )] \\
& = \LL^{(1-g)n} \sum_{\xi \in k(C) ^n } \sum_{\m \in S^{\nn ^\beta} (  C )}   \mathscr F ( \mathbf 1_{H(\m , \beta )} ) (\xi )   
\end{align*}
and for the refined subspaces
\begin{align*} 
\left [   M_{\nn , \m_S}^\beta \right ] &  = \sum_{ \m \in S^{\nn_U ^\beta} ( C  ) }  [G(F) \cap H ( \m + \m_S , \beta ) ] \\
& = \LL^{(1-g)n} \sum_{\xi \in k(C) ^n } \sum_{ \m \in S^{\nn_U^\beta} ( C  ) } \mathscr F ( \mathbf 1_{H(\m + \m_S , \beta )} ) (\xi ) . 
\end{align*}
Since $\left  (  \mathscr F ( \mathbf 1_{H(\m , \beta )} ) \right )_{\m \in S^\nn C} = S^\nn ( (  \mathscr F \mathbf 1_{H(\mm , \beta )}    )_{\mm \in \NN^\AAA} )  $ (see \cite[Proposition 5.4.4.2]{bilu2018motivic}), 
by definition of the motivic Euler product notation
one can write
\begin{align*}
Z^\beta ( \TT ) & = \LL^{(1-g)n} \sum_{\xi \in k(C) ^n } \sum_{\nn \in \NN^\AAA} \sum_{\m \in S^\nn C} \mathscr F ( \mathbf 1_{H(\m , \beta )} ) (\xi )\TT^\nn \\
& = \LL^{(1-g)n} \sum_{\xi \in k(C) ^n } \prod_{v\in C} \left (  \sum_{\m_v \in \NN^\AAA } \mathscr F (  \mathbf 1_{H(\m_v , \beta_v )_v} )   (\xi_v )\TT^{\m_v} \right ) \\
& = \LL^{(1-g)n} \sum_{\xi \in k(C) ^n } \prod_{v\in C} Z^\beta_v ( \TT , \xi ) 
\end{align*}
and since the Euler product is compatible with finite products, we have 
\begin{align*}
\ZZZ^\beta ( \UU )  = & \LL^{(1-g)n} \sum_{\xi \in k(C) ^n } \sum_{\substack{\nn_U \in \NN^{\AAA_U} \\ \nn_D \in \NN^{\AAA_D} }} \sum_{ \substack{ \m \in S^{\nn_U} (C) \\ \m_S \in S^{\nn_D} ( C \setminus C_0 ) } } \mathscr F ( \mathbf 1_{H(\m + \m_S , \beta )} ) (\xi )\UU^{\nn_U + \m_S} \\ 
& =\LL^{(1-g)n} \sum_{\xi \in k(C) ^n } \left ( \prod_{v\in C_0} \left (  \sum_{\m_v \in \NN^{\AAA_U} } \mathscr F (  \mathbf 1_{H(\m_v  , \beta_v )_v} )   (\xi_v )\UU^{\m_v} \right ) \right. \\
& \qquad \times \left. \prod_{s\in S} \left (  \sum_{\m_s \in \NN^{\AAA_U } \times \NN^{ \{ s \} \times \AAA_D } } \mathscr F (  \mathbf 1_{H(\m_s  , \beta_s )_s} )   (\xi_s )\UU^{\m_s} \right ) \right ) \\
& = \LL^{(1-g)n} \sum_{\xi \in k(C) ^n } \prod_{v\in C_0} \ZZZ_v ^{\beta_v} \left  ( \UU_0 , \xi \right )  \prod_{s\in S} \ZZZ_s ^{\beta_s} \left  ( \left ( \UU_0, \UU_s \right ) , \xi \right ).
\end{align*} 
This last decomposition is consistent with the one of $Z^\beta ( \TT ) $ if one applies the specialisation of Remark \ref{remark:refined-corse-specialisation}.
Combined with (\ref{equation:Z(T)-rewritten-with-M-n-beta}) it finally gives 
\[
\ZZZ ( \UU ) 
=
 \LL^{(1-g)n} \sum_{\xi \in k( C ) ^n } \prod_{v\in C} \left  (    \sum_{ \beta_v\in \BU_v  } \UU^{ \mathbf e^{\beta_v} }  \ZZZ^{\beta_v}_v \left  ( \UU , \xi_v \right ) \right ) .
\]
In this expression, the summation over $k(C)^n$ is actually a summation over $ L ( \tilde E  )  ^n$, where $ L ( \tilde E  ) $ is the Riemann-Roch space of the $k$-divisor 
\begin{equation}\label{def:divisor-E-tilde}
\tilde{E} = - \sum_{v} ( \nu_v - s_v ) [ v ] ,
\end{equation}
$s : C \to \ZZ $ being the almost zero function of \cite[Proposition 6.2.3.3]{bilu2018motivic}
and 
$\nu $ being the order function of $\omega \in \Omega_{k(C)/k}$
defined at the beginning of \S \ref{subsubsection:motivic-fourier-transform-in-families}. 

Hence the multivariate zeta functions $\ZZZ ( \UU ) $ and $Z(\TT )$ can be written as 
\begin{align*}
\ZZZ ( \UU ) & =   \LL^{(1-g)n} \sum_{\xi \in V} \ZZZ ( \UU , \xi ) \\
Z ( \TT ) &  =  \LL^{(1-g)n} \sum_{\xi \in V} Z ( \TT , \xi )
\end{align*}
where $V$ is the finite dimensional $k$-vector space $V = L(\tilde E)^n$ and each $\ZZZ(\UU , \xi ) $ (respectively $Z(\TT , \xi )$) can be expressed as a motivic Euler product with local factors 
\[ 
\ZZZ_v (\UU , \xi ) = \sum_{\beta_v \in \BBB_{\UUU , v }} \UU^{\mathbf e^{\beta_v}} \ZZZ_v^{\beta_v} (\UU , \xi )
\]
(resp. $Z_v(\TT , \xi ) $) for any place $v\in C$. 
For  any $\xi $ in $V$, we  will first study the asymptotic behaviour of the $\mm$-th coefficient of the Euler product 
\begin{equation}
\notag
	\prod_{v\in C} \ZZZ_v ( \UU , \xi )
\end{equation}
when $\min_{ ( \alpha , s ) \in \AAA \times ( \{ 0 \} \cup S ) }  (m_{\alpha , s} ) $ tends to infinity. 
When we restrict this product to $C_0$, its local factors coincide with the ones of the coarse version $Z ( \TT , \xi ) $. Therefore it is natural for us to identify them and keep the notations of \cite{bilu2018motivic,chambert2016motivic} for places of $C_0$.
 
Note that this notation is slightly abusive: it actually means that we restrict a certain relative motivic Euler product to a constructible subset of $V$ containing $\xi$. Indeed,
recall that we use the commutativity of the diagram
\[
\begin{tikzcd}
\ExpM_{V \times S^{\nn} C}  \arrow[r] \arrow[d] & \ExpM_{S^{\nn} C} \arrow[d]\\
\ExpM_{V }  \arrow[r] & \ExpM_k
\end{tikzcd}
\]
while summing over $\xi$ and $\mm$ the families $\left  ( \theta_\nn^*   \mathscr F ( \mathbf 1_{H(\m , \beta )} ) \right )_{\m \in S^\nn C} \in \ExpM_{ V \times S^{\nn} C}$. In particular,
one can understand 
\[
\prod_{v\in C} \left ( \sum_{\mathsf m_v \in \NN^\AAA} \FFF (  \mathbf 1_{H(\m_v , \beta_v )_v} )   \TT^{\m_v}  \right )  = \sum_{\nn \in \NN^\AAA} \sum_{\m \in S^{\nn } (  C )}   \mathscr F ( \mathbf 1_{H(\m , \beta )} ) \TT^\nn 
\]
as a relative motivic Euler product with coefficients in $\ExpM_{V} $ (in which we omit $\theta_\nn^*$). Its convergence will be studied with respect to a finite constructible partition of $V$ (described in \S \ref{subsection:uniform-convergence}).
Since the motivic sum 
\[
\sum_{\xi \in V} : \ExpM_{V}  \to \ExpM_k 
\]
is a morphism, it is compatible with such a cutting process. 
In the end, the remaining task will be to justify that one can actually permute the sum over rational points $\sum_{\xi \in k(C)^n }$ and the limit of the coefficients of $\ZZZ ( \UU , \xi ) $. 

%\subsubsection{Expression of the local factors}

It is shown in \cite{chambert2016motivic} that the local factor of $Z ( \TT ) $ can be rewritten as a motivic integral over the arc spaces $\Omega_v ( A , \beta ) $ defined in \S \ref{subsection:back-to-our-setting}. When $v\in C_0$, this procedure gives 
\[
Z_v ( \TT , \xi ) = \sum_{\beta_v \in \BU_v }
 \TT^{e^{\beta_v} } Z_v^{\beta_v} ( \TT , \xi ) = \sum_{\substack{A\subset \AAA_U \\ \beta_v \in \BU_v }}
 \TT^{e^{\beta_v} } \LL^{\rho^{\beta_v}} 
\int_{\Omega_v( A ,\beta_v )} \prod_{\alpha \in A} \left ( \LL^{\rho_\alpha } T_\alpha \right ) ^{\ord_v ( x_\alpha ) } \ee ( \langle x , \xi \rangle ) \mathrm{d}x  
\]
(see  \cite[\S 6.1-2]{chambert2016motivic} and \cite[\S 6.3.1]{bilu2018motivic}). The $\rho$-weight-linear convergence of the local factor at $T_\alpha = \LL^{- \rho_\alpha ' }$ is actually proved in \cite[Lemma 6.3.5.1]{bilu2018motivic}. 
Above places $s\in S$, the expression looks similar, replacing $\AAA_U$ by $\AAA$ 
\[
Z_s ( \TT , \xi ) = 
\sum_{\beta_s \in \BU_s }
 \TT^{e^{\beta_s} } Z_s^{\beta_s} ( \TT , \xi ) 
= \sum_{\substack {A\subset \AAA \\ \beta_s \in \BU_s } }
 \TT^{e^{\beta_s} }
\int_{\Omega_s( A ,\beta )} \prod_{\alpha \in A} \left ( \LL^{\rho_\alpha } T_\alpha \right ) ^{\ord_s ( x_\alpha ) } \ee ( \langle x , \xi \rangle ) \mathrm{d}x  .
\] 
\begin{myremark}\label{remark:dependance-of-local-factors}
Remark that for any character $\xi$ and place $v \in C_0 ( k )$  we have a decomposition 
\[
Z_v^{\beta_v} ( \TT , \xi ) = \sum_{A\subset \AAA_U  }    \TT^{e^{\beta_v} } Z_{v,A}^{\beta_v} ( \TT  ,  \xi ) 
\]
where $Z_{v,A}^{\beta_v}  ( \TT , \xi ) $ only depends on the indeterminates $(T_\alpha )_{\alpha \in A}$.

This remark remains valid for places $s$ of $S$, if one just  replaces $\AAA_U$ by $\AAA$ in the decomposition above, and adapting this decomposition to $\ZZZ ( \UU , \xi )$ is straightforward. 
\end{myremark}

%%% SUBSECTION %%%
\subsection{Convergence of Euler products in our setting} 
As the reader familiar with Tauberian theorem in complex analysis might guess, 
the asymptotic behaviour of the coefficients of $Z(T)$ is closely linked to the poles of the series. 
The precise study of these poles has been done by Bilu in \cite[Chapter 6]{bilu2018motivic}. 
In this subsection we will both recall and adapt the relevant results for our purpose; 
in particular, we precisely check weight-linear convergence of these products, 
which is locally uniform with respect to a stratification of the space of characters. 
This information will be crucial both for the control of various error terms appearing through the last section of this paper and for the final step of our proof.   

Since the motivic Euler product notation is compatible with finite products, it will be enough to prove convergence over $C_0$ (and even $C_1$). As we already pointed out, in this situation the local factors of the coarse zeta function and refined zeta function coincide. Therefore the content of this paragraph stays valid if one replaces everywhere $ Z_v ( \TT , \xi )  $ by $\ZZZ_v ( \UU , \xi ) $ and the inderterminates $T_\alpha$ by $U_\alpha$, $\alpha \in \AAA_U$, in Notation \ref{notation:F_v(T,0)} and \ref{notation:F_v(T,xi)} below.

\subsubsection{Trivial character} In this paragraph we study the main term of the motivic Zeta function. 
\begin{mynotation}\label{notation:F_v(T,0)}
For any place $v$  of $C$ we set 
\[ 
F_v ( \TT , 0 ) =  Z_v ( \TT , 0 ) \prod_{\alpha \in \AAA_U } ( 1 - \LL^{\rho_\alpha  - 1 }T_\alpha ) . \] 
\end{mynotation}
One can directly compute the local factors corresponding to the trivial character \cite[\S 6]{chambert2016motivic}. First assume that $v$ is a place of $C_0$. Then,
\begin{align*}
Z_v ( \TT , 0 )  & = \sum_{\beta \in \BU_v }
 \TT^{e^{\beta_v} } Z_v^{\beta_v} ( \TT , 0 ) \\
 & =  \sum_{\beta_v \in \BU_v }  \TT^{e^{\beta_v} } \LL^{\rho^\beta} \sum_{A\subset \AAA_U } [\Delta_v ( A , \beta ) ] \LL^{-n+|A|} ( 1 - \LL^{-1})^{|A|} \prod_{\alpha\in A} \frac{\LL^{\rho_\alpha -1}T_\alpha}{1- \LL^{\rho_\alpha -1}T_\alpha} 
\end{align*}
In case $v$ is a place of the dense subset $C_1\subset C_0$, since $\BBB_v = \{ \beta_v \}$ and both $e_\alpha^\beta$ and $\rho_{\beta}$ equal zero, this expression becomes slightly nicer: 
\[
Z_v ( \TT , 0 ) = Z_v^{\beta_v} ( \TT , 0 ) =  \sum_{\substack{  A\subset \AAA_U }} [\Delta_v ( A ) ] \LL^{-n+|A|} ( 1 - \LL^{-1})^{|A|} \prod_{\alpha\in A} \frac{\LL^{\rho_\alpha -1}T_\alpha}{1- \LL^{\rho_\alpha -1}T_\alpha} .
\]
Concerning places of $S=C\setminus C_0$, the local factor becomes
\[
Z_v ( \TT , 0 ) = \sum_{\substack{ \beta \in \BU_v \\ A\subset \AAA }}  \TT^{e^{\beta_v} } \LL^{\rho^\beta} [\Delta_v ( A , \beta ) ] \LL^{-n+|A|} ( 1 - \LL^{-1})^{|A|} \prod_{\alpha\in A} \frac{\LL^{\rho_\alpha -1}T_\alpha}{1- \LL^{\rho_\alpha -1}T_\alpha} .
\]

We will use these expressions to prove the following proposition. 
\begin{myptn} \label{proposition:convergence-residue-of-Zeta-0}
The product 
\[
\prod_{v\in C_0} F_v ( \TT  , 0 ) 
\]
is $\rho$-weight-linearly convergent at $T_\alpha = \LL^{-\rho_\alpha }$ and the resulting sum is a non-zero effective element of $\widehat{ \mathscr M_k}$.  
%Moreover,
%converges for $|T|<\LL^{-1+\delta}$ and takes a non-zero effective value in $\widehat{\mathscr M}_k $ at $T=\LL^{-1}$. 
\end{myptn}

\begin{proof}%[Proof of Proposition \ref{proposition:convergence-residue-of-Zeta-0} ] 
The proof is essentially the same as the one of Proposition 6.3.5.2 of \cite{bilu2018motivic}. We write down the details here for the sake of completeness. In particular, it is important to check that we are in the situation of Proposition \ref{proposition:linear-conv-criterion}, in order to obtain linear convergence with respect to $\rho $. 

It is enough to prove convergence for places of the dense open subset $C_1$ of $C_0$. Furthermore one can assume $\AAA_U = \AAA$.
For a place $v\in C_1$ the local factor $F_v ( \TT, 0 ) =  \prod_{\alpha \in \AAA} ( 1 - \LL^{\rho_\alpha -1} \mathbf T ) Z_v ( \TT , 0 ) $ is given by 
\[
F_v ( \TT , 0 )  = \sum_{A\subset \AAA } [\Delta_v ( A ) ] \LL^{-n+|A|} ( 1 - \LL ^{-1} )^{|A|} \prod_{\alpha \in A } \LL^{\rho_\alpha -1} T_\alpha \prod_{\alpha \in \AAA \setminus A} ( 1 - \LL^{\rho_\alpha - 1 }T_\alpha ) 
\]
If we expand the product over $\AAA \setminus A$ for any $A\subset \AAA$ we get 
\[
\prod_{\alpha ' \in \AAA \setminus A } ( 1 - \LL^{\rho_{\alpha '} - 1 }T_{\alpha ' } ) = \sum_{B\subset \AAA \setminus A } (-1)^{|B|} \prod_{\alpha ' \in B } \LL^{\rho_{\alpha ' } - 1 } T_{\alpha ' }
\]
and 
\[
F_v ( \TT , 0 )  = \sum_{A\subset \AAA }  \sum_{B\subset \AAA \setminus A } (-1)^{|B|} [\Delta_v ( A ) ] \LL^{-n+|A|} ( 1 - \LL ^{-1} )^{|A|} \prod_{\alpha \in A } \LL^{\rho_\alpha -1} T_\alpha  \prod_{\alpha ' \in B } \LL^{\rho_{\alpha ' } - 1 } T_{\alpha ' } . 
\]
Remark that $[\Delta_v ( \varnothing ) ] = [\mathbf G_a ^n ] = \LL^n$.
Then the previous sums can be decomposed with respect to the cardinalities of  the sets $A$ and $B$: 
\begin{align*}
F_v ( \TT , 0 ) & =  1  -  \sum_{\alpha '  \in \AAA } \LL^{\rho_{\alpha ' } -1}  T_{\alpha ' }  +  \sum_{\substack{B\subset \AAA \\ |B|\geqslant 2 } } (-1)^{|B|}   \prod_{\alpha ' \in B } \LL^{\rho_{\alpha ' } - 1 } T_{\alpha ' }  \tag{$A=\varnothing$} \\
& + \sum_{\alpha \in \AAA }   [\Delta_v ( \{ \alpha \} ) ] \LL^{-n+1} ( 1 - \LL ^{-1} ) \LL^{\rho_\alpha -1} T_\alpha  \tag{$|A|=1$, $B=\varnothing$}  \\
& + \sum_{\alpha \in \AAA }  \sum_{ \substack{B\subset \AAA \setminus \{ \alpha \} \\ B \neq \varnothing } } (-1)^{|B|} [\Delta_v ( \{ \alpha \} ) ] \LL^{-n+1} ( 1 - \LL ^{-1} )  \LL^{\rho_\alpha -1} T_\alpha  \prod_{\alpha ' \in B } \LL^{\rho_{\alpha ' } - 1 } T_{\alpha ' }  \tag{$|A|=1$, $B\neq \varnothing $}   \\
&  +  \sum_{\substack{A\subset \AAA \\ |A|\geqslant 2 } }  \sum_{B\subset \AAA \setminus A } (-1)^{|B|} [\Delta_v ( A ) ] \LL^{-n+|A|} ( 1 - \LL ^{-1} )^{|A|} \prod_{\alpha \in A } \LL^{\rho_\alpha -1} T_\alpha  \prod_{\alpha ' \in B } \LL^{\rho_{\alpha ' } - 1 } T_{\alpha ' } \tag{$|A|\geqslant 2$} .
\end{align*}
The definition of $\Delta_v ( A ) $ gives the equality of classes 
\[
 [\mathscr D_{\alpha , v} ] =  [\Delta_v ( \{ \alpha \} ) ] + \sum_{\substack{ A' \subset \AAA\setminus \{ \alpha \} \\ A' \neq \varnothing} } [\Delta_v ( A ' \cup \{ \alpha \} ) ] 
\] for every $\alpha \in \AAA$.
Finally %\textcolor{blue}{[details to be given]}
\[
F_v ( \TT , 0 ) = 1 + \sum_{ \alpha \in \AAA } \left ( [\mathscr D_{\alpha , v} ] - \LL^{n-1} \right )\LL^{-n} \LL^{\rho_\alpha} T_\alpha  + P_v ( \TT )
\]
where $P_v ( \TT ) $ is the polynomial
\begin{align*}
P_v ( \TT )  & = \sum_{\substack{B\subset \AAA \\ |B|\geqslant 2 } } (-1)^{|B|}   \prod_{\alpha ' \in B } \LL^{\rho_{\alpha ' } - 1 } T_{\alpha ' } \\
& - \sum_{\alpha \in \AAA }   \sum_{\substack{ A' \subset \AAA\setminus \{ \alpha \} \\ A' \neq \varnothing } } [\Delta_v ( A ' \cup \{ \alpha \} ) ] \LL^{-n+1} ( 1 - \LL ^{-1} ) \LL^{\rho_\alpha -1} T_\alpha \\
& + \sum_{\substack{A\subset \AAA \\ A \neq \varnothing  } }  \sum_{ \substack{ B\subset \AAA \setminus A  \\ |A|=1 \Rightarrow B \neq \varnothing } } (-1)^{|B|} [\Delta_v ( A ) ] \LL^{-n+|A|} ( 1 - \LL ^{-1} )^{|A|} \prod_{\alpha \in A } \LL^{\rho_\alpha -1} T_\alpha  \prod_{\alpha ' \in B } \LL^{\rho_{\alpha ' } - 1 } T_{\alpha ' }  .
\end{align*}
% whose coefficients verify the requisites of Proposition \ref{proposition:linear-conv-criterion}, for $M=\deg ( P (T^{\rho_\alpha} )) $, $X=C_1$, $\varepsilon = \frac{1}{2}$ and $\beta = 0$.
Let us analyse the dimensions of the coefficients of this polynomial term by term, and compare them to the multidegrees of the corresponding monomials. In the sum of the first line, since $|B|\geqslant 2$, the corresponding coefficient has dimension at most $-2 + \sum_{\alpha \in B} \rho_\alpha $. 
Here the corresponding monomial is $\TT ^\mm $ with $m_\alpha = 1$ if $\alpha \in B $ and $0$ otherwise; thus  $\sum_{\alpha \in B } \rho_\alpha  = \langle \rho , \mm \rangle $. 
Concerning the sum of the second line, 
since we assumed that $\XXX$ is a good model, 
the dimension of $\Delta_v ( A ' \cup \{ \alpha \} ) $ 
is at most $n - 2 $ and the whole coefficient has dimension at most $ \rho_\alpha -2  $. 
This remark still applies to the coefficients of the third and last sum when $|A|\geqslant 2$:
they have dimension at most $-2 + \sum_{\alpha \in A\cup B} \rho_\alpha $. 
It corresponds to the monomial $\mathbf T^\mm  $  with $m_\alpha = 1$ if $\alpha \in A \cup B $ and $0$ otherwise; thus  $\langle \rho  , \mm \rangle = \sum_{\alpha \in A\cup B} \rho_\alpha$. 
Finally, in the case $|A|=1$, then $|A\cup B|\geqslant 2 $ 
and the relation 
\[ \prod_{\alpha \in A \cup B } \LL^{\rho_\alpha -1} = \LL^{ \sum_{\alpha \in A\cup B} \rho_\alpha  - |A \cup B|} \]
ensures that the dimension of the corresponding coefficient is again at most $-2 + \sum_{\alpha \in A\cup B} \rho_\alpha $. 

Remark that our computations do not depend on $v$ in the sense that there exist polynomials $F ( \TT , 0 )$ and $P ( \TT ) $ with coefficients in $\ExpM_{C_1}$, such that their pull-backs by $v$ are respectively $F_v$ and $P_v$. 
Therefore we will be able to use Proposition \ref{proposition:linear-conv-criterion} 
with $X=C_1$, ${ \varepsilon = \frac{1}{2}} $, ${ M=\deg ( P ( (T^{\rho_\alpha} )_{\alpha \in \AAA} )) }$ and $\beta = 0$ 
so that the first condition
of Proposition \ref{proposition:linear-conv-criterion} 
becomes in our case 
\begin{equation}\label{inequality:weight-of-polynomial-P(xi)}
w_{C_1} (  \mathfrak c_\mm ) \leqslant 2 \left (  \langle \rho , \mm \rangle - 1 \right )  . 
\end{equation}
By the last property of the weight recalled in Proposition \ref{proposition:properties-of-the-weight}, we  obtain the crucial argument of this proof, which is the inequality
\[
w_{C_1} \left ( \left ( [\mathscr D_{\alpha , v} ] - \LL^{n-1} \right )\LL^{-n} \LL^{\rho_\alpha}  \right ) \leqslant 2 ( n - 1 ) - 2 n + 2 \rho_\alpha = 2 ( \rho _\alpha - 1 ).
\]
As Bilu points out in \cite[Remarks 6.3.4.2-3]{bilu2018motivic}, the bounds on dimensions obtained in the previous paragraph ensure that the coefficients of multidegree $\mm $ of $P_v ( \TT ) $ satisfy (\ref{inequality:weight-of-polynomial-P(xi)}). 
Thus we can apply Proposition \ref{proposition:linear-conv-criterion} and $\rho$-weight-linear convergence over $C_1$ at $T_\alpha = \LL^{-\rho_\alpha}$ follows. Since the Euler product notation is compatible with finite products and since we already know that the local factors converge weight-linearly, applying Lemma \ref{lemma:weight-convergence-product-stability} we deduce that the $\rho$-weight-linear convergence holds for the product over $C_0$. 
\end{proof}

\begin{myremark}\label{remark:interpretation-local-term-volume}
Since for any place $v\in C_1$ the fibre $\UUU_v $ is the disjoint union of all the $\Delta_v ( A) $ for $A\subset \AAA_U$, and in general $E_{\beta , v }$ is the disjoint union of all the $\Delta_v ( A , \beta ) $, 
it is straightforward to check that the value at $T_\alpha =\LL^{-\rho_\alpha '} $ of the motivic Euler product $\prod_{v\in C_0} \TT^{\mathbf e^{\beta_v}} F_v^\beta ( \TT , 0 ) $ is 
\begin{align*}
& \prod_{v\in C_0} \left ( \TT^{\ee^{\beta_v}} Z_v^{\beta_v} ( \TT , 0 ) \prod_{\alpha \in \AAA_U} ( 1 - \LL^{\rho_\alpha -1} T^{\rho_\alpha}  ) \right ) \left  (  ( \LL^{-\rho_\alpha } )_{\alpha \in \AAA} \right ) \\
& =  \prod_{v\in C_1 } \left ( 
 (1-\LL^{-1})^{\rg (\Pic (U))} \frac{[\UUU_v]}{\LL^{n}} 
 \right) 
\times \prod_{v\in C_0 \setminus C_1}   \left  ( (1-\LL^{-1})^{\rg (\Pic (U)) }   \LL^{\rho^{\beta_v} - \langle \rho , \mathbf e^\beta \rangle} \frac{[ E_{\beta_v} ^\circ ]}{\LL^n }  \right ) . \\
%& \times \prod_{v\in C \setminus C_0} \left  ( (1-\LL^{-1})^{\rg (\Pic (X)) }  \LL^{\rho^\beta } \frac{[ E_\beta ]}{\LL^n }  \right ) .
\end{align*} 
The local term $ \LL^{\rho^{\beta_v} - \langle \rho , \mathbf e^\beta \rangle } [ E_{\beta_v} ^\circ ]\LL^{-n } $ (which is just $[\UUU_v]\LL^{-n} $ if $v\in C_1$) can be interpreted as the motivic integral 
\[
\int_{G(F_v , \beta_v )} \LL^{-(g, \LLL_{\rho ' } ) } | \omega_X | 
\]
over the adelic space $G(F_v , \beta_v ) = \{ g \in G ( F_v ) \mid ( g , E_{\beta_v } ) = 1 \} $. By \cite[Lemma 6.1.1]{chambert2016motivic} this integral can we rewritten as the motivic integral over arc spaces
\begin{align*}
\int_{\LLL ( \XXX_v , E_{\beta_v} ) } \LL^{-\ord_{\LLL_{\rho ' }} ( x  ) } \LL^{- \ord_{\omega_{X}} ( x  )  } & = \int_{\LLL ( \XXX_v , E_{\beta_v} ) }  \LL^{-\sum_{\alpha \in \AAA_U} \rho_\alpha \left ( \ord_v ( x_\alpha ) + e_\alpha^{\beta_v} \right )} \LL^{\rho^{\beta_v} + \sum_{\alpha \in \AAA_U} \rho_\alpha \ord_v ( x_\alpha )} \\
& = \int_{\LLL ( \XXX_v , E_{\beta_v} ) }  \LL^{\rho^{\beta_v} -  \langle \rho , \mathbf e^\beta \rangle } \\
& = \LL^{\rho^{\beta_v} -  \langle \rho , \mathbf e^\beta \rangle } \vol ( \LLL ( \XXX_v , E_{\beta_v} ) ) \\
& = \LL^{\rho^{\beta_v} -  \langle \rho , \mathbf e^\beta \rangle } \frac{[ E_{\beta_v} ^\circ ]}{\LL^n } .
\end{align*}
This comes from the fact that $\XXX ( \mathcal O_v ) = X ( F_v ) $ so that one can view $G( F) $ as a subset of $\XXX ( \mathcal O_v ) $ and any Schwartz-Bruhat function on $G(F_v)$ as a motivic function on $\LLL ( \XXX ) $, see \cite[\S 6.1]{chambert2016motivic}.
\end{myremark}

\subsubsection{Non-trivial characters} 

Given a place $v\in C$ and a non-trivial character $\xi \in G(F)^\vee$,
the linear form $x \in G ( F ) \mapsto \langle x , \xi \rangle $ 
can be seen as a rational function $f_{\xi}$ on $X$ whose divisor of poles has support contained in the union of the $D_\alpha $. We denote by $d_\alpha ( \xi) $ the order of the pole of $f_\xi$ with respect to $D_\alpha$
and define a subset of $\AAA_U$ by setting
\[
\AAA_U^0 (\xi ) = \{ \alpha \in \AAA_U \mid d_\alpha (\xi ) = 0 \} . 
\]
If $U\neq X$, this is automatically a proper subset of $\AAA$.
Otherwise if $U = X$, since $X$ is projective and $\xi$ is non-trivial, this is a proper subset of $\AAA$. 
\begin{mynotation}\label{notation:F_v(T,xi)}
For any place $v$ and any non-trivial character $\xi$, we write
\[
F_v ( \TT , \xi_v  ) = Z_v ( \TT , \xi_v ) \prod_{\alpha \in \AAA_U ^0 ( \xi ) } \left ( 1 - \LL^{\rho_\alpha -1 } T_\alpha \right ).
\]
\end{mynotation}
\begin{myptn}\label{proposition:convergence-residue-of-Zeta-xi}
The product 
\[
\prod_{v\in C_0}  F_v ( \TT  , \xi_v ) 
\]
is weight-linearly convergent with respect to $\rho$ at $T_\alpha = \LL^{-\rho_\alpha }$ and the resulting sum is a non-zero effective element of $\widehat{ \mathscr M_k}$.  
\end{myptn}
The proof of this proposition we give here consists of a summary of the arguments of the proof of Proposition 6.3.5.3 in Bilu's thesis  \cite{bilu2018motivic}, since our main interest lies in obtaining weight-linear convergence of $\prod_{v\in C_0} F_v ( \TT, \xi ) $ at $T_\alpha = \LL^{-\rho_\alpha}$. 
\begin{proof}[Proof of Proposition \ref{proposition:convergence-residue-of-Zeta-xi}]
Once again it is enough to check convergence on the dense subset $C_1 \subset C$.  In order to apply Proposition \ref{proposition:linear-conv-criterion}, one has to bound the weight of the coefficients of $F_v (\TT, \xi )$. We start by doing this for $Z_v ( \TT , \xi ) $.

First, the divisor of $f_\xi$ can be written 
 \[
 \ddiv ( f_\xi ) = \Xi_\xi - \sum_{\alpha \in \AAA } d_\alpha  (\xi ) D_\alpha 
 \]
 where $\Xi_\xi$ is the Zariski-closure of $\{ \langle x , \xi \rangle = 0 \}$ in $X$.  
 Over a place $v\in C_1$, one has $\BU_v = \{ \beta_v \}$ and the local factor $Z_v(\TT , \xi )$ can be written $Z_v (\TT , \xi ) = \sum_{A\subset \AAA_U} Z_{v,A} ( \TT , \xi ) $ with 
 \[
 Z_{v,A} ( \TT , \xi ) = \int_{\Omega_v ( A , \beta_v ) } \prod_{\alpha \in A} \left ( \LL^{\rho_\alpha } T_\alpha \right  ) ^{\ord x_\alpha } \ee ( f_\xi ( x ) ) \dd x 
 \]
 as we already pointed out in Remark \ref{remark:dependance-of-local-factors}. 
It is possible to compute or at least bound the weight of this integral. We refer to \cite[\S 6.3.5.2]{bilu2018motivic} for the details of such computations; in this proof we will restrict ourselves to giving a list of results we need to apply Proposition \ref{proposition:linear-conv-criterion}. In particular we obtain weight-linear convergence of the Euler product $\prod_{v\in C_1} F_v ( \TT , \xi )$. In what follows $\XXX_{\mathcal O_v}$ is written $\XXX$ for conciseness and the index $v$ may be dropped.  

If $A=\varnothing $ then the integral above equals one. 
If $A= \{ \alpha \}$ the intermediate step is to cut the integral into two pieces: one part corresponding to arcs with origin in the zero divisor $\Xi_\xi$ of $f_\xi$ and another part corresponding to arcs with origin outside $\Xi_\xi$. 
 \begin{align*}
 Z_{v,\{ \alpha \} } ( \TT, \xi ) &= \int_{ \mathscr L ( \XXX , D_\alpha^\circ \setminus \Xi_\xi )} \prod_{\alpha \in A} \left ( \LL^{\rho_\alpha } T_\alpha \right  ) ^{\ord x_\alpha } \ee ( f_\xi ( x ) ) \dd x \\
 & +  \int_{\mathscr L ( \XXX , D_\alpha^\circ  \cap \Xi_\xi ) } \prod_{\alpha \in A} \left ( \LL^{\rho_\alpha } T_\alpha \right  ) ^{\ord x_\alpha } \ee ( f_\xi ( x ) ) \dd x
 \end{align*}
Then there are several cases to distinguish, according to the order of the pole of $f_\xi$ at $D_\alpha$. On one hand, concerning arcs with origins outside $\Xi_\xi$, one has the following results. 
 \begin{itemize}
 \item If $d_\alpha (\xi ) = 0 $ then 
 \[
 \int_{\mathscr L ( \XXX , D_\alpha^\circ \setminus \Xi_\xi )} \left ( \LL^{\rho_\alpha } T_\alpha \right  ) ^{\ord x_\alpha } \ee ( f_\xi ( x ) ) \dd x = ( 1 - \LL^{-1} ) \frac{\LL^{\rho_\alpha - 1 } T_\alpha }{1 - \LL^{\rho_\alpha - 1 } T_\alpha } [D_\alpha^\circ \setminus \Xi_\xi ] \LL^{-n+1} .
 \]
 \item If $d_\alpha (\xi ) =1 $ then 
  \[
 \int_{\mathscr L ( \XXX , D_\alpha^\circ \setminus \Xi_\xi ) } \left ( \LL^{\rho_\alpha } T_\alpha \right  ) ^{\ord x_\alpha } \ee ( f_\xi ( x ) ) \dd x = - \LL^{-2} [D^\circ_\alpha  \setminus \Xi_\xi  ]\LL^{1-n} \LL^{\rho_\alpha} T_\alpha .
 \]
 Since $\dim ( D_\alpha^\circ \setminus \Xi_\xi ) = n-1$, one has the following upper bound on the weight
 \[
 w_{C_1} \left ( - \LL^2 [D^\circ_\alpha  \setminus \Xi_\xi  ]\LL^{1-n} \LL^{\rho_\alpha} \right ) \leqslant 2 (  \rho_\alpha - 2  ) + 1 \leqslant 2  \rho_\alpha  -1 .
 \]
 \item If $d_\alpha (\xi ) > 1 $ then this integral equals zero. 
 \end{itemize}
On the other hand, concerning arcs with origin in $\Xi_\xi$,  the corresponding integral can be rewritten 
\begin{align*}
	& \int_{\mathscr L ( \XXX , D_\alpha^\circ  \cap \Xi_\xi ) } \left ( \LL^{\rho_\alpha } T_\alpha \right  ) ^{\ord x_\alpha } \ee ( f_\xi ( x ) ) \dd x  \\
& = 
\sum_{m\geqslant 1} \left ( \LL^{\rho_\alpha } T_\alpha \right  ) ^{m} \int_{ \substack{ \left ( D_\alpha^\circ \cap \Xi_\xi \right ) \times  \mathscr L ( \mathbf A^1 ) \times \mathscr L ( \mathbf A^1 , 0 )^{n-1} \\ \ord x = m }} \ee ( f_\xi ( t^m x , \mathbf y ) ) \dd x \dd \mathbf y .
\end{align*}
By (\ref{inequ-weight-integral-volume}), the weight of the coefficient of order $m$ is smaller than the weight of the motivic volume 
\begin{align*}
	& \vol \left  ( \{ ( w , x , \mathbf y ) \in \left ( D_\alpha^\circ \cap \Xi_\xi \right ) \times \mathscr L ( \mathbf A^1 ) \times \mathscr L ( \mathbf A^1 , 0  ) ^{n-1} \mid \ord x = m \} \right ) \\
& =
\left [ D_\alpha^\circ \cap \Xi_\xi  \right ] \LL^{-m} \left ( 1 - \LL^{-1} \right ) \LL^{-n+1}.
\end{align*}
One has the upper bound on dimensions 
\[
\dim \left ( \LL^{\rho_\alpha m} [D_\alpha^\circ \cap \Xi_\xi  ] \LL^{-m} \left ( 1 - \LL^{-1} \right ) \LL^{-n+1} \right ) \leqslant m\rho_\alpha + ( n -2 ) - m - ( n -1  ) = m ( \rho_\alpha   -  1  ) - 1 
\]
and thus on weights 
\[
w_{C_1} \left ( \LL^{\rho_\alpha m} [D_\alpha^\circ \cap \Xi_\xi  ] \LL^{-m} \left ( 1 - \LL^{-1} \right ) \LL^{-n+1} \right ) \leqslant 2 ( m ( \rho_\alpha - 1 )  - 1 ) + \dim ( C_1 ) .
\]
Setting $c = \underset{\alpha \in \AAA_U}{\max}\left  ( 1 - \frac{1}{2\rho_\alpha } \right )$ as in \cite{bilu2018motivic}, one gets $2\rho_\alpha - 1 \leqslant 2 \rho_\alpha c $ for every $\alpha \in \AAA_U$ and  
\[
w_{C_1} \left (\LL^{\rho_\alpha m} [D_\alpha^\circ \cap \Xi_\xi  ] \LL^{-m} \left ( 1 - \LL^{-1} \right ) \LL^{-n+1} \right ) \leqslant 2 c m \rho_\alpha - 1 . 
\]
Now if $|A | \geqslant 2 $ the idea is similar. The corresponding local term is 
\[
Z_{v,A} = \sum_{\mm \in \NN^A_{>0} } \prod_{\alpha \in A} \left ( \LL^{\rho_\alpha} T_\alpha \right )^{m_\alpha} \int_{\ord x_\alpha =  m_\alpha } \ee ( f_\xi ( x , y ) ) \dd x \dd y .
\]
The volume of the constructible subsets of $\Omega_v ( A)$ over which integration is done is actually
\begin{align*}
& \vol \left ( \left \{  (w, ( x_\alpha )_{\alpha \in A} , y ) \in D_A^\circ \times \mathscr L ( \mathbf A^1 )^{A} \times \mathscr L ( \mathbf A^1 , 0 ) ^{n-|A|} \mid \ord x_\alpha = m_\alpha \text{ for every } \alpha \in A \right \} \right ) \\
& = [D_A^\circ ] \prod_{\alpha \in A} \left ( \LL^{-m_\alpha} \left ( 1  - \LL^{-1} \right ) \right ) \LL^{-n+|A|} .
\end{align*}
By definition of $D_A^\circ$, it has dimension at most $n-|A|$. Thus the dimension of the integral is at most $- \sum_{\alpha \in A} m_\alpha $
and, using again the inequality of \S \ref{inequ-weight-integral-volume}, the weight of the $\mm$-th coefficient of $Z_{v,A} ( \TT , \xi ) $ is bounded by
\[
2 \sum_{\alpha \in A } m_\alpha ( \rho_\alpha -1 ) + \dim C_1 \leqslant 2 c \sum_{\alpha \in A} m_\alpha \rho_\alpha  - 1 .
\]

To summarize, we obtained 
\begin{align*}
Z_v ( \TT , \xi ) & = 1 + \sum_{ \substack{ \alpha \in  \AAA_U \\ d_\alpha ( \xi ) \leqslant 1  } } Z_{v,\alpha} ( \TT, \xi ) + \sum_{|A|\geqslant 2} Z_{v,A} ( \TT, \xi )\\
Z_v ( \TT , \xi ) &=  1 + \LL^{-n} \left ( 1 - \LL^{-1}\right ) \sum_{\alpha \in \AAA_U^0 (\xi ) }  [ D^\circ _\alpha \setminus \Xi_\xi ] \frac{\LL^{\rho_\alpha} T_\alpha }{1- \LL^{\rho_\alpha -1} T_\alpha }  - \LL^{-n-1} \sum_{\substack{ \alpha \in  \AAA_U ( \xi ) \\ d_\alpha ( \xi ) = 1  }}   [ D^\circ _\alpha \setminus \Xi_\xi ] \LL^{\rho_\alpha } T_\alpha \\
&+ \text{ terms of weights bounded as in Proposition \ref{proposition:linear-conv-criterion}, locally uniformly in $\xi$}\\
% & = 1 + \LL^{-n}  \sum_{\alpha \in \AAA_U^0 (\xi ) } [ D^\circ _\alpha \setminus E ] \frac{\LL^{\rho_\alpha} T_\alpha }{1- \LL^{\rho_\alpha} T_\alpha } + \text{ terms bounded as in Proposition \ref{proposition:linear-conv-criterion}.}
\end{align*}

The last step of the proof consists in multiplying by $\prod_{\alpha \in \AAA_U^0 (\xi )} \left ( 1 - \LL^{\rho_\alpha -1 }T_\alpha \right )$. This operation does not affect the bounds we need in order to apply Proposition \ref{proposition:linear-conv-criterion}, since $\rho_\alpha - 1 \leqslant c\rho_\alpha$,
and allows one to control the diverging term \[
\sum_{\alpha \in \AAA_U^0 (\xi ) }  [ D^\circ _\alpha \setminus \Xi_\xi ] \frac{\LL^{\rho_\alpha} T_\alpha }{1- \LL^{\rho_\alpha -1} T_\alpha } 
\]
in the expression of $Z_v ( \TT , \xi )$. 
One gets 
\begin{align*}
F_v ( \TT , \xi ) & = 1 - \sum_{\alpha \in \AAA_U^0 (\xi ) } \LL^{\rho_\alpha -1} T_\alpha + \LL^{1-n} \sum_{\alpha \in \AAA_U^0 (\xi ) } [ D^\circ _\alpha \setminus \Xi_\xi ] \LL^{\rho_\alpha -1 }  T_\alpha + P_v ( \TT, \xi ) \\
& = 1 + \sum_{\alpha \in \AAA_U^0 (\xi ) } \left ( [ D_\alpha^\circ \setminus \Xi_\xi ] -  \LL^{n-1} \right )\LL^{1-n} \LL^{\rho_\alpha -1 }  T_\alpha + P_v ( \TT, \xi )
\end{align*}
where $P_v ( \TT, \xi ) $ is a Laurent series consisting of terms satisfying the bounds of Proposition \ref{proposition:linear-conv-criterion} for 
\[ 
c = \underset{\alpha \in \AAA_U}{\max}\left  ( 1 - \frac{1}{2\rho_\alpha } \right ).
\] 
Finally, by Proposition \ref{proposition:properties-of-the-weight} one has for every $\alpha \in \AAA_U^0 ( \xi)$
\[
w_{C_1} \left ( \left ( [ D_\alpha^\circ \setminus \Xi_\xi ] -  \LL^{n-1} \right )\LL^{1-n} \LL^{\rho_\alpha -1 }   \right )  \leqslant 2 ( n -1 ) - 2 ( n - 1 ) + 2 ( \rho_\alpha - 1 ) = 2 ( \rho_\alpha - 1 ) .
\]
Our analysis shows that we can apply Proposition \ref{proposition:linear-conv-criterion} with $X=C_1$, $\varepsilon = \frac{1}{2}$, $M$ arbitrary and $\beta = 0$. The result follows. 
\end{proof}

\begin{myremark}\label{remark:uniform-convergence-finite-stratum}
The previous proof shows that the weight-linear convergence of $\prod_{v\in C_1} F_ v ( \TT , \xi )$ with respect to $\rho$ is uniform on each set of a finite partition of $V\setminus \{ 0\}$ given by 
\begin{align*}
V_{A_0^D , A_1^D , A_{\geqslant 2}^D } & = \left ( \bigcap_{ \alpha \in A_0^D } d_\alpha^{-1} ( \{0 \} ) \right ) \cap \left ( \bigcap_{ \alpha \in A_1^D } d_\alpha^{-1} ( \{ 1 \} ) \right ) \cap \left ( \bigcap_{ \alpha \in A_{\geqslant 2}^D } d_\alpha^{-1} ( \NN_{\geqslant 2} ) \right ) \\
A_0^D \sqcup A_1^D \sqcup A_{\geqslant 2}^D & = \AAA_U 
\end{align*}
(recall that $V$ is the $n$-th power of the Riemann-Roch space of the divisor $\tilde{E}$ (\ref{def:divisor-E-tilde})).
\end{myremark}

%%%%%%%%%%%%
%%% SECTION %%% 
%%%%%%%%%%%%

\section{Moduli spaces of curves: asymptotic behaviour}

%%% SUBSECTION %%%
\subsection{Simplified case: rational curves}

For the sake of simplicity we start the proof of our result assuming that $C$ is the projective line over $k$.
For any $\mathbf a $ and $ \mathbf b$  in $\mathbf Z^{\AAA} $ we write $\mathbf a\leqslant \mathbf b $ if and only if $\mathbf b - \mathbf a$ lies in $ \NN^\AAA$. 

Recall that by \cite[Theorem 2.7]{tschinkel1999geometry} there exists an $r$-tuple of integers $(\rho_\alpha)_{\alpha\in \AAA}$ such that $\rho_\alpha \geqslant 2$ for all index $\alpha$ and 
\[
\sum_{\alpha \in \AAA} \rho_\alpha D_\alpha 
\] 
is an anticanonical divisor on $X = \XXX_F$. 
A log-divisor with respect to the boundary $D$ is then given by 
\[
\sum_{\alpha \in \AAA} \rho_\alpha ' D_\alpha 
\]
where $\rho_\alpha ' = \rho_\alpha - 1 $ if $\alpha \in \AAA_D $ and $\rho_\alpha ' = \rho_\alpha$ otherwise. 
As it has already been the case in the introduction, for any $r$-tuple $\mm = (m_\alpha )_{\alpha\in\AAA} $ of %non-negative 
integers, we will use freely the pairing
\[
\langle \rho  ' , \mm \rangle = \sum_{\alpha \in \AAA} \rho_\alpha '  m_\alpha 
\]
as well as its obvious restriction $\langle \rho ' , \mm \rangle_{A}$ to any subset $A $ of $\AAA$. 
Recall that this quantity corresponds to a generic log-anticanonical degree.

\subsubsection{Main term}

First we study the contribution of the trivial character $\xi = 0$, which is expected to be the only one to contribute asymptotically after normalization by $\LL^{\langle \rho ' , \mm \rangle}$. 
For sake of simplicity we begin assuming $\UUU = \XXX$ 
so that $\AAA_U = \AAA$ and the coarse and refined height zeta functions coincide. 
% Furthermore, in this paragraph we assume that $e^\beta_\alpha = 0 $ for every $\alpha $ and $\beta$ (see \ref{subsection:back-to-our-setting} for the definition of $e_{\alpha}^\beta $). 

For every place $v $ of $ \PP^1_k$ recall that we already defined a polynomial 
\[
F_v ( \TT , 0 ) = Z_v ( \TT , 0 ) \prod_{\alpha \in \AAA} \left ( 1 - \LL ^{\rho_\alpha -1 } T_\alpha \right ) 
\] 
with Notation \ref{notation:F_v(T,0)}.
Thanks to Proposition \ref{proposition:euler-product:multiplicativity} 
we are able to permute the products $\prod_{\alpha \in \AAA} $ and $\prod_{v\in \PP^1_k}$ 
when considering $Z ( \TT , 0 )$. 
It means that 
\[
Z ( \TT , 0 ) = 
\prod_{v\in \PP^1_k} Z_v ( \TT , 0 ) = \left ( \prod_{v\in \PP^1_k} F_v ( \TT , 0) \right ) \left (  \prod_{\alpha \in \AAA} \prod_{v\in \PP^1_k}  \left ( 1 - \LL ^{\rho_\alpha -1 } T_\alpha \right )^{-1}   \right ) . 
\]
By definition, the Kapranov Zeta function of $\PP^1_k$ is the power series
\[
Z_{\PP^1_k} ( T )  = \prod_{v\in \PP^1_k} ( 1 - T )^{-1} 
\]
which is given explicitly (see e.g. \cite[Theorem 1.1.9]{kapranov2000elliptic} or \cite[Chapter 7, Theorem 1.3.1]{chambert2018motivic}) by 
\[
Z_{\PP^1_k} ( T )   = \frac{1}{(1-T)(1-\LL T)} = \frac{1}{1-\LL} \left ( \frac{1}{1-T} - \frac{\LL}{1-\LL T} \right ). 
\]
In our case, this gives for every index $\alpha$ 
\[
\prod_{v\in \PP^1_k} \left ( 1-\LL^{\rho_\alpha -1} T_\alpha \right )^{-1}  
= Z_{\PP^1_k} ( \LL^{\rho_\alpha -1} T_\alpha  )
= \frac{1}{1-\LL}  \left ( \frac{1}{1-\LL^{\rho_\alpha -1} T_\alpha} - \frac{\LL}{1-\LL^{\rho_\alpha} T_\alpha } \right ) .
\]
Expanding the right side of this equality, 
one gets that the coefficient of $T_\alpha^m$ 
for any non-negative integer $m$ 
is given by
\[
\frac{ \LL^{\rho_\alpha m} ( \LL^{-m} - \LL )}{1-\LL} = \frac{1}{1-\LL^{-1}} \left ( 1 - \LL^{-m -1}  \right ) \LL^{\rho_\alpha m}  . 
\]
Thus, loosely speaking, after normalisation by $\LL^{\rho_\alpha m}$, 
the contribution of $(1-\LL^{\rho_\alpha -1} T_\alpha )^{-1}$ in this last formula tends to zero when $m$ tends to infinity.  
In order to find the expected limit of $\LL^{-\langle \rho , \mm \rangle}  [M_{U,\mm}] $ 
when $\min_{\alpha\in\AAA} (m_\alpha )$ tends to infinity, 
it may be natural to consider for a while the coefficients of the series 
\[
\overline{Z} ( \TT )  = 
\left ( \prod_{\alpha\in \AAA} \left ( \left   (1-\LL^{-1} \right ) \left ( 1 - \LL^{\rho_\alpha} T_\alpha \right ) \right ) ^{-1}  \right ) \left ( 
\prod_{v\in \PP^1_k} F_v ( \TT , 0 ) \right ) 
\]
instead of considering those of $Z(\TT , 0 )$. 

For any $r$-tuple $\mathbf{m} = (m_\alpha)_{\alpha \in \AAA} $ of integers, 
let us denote by $\overline{ \mathfrak a} _{\mathbf{m}}$ 
the coefficient of the monomial $\TT^\mathbf{m}$ in $\overline{Z}(\TT)$. 
Denoting by $\mathfrak b_{\nn}$ the coefficient of $\TT^{\nn}$ in 
$\prod_{v\in \PP^1_k} F_v ( \TT , 0 )$ 
for any $r$-tuple $\nn$ of integers, 
one gets the finite sum 
\[
\overline{ \mathfrak a}_{\mathbf{m}}  = 
\frac{1}{\left (1- \mathbf L^{-1} \right  ) ^r}
 \sum_{\substack{\nn \in \ZZ^\AAA  \\ \nn ' \in  \NN^\AAA \\   \nn + \nn ' =  \mm }}
\mathfrak b_{\nn} \times  \LL^{ \langle  \rho , \nn' \rangle }.
\]
After normalisation by $\LL^{\langle \rho , \mm \rangle }$, 
this is actually the $\mm$-th partial sum of the Euler motivic product 
$\prod_{v\in \PP^1_k} F_v \left ( \left ( \LL^{-\rho_\alpha }  \right )_{\alpha \in \AAA} , 0 \right )$. 
Therefore  by Proposition \ref{proposition:convergence-residue-of-Zeta-0}
\[
\overline{\mathfrak a}_{\mathbf{m}} \LL^{- \langle \rho , \mm \rangle }  
\]
converges in $\widehat{\mathscr M_k} $ to 
\[
\frac{1}{\left  (1-\LL^{-1} \right )^r } 
\prod_{v\in \PP^1_k} F_v \left ( \left ( \LL^{-\rho_\alpha }  \right )_{\alpha \in \AAA} , 0 \right )
\]
when $\min_{\alpha \in \AAA} (m_\alpha )$ tends to infinity.

%which is exactly the expected constant up to a factor $\LL^n$. %\textcolor{blue}{[Details to be given.]}

This heuristic argument given, now we have to justify that this is indeed the limit of 
$\LL^{-\langle \rho , \mm \rangle}  [M_{U,\mm}] $ (up to a factor $\LL^n $) 
in $\widehat{\mathscr M_k}$  
when $\min_{\alpha\in \AAA} (m_\alpha )$ 
becomes infinitely large. 
First we evaluate the error term we introduced,
then we check that terms corresponding to non-trivial characters do not constribute to the limit. 
We postpone the summation over all characters to the last section of this paper, 
where it will be performed in full generality. 

\subsubsection{Contribution of the error term}\label{paragraph:Manin-Peyre-projective-line:trivial-characters}
In order to control the error term we implicitly introduced in the previous paragraph, we develop the denominator of $Z(\TT, 0)$ as follows. We still assume $\UUU = \XXX$. Then
\begin{align*}
G ( \TT , 0 ) & = \prod_{\alpha \in \AAA} Z_{\PP^1_k} ( \LL^{\rho_\alpha -1} T_\alpha  ) \\
%%%%%%%%% 2ND LINE %%%%%%%%%
& = \prod_{\alpha \in \AAA} \prod_{v\in \PP^1_k} \left ( 1-\LL^{\rho_\alpha -1} T_\alpha \right )^{-1} \\
%%%%%%%%% 3RD LINE %%%%%%%%%
& = \prod_{\alpha \in \AAA} \frac{1}{1-\LL}  \left ( \frac{1}{1-\LL^{\rho_\alpha -1} T_\alpha } - \frac{\LL}{1-\LL^{\rho_\alpha} T_\alpha } \right ) \\
%%%%%%%%% 4TH LINE %%%%%%%%%
& = \frac{1}{(1-\LL)^r} \sum_{\varepsilon \in \{0,1\}^{\AAA}} \frac{(-\LL)^{| \AAA | - |\varepsilon |  }}{\prod_{\alpha\in \AAA} \left ( 1-\LL^{\rho_\alpha - \varepsilon_\alpha } T_\alpha  \right) } .
\end{align*}
where $|\varepsilon | = \sum_\alpha \varepsilon_\alpha$. 
Let us introduce a few convenient notations. 
For any $\varepsilon \in \{0,1\} ^{\AAA} $ we define
\[
G_\varepsilon ( \mathbf T) =  \prod_{\alpha\in \AAA}  ( 1-\LL^{\rho_\alpha - \varepsilon_\alpha } T_\alpha  ) ^{-1} 
= \sum_{\mm \in \NN^\AAA } \LL^{\langle \rho - \varepsilon , \mm \rangle } \TT^{\mm} 
\] 
so that $G(\TT , 0 )$ admits a decomposition of the form
\[
G(\mathbf T , 0) = \sum_{\varepsilon \in \{0,1\}^{\AAA} }  \frac{(-\LL)^{r - |\varepsilon |}}{(1-\LL)^r } G_\varepsilon ( \mathbf T ) .
\]   
Now one easily sees that the term of multidegree $\mm \in \mathbf Z^\AAA$ of 
$G _\varepsilon (\mathbf T ) \prod_{v\in \mathbf P^1_k} F_v ( \mathbf T , 0 ) $ 
is the finite sum 
\[
\mathfrak g_{\mm} ^\varepsilon = 
 \sum_{\substack{\nn \in  \ZZ^\AAA  \\ \nn ' \in  \NN^\AAA \\   \nn + \nn ' =  \mm }} \mathfrak b_{\nn } \LL^{\langle \rho - \varepsilon  , \nn ' \rangle }
\]
from which we deduce 
%\[
%\mathfrak e_{\mm} ^\varepsilon \LL^{- \langle  \rho - \varepsilon , \mm \rangle } 
% = \sum_{\substack{\mm ' \in \NN^\AAA  \\  \mm' \leqslant   \mm  }} \mathfrak b_{\mm ' } \LL^{-\langle \rho - \varepsilon , \mm ' \rangle}  
%\]
%and 
\[
\mathfrak g_{\mm} ^\varepsilon \LL^{- \langle  \rho, \mm \rangle } 
 = \sum_{\substack{\nn \in   \ZZ^\AAA  \\ \nn ' \in  \NN^\AAA \\   \nn + \nn ' =  \mm }} \mathfrak b_{\nn } \LL^{-\langle \rho , \nn \rangle}  \LL^{ -\langle \varepsilon   , \nn ' \rangle   } . 
\]
%The first normalisation of $\mathfrak e_{\mm} ^\varepsilon$  by $ \LL^{- \langle  \rho - \varepsilon , \mm \rangle } $ might be seen as the $\mm$-th partial sum of 
%$\prod_{v\in \PP^1_k} F_v (\TT , 0 ) $ specialised at $\TT = \left  ( \LL^{-(\rho_\alpha - \varepsilon_\alpha  )} \right )_{\alpha \in \AAA}$, but we do not know \textcolor{blue}{yet} if this last Euler product converges at this point. 
%The second normalisation by $\LL^{- \langle  \rho, \mm \rangle } $ will provide our result. 
%Then the coefficient of order $\mm$ of $\prod_{v\in \mathbf P^1_k} F_v ( \mathbf T , 0 ) E (\mathbf T )$ 
%is given by $\mathfrak e_{\mm} = \sum_{\varepsilon \neq \{ 1, ... ,\} }  and 
%\[
%\mathfrak e_{\mm} \LL^{-\sum_{\alpha \in \AAA} \rho_\alpha m_\alpha } 
 %= \sum_{\varepsilon \neq \{ 1 , ... , 1\} } \frac{(-\LL)^{\sum_\alpha \varepsilon_\alpha }}{(1-\LL)^r } \LL^{-\sum_{\alpha \in \AAA} m_\alpha ( 1 - \varepsilon_\alpha ) }  \sum_{\substack{\mm ' \in \NN^\AAA  \\  \mm ' \leqslant \mm' }} \mathfrak b_{\mm ' } \LL^{\sum_{\alpha \in \AAA} - m_\alpha ' (\rho_\alpha + \varepsilon_\alpha - 1 )  }
%\]
The case $\varepsilon =\mathbf 0 $ has been studied in the previous paragraph and it suffices to apply Lemma \ref{lemma:negligible-terms} to $\sum \mathfrak b_{\nn} \LL^{-\langle \rho , \nn \rangle }$ in case $\varepsilon \neq \mathbf 0$. 
We obtain the following proposition. 
\begin{myptn}[$\UUU = \XXX$, $C=C_0=\PP^1_k$, $\xi = 0$]
There exists a decomposition 
\[
Z ( \TT , 0 ) =
 \sum_{\mm \in \mathbf Z^{\AAA} }  \mathfrak a _{\mm} \TT^{\mm} 
=  \
\sum_{\varepsilon \in \{ 0 , 1 \}^{\AAA} } \sum_{\mm \in \mathbf  Z^{\AAA} } \mathfrak a _{\mm} ^\varepsilon  
 \TT^{\mm} 
\]
and 
a real number $\delta >0 $
such that for all $\varepsilon \in \{ 0 , 1\}^{\AAA } \setminus \{ \mathbf 0 \} $
%\[
%\mathfrak a_{\mm} ^\varepsilon \LL^{- \langle \rho - \varepsilon , \mm \rangle }
 %\longrightarrow 
 %\frac{(-\LL)^{-| \varepsilon |}}{(1-\LL^{-1})^{| \AAA |} }  \prod_{v\in \PP^1_k} F_v \left ( ( \LL ^{ - ( \rho_\alpha - \varepsilon_\alpha )  } )_{\alpha_\in \AAA} , 0 \right )   
%\]
\[
w \left ( \mathfrak a_{\mm} ^\varepsilon \LL^{- \langle \rho  , \mm \rangle } \right )  < - \delta \langle \varepsilon , \mm \rangle
\]
while 
\[
 \mathfrak a _{\mm}^{\mathbf 0} \LL^{- \langle \rho , \mm \rangle }
\underset{\rho\text{-weight-lin.}}{ \longrightarrow }
 \frac{1}{(1-\LL^{-1}  )^{\rg (\Pic (X)) }}
\prod_{v\in \PP^1_k} F_v \left ( ( \LL^{  - \rho_\alpha    } )_{\alpha_\in \AAA} , 0 \right )
\]
when $ \min_{\alpha \in \AAA} ( m_\alpha  ) $ tends to infinity.

%Furthermore, in case $ \prod_{v\in \PP^1_k} F_v \left ( ( \LL^{  - \rho_\alpha + \varepsilon_\alpha   } )_{\alpha_\in \AAA} , 0 \right )$ converges, then 
%\[
%\mathfrak a_{\mm} ^\varepsilon \LL^{- \langle \rho - \varepsilon , \mm \rangle } \longrightarrow  \frac{(-1)^{|\varepsilon |}}{\LL^{|\varepsilon |} ( 1 - \LL^{-1} )^{\rg (\Pic (X))}  }
%\prod_{v\in \PP^1_k} F_v \left ( ( \LL^{  - \rho_\alpha + \varepsilon_\alpha   } )_{\alpha_\in \AAA} , 0 \right )
%\]
%when $ \min_{\alpha \in \AAA} ( m_\alpha  )  $ becomes arbitrary large.  
\end{myptn}
As an immediate corollary of this proposition, 
the terms $\mathfrak a_{\mm} ^\varepsilon \LL^{- \langle \rho  , \mm \rangle }$, for $\varepsilon \neq \mathbf 0$, are negligible when $ \min_{\alpha \in \AAA} ( m_\alpha  ) $ becomes arbitrarily large, 
in comparison with $\mathfrak a _{\mm}^{\mathbf 0} \LL^{- \langle \rho , \mm \rangle }$. 

\subsubsection{Non-trivial characters} \label{paragraph:Manin-Peyre-projective-line:non-trivial-characters}
We now study the asymptotic contribution of the coefficients of $Z(\TT, \xi ) $ for $\xi \in G ( F ) ^\vee$ non-trivial. We still assume $C=\PP^1_k $.

Recall that given a place $v$ of $C$, 
the linear form $x\mapsto \langle x , \xi \rangle $ on $G_F$
can be seen as a rational function $f_{\xi}$ on $X$ with poles contained in $\cup_{\alpha \in \AAA} D_\alpha$. If $d_\alpha ( \xi) $ denotes the order of the pole of $f_{\xi}$ with respect to $D_\alpha$, one can define a subset of $\AAA_U $ by setting
\[
\AAA_U^0 (\xi ) = \{ \alpha \in \AAA_U \mid d_\alpha (\xi ) = 0 \} . 
\]
The character $\xi $ being non-trivial, the cardinality of this set is strictly smaller than $|\AAA |$. 
Using Notation \ref{notation:F_v(T,xi)}, 
the term of the motivic height Zeta function corresponding to the character $\xi$ is 
\[
Z ( \TT , \xi ) = \prod_{v\in \PP^1_k } \left ( F_v ( \TT , \xi_v ) \prod_{\alpha \in \AAA_U^0( \xi ) } \left ( 1 - \LL^{\rho_\alpha -1 } T_\alpha \right )^{-1} \right ) . 
\]
Again we restrict our analysis to the case $\UUU = \XXX$ so that $\AAA = \AAA_U $. The general case will be treated together with the case of a general curve. 
For any $\varepsilon \in \{ 0, 1 \}^{\AAA_U^0 ( \xi )}$
we introduce 
\[
G_\varepsilon  ( \TT , \xi ) = \prod_{\alpha \in \AAA_U^0 ( \xi )}    \left(1-\LL^{ \rho_\alpha - \varepsilon_\alpha } T_\alpha \right )^{-1}  . 
\]
Adapting the computation done in the previous paragraph, one gets
\[
G  ( \TT , \xi )  = 
\prod_{v\in \PP^1 _k }  \left ( \prod_{\alpha \in \AAA_U^0(\xi ) } ( 1 - \LL^{\rho_\alpha - 1 } T_\alpha )^{-1} \right ) 
= 
\sum_{\varepsilon \in \{ 0, 1 \}^{\AAA_U^0( \xi )} } \frac{(-\LL)^{|\AAA_U^0( \xi ) | - |\varepsilon | }}{(1-\LL)^{\left |\AAA_U^0(\xi) \right |} } G_\varepsilon  ( \TT , \xi ) 
\]
and the coefficient of multidegree $\mm \in \NN^\AAA$ of the product 
\[
G_\varepsilon  ( \TT , \xi ) \prod_{v\in \PP^1_k} F_v ( \TT, \xi_v ) 
\]
is 
\[
\mathfrak g_{\varepsilon , \mm} ^{\xi } =  \sum_{\substack{\nn \in   \ZZ^\AAA  \\ \nn ' \in  \NN^\AAA \\   \nn + \nn ' =  \mm }} \mathfrak b^\xi_{\nn} \LL^{ \langle \rho - \varepsilon  , \nn ' \rangle_{\AAA_U^0 ( \xi )} }   
\]
where $\mathfrak b_{\nn} ^\xi $ is the coefficient of multidegree $\nn \in \ZZZ^\AAA $ of $\prod_{v\in \PP^1_k } F_v ( \TT , \xi)$. 
%Then we may consider the two following normalisations.
Then we consider the following normalisation:
%\begin{align*}
%\mathfrak g_{\varepsilon , \mm} ^{\xi } \LL^{-\langle \rho -\varepsilon , \mm \rangle}  & = 
% \sum_{\substack{\mm ' \in \NN^\AAA  \\  \mm ' \leqslant \mm }} \mathfrak b^\xi_{\mm' } \LL^{ - \langle \rho - \varepsilon  , \mm' \rangle  }   \LL^{ - \langle \rho  , \mm - \mm' \rangle_{^\complement \AAA_0 ( \xi )} } \\
% \mathfrak g_{\varepsilon , \mm} ^{\xi } \LL^{-\langle \rho  , \mm \rangle}  & = 
% \sum_{\substack{\mm ' \in \NN^\AAA  \\  \mm ' \leqslant \mm }} \mathfrak b^\xi_{\mm' } \LL^{ - \langle \rho   , \mm' \rangle  }   \LL^{ - \langle \rho , \mm - \mm' \rangle_{^\complement \AAA_0 ( \xi )} - \langle \varepsilon  , \mm - \mm' \rangle_{ \AAA_0 ( \xi )}}
%\end{align*}
\[
\mathfrak g_{\varepsilon , \mm} ^{\xi } \LL^{-\langle \rho  , \mm \rangle}   = 
\sum_{\substack{\nn \in   \ZZ^\AAA  \\ \nn ' \in  \NN^\AAA \\   \nn + \nn ' =  \mm }} \mathfrak b^\xi_{\nn } \LL^{ - \langle \rho   , \nn \rangle  }   \LL^{ - \langle \rho , \nn '  \rangle_{\AAA\setminus \AAA_U^0 ( \xi )} - \langle \varepsilon  , \nn ' \rangle_{ \AAA_U^0 ( \xi )}} . 
\]
%Assuming convergence of  $\prod_{v\in \PP^1_k} F_v ( \TT , \xi_v )$ at $( \LL^{- \rho_\alpha + \varepsilon }  )_{\alpha \in \AAA}$, the first normalisation, combined with Lemma \ref{lemma:negligible-terms}, gives us that
%$ \mathfrak g_{\varepsilon , \mm} ^{\xi } \LL^{-\langle \rho -\varepsilon , \mm \rangle} $
%tends to zero when $\min_{\alpha \in \AAA} ( m_\alpha ) $ becomes large. 
%The second normalisation, again with use of Lemma \ref{lemma:negligible-terms}, shows that the contribution of the character $\xi$ is indeed negligible : the term of multi-degree $\mm$ of $Z( \TT , \xi )$, that is, of the product $G_\varepsilon  ( \TT , \xi ) \prod_{v\in \PP^1_k} F_v ( \TT, \xi_v ) $, normalised by $ \LL^{-\langle \rho  , \mm \rangle} $, tends to zero when $\min_{\alpha \in \AAA} ( m_\alpha ) $ tends to infinity. 
Applying Lemma \ref{lemma:negligible-terms} to this sum for every $\varepsilon$, we get the following proposition. 
\begin{myptn}[$\UUU = \XXX$, $C=C_0=\PP^1_k$, $\xi \neq 0$]
There exists a decomposition 
\[
Z ( \TT , \xi ) = \sum_{\varepsilon \in \{ 0 , 1 \}^{\AAA_U^0 ( \xi ) }} \sum_{\mm \in \mathbf Z^\AAA} \mathfrak a_\mm^{\xi , \varepsilon} \TT^\mm
\]
and a positive real number $\delta_\xi > 0 $ (which only depends on the $d_\alpha ( \xi )$) such that for every $\varepsilon \in \{ 0 , 1 \}^{\AAA_0 ( \xi ) }$ 
\[
w \left  ( \mathfrak a_\mm^{\xi , \varepsilon} \LL^{-\langle \rho , \mm \rangle } \right ) < - \delta_\xi \left ( \langle \rho , \mm  \rangle_{\AAA_U \setminus \AAA_U^0 ( \xi )} + \langle \varepsilon  , \mm  \rangle_{ \AAA_U^0 ( \xi )}\right )
\] 
when $\min ( m_\alpha )\to \infty$. 
\end{myptn}
Since $\AAA_U \setminus \AAA_U^0 ( \xi ) $ is non-empty for every $\xi \neq 0$, the inequality of this last proposition means that the normalised terms $\mathfrak a_\mm^{\xi , \varepsilon} \LL^{-\langle \rho , \mm \rangle }$ coming form $Z ( \TT , \xi )$ are negligible in comparison with the main term $\mathfrak a_\mm^{\mathbf 0 } \LL^{-\langle \rho , \mm \rangle } $ coming from $Z ( \TT , 0 ) $. 
Furthermore we have a linear bound on the weight which is uniform on a finite partition of the space of characters
by Remark \ref{remark:uniform-convergence-finite-stratum}. 
%%% SUBSECTION %%%

\subsection{The case of a general curve} \label{paragraph:case-general-curve}
From now on we assume that $C$ is a projective smooth irreducible curve of genus $g$ over $k$ having a $k$-rational point. 
%The case $g=0$ is the previously treated case of $\PP^1_k$ so we may assume $g>0$. 
In this case, for $m \geqslant 2g-1$ the class of $\Sym ^m ( C) $ in $\KVar _k$ is $[ \Pic ^0 ( C) ][\PP^{m-g}_k ]$ 
(it follows from Riemann-Roch and Serre's duality, see \cite[Chapter 7, Example 1.1.10]{chambert2018motivic}) 
and the Kapranov zeta function of $C$ is still rational (see \cite[Theorem 1.1.9]{kapranov2000elliptic} or \cite[Chapter 7, Theorem 1.3.1]{chambert2018motivic}). We have 
\[
Z_C^\text{Kap} ( T ) = \sum_{m\in \NN} [\Sym ^m ( C) ] T^m = \sum_{m=0}^{2(g-1)} [\Sym ^m ( C ) ] T^m  +    [\Pic ^0 ( C ) ] \sum_{m\geqslant 2g-1} \frac{\LL^{m-g+1} - 1 }{\LL - 1 } T^m 
\]
(where the first sum is empty if $g=0$).
Consider for any $\xi \in G ( F ) ^\vee$ 
\[
G(\TT, \xi ) =\prod_{v\in C} \prod_{\alpha \in \AAA_U^0 (\xi ) }  ( 1 - \LL^{\rho_\alpha - 1 } T_\alpha )^{-1} 
\]
(remark that if $\xi = 0$ then $\AAA_U^0 ( \xi ) = \AAA_U$). 
By Proposition \ref{proposition:euler-product:multiplicativity} one has   
\[
G(\TT,\xi ) =
\prod_{\alpha \in \AAA_U^0 (\xi )} Z_C^\text{Kap} \left ( \LL^{\rho_\alpha - 1 } T_\alpha \right )
% = \sum_{\mm \in \NN ^{\AAA}} [\times_{\alpha \in \AAA} \Sym^{m_\alpha} ( C ) ] \LL^{\langle \rho - 1 , \mm \rangle} \TT^{\mm}  .
\]
In order to easily adapt the computations done before, we may assume $g\geqslant 1 $, replace $Z_C^\text{Kap} $ by 
\[
\widetilde{Z^\text{Kap}_C} ( T ) = \sum_{m\geqslant 0 }   [\Pic ^0 ( C ) ] \frac{\LL^{m-g+1} - 1 }{\LL - 1 } T^m
\]
in the expression of $G( \TT , \xi ) $ and then control the error term coming from this slight modification. If we do so, then we obtain 
\[
Z ( \TT , \xi ) = G ( \TT , \xi ) \prod_{v\in C} F_v ( \TT , \xi ) = \widetilde{G} ( \TT , \xi ) \prod_{v\in C} F_v ( \TT , \xi ) + H ( \TT , \xi ) \prod_{v\in C} F_v ( \TT , \xi )
\]
where 
\begin{equation}
\label{eq-def:G-tilde(T,xi)}
\widetilde{G}(\TT , \xi ) =
\prod_{\alpha \in \AAA_U^0 (\xi )} \widetilde{Z_C^\text{Kap}} \left ( \LL^{\rho_\alpha - 1 } T_\alpha \right )
\end{equation}
and 
\begin{equation}
\label{eq-def:H(T,xi)}
H ( \TT , \xi ) 
 = 
 \prod_{\alpha \in \AAA_U^0 (\xi )} Z_C^\text{Kap} \left ( \LL^{\rho_\alpha - 1 } T_\alpha \right )
 - \prod_{\alpha \in \AAA_U^0 (\xi )} \widetilde{Z_C^\text{Kap}} \left ( \LL^{\rho_\alpha - 1 } T_\alpha \right ).
\end{equation}
 We then put 
 \begin{equation}\label{eq-def:zeta-tilde(T,xi)}
 \widetilde{Z}(\TT , \xi )  = \widetilde{G} ( \TT , \xi ) \prod_{v\in C} F_v ( \TT , \xi ). 
 \end{equation}

\subsubsection{Main term}

Basically, the contribution of $\widetilde{Z} ( \TT , \xi ) $ has already been treated in \S \ref{paragraph:Manin-Peyre-projective-line:trivial-characters} for the case $\xi = 0$ and in \S \ref{paragraph:Manin-Peyre-projective-line:non-trivial-characters} for non-trivial characters. The only difference with the particular case of the projective line is a factor $[ \Pic ^0 ( C ) ]\LL ^{(1-g)}$. Indeed, starting from (\ref{eq-def:G-tilde(T,xi)}) one gets
\begin{align*}
\widetilde{G} ( \TT , \xi ) &  = \prod_{\alpha \in \AAA_U^0( \xi )} \frac{ [\Pic ^0 ( C ) ]}{1-\LL}  \left ( \frac{1}{1-\LL^{\rho_\alpha -1} T_\alpha} - \frac{\LL^{1-g}}{1-\LL^{\rho_\alpha} T_\alpha}  \right ) \\
%%%%%%%%%
& = \frac{ [\Pic ^0 ( C ) ]^{ \left |\AAA_U^0( \xi ) \right | } }{(1-\LL)^{|\AAA|}} \sum_{\varepsilon \in \{0,1\}^{\AAA_U^0( \xi )}} \frac{\left (-\LL^{  1 - g} \right )^{  \left | \AAA_U^0( \xi ) \right | - |\varepsilon | }}{\prod_{\alpha\in \AAA_U^0( \xi )} \left ( 1-\LL^{ \rho_\alpha -\varepsilon_\alpha } T_\alpha \right ) } \\
%%%%%%%%% 
& =  \left  ( \frac{ [ \Pic ^0 ( C ) ] \LL^{1-g}    }{ \LL - 1  } \right ) ^{ \left |\AAA_U^0( \xi )\right  | }  \sum_{\varepsilon \in \{0,1\}^{\AAA_U^0( \xi )}} (-1)^{|\varepsilon | }   \widetilde{G_\varepsilon}(\TT , \xi  )  \LL^{ |\varepsilon | (g-1)}
\end{align*}
where 
\[
\widetilde{G_\varepsilon}(\TT , \xi  ) = \prod_{\alpha\in \AAA_U^0( \xi )} \left ( 1-\LL^{ \rho_\alpha -\varepsilon_\alpha } T_\alpha \right )^{-1} = 
\sum_{\mm\in \NN^{\AAA_U^0( \xi )}} \LL^{\langle \rho - \varepsilon , \mm \rangle_{\AAA_U^0( \xi )} } \TT ^{\mm} .
\]
The following proposition summarizes what we obtain if we replace $G_\varepsilon(\TT , \xi  ) $ by $\widetilde{G_\varepsilon}(\TT , \xi  )$ in the previous paragraph. Again, it ensures the negligibility of the terms corresponding to non-trivial characters (up to the error term coming from $H ( \TT , \xi ) $, which is treated in the next paragraph). 
\begin{myptn}[$\UUU = \XXX$, $C=C_0$, $\AAA = \AAA_U$]
For any character $\xi$, there exists a decomposition 
\[
\widetilde{Z} ( \TT , \xi ) =
 \sum_{\mm \in \mathbf Z ^{\AAA} }  \tilde{\mathfrak a}^\xi_{\mm} \TT^{\mm} 
=\sum_{\varepsilon \in \{ 0 , 1 \}^{\AAA_U^0 ( \xi ) } }   \sum_{\mm \in \mathbf Z ^{\AAA} } 
\tilde{\mathfrak a} _{\mm} ^{\xi , \varepsilon }
 \TT^{\mm} 
\]
and a real number $\delta_\xi >0 $ (depending only on the $d_\alpha ( \xi ) $) 
such that 
\begin{itemize}
\item if $\xi = 0 $ 
\begin{align*}
\tilde{\mathfrak a}_{\mm}^{0,\mathbf 0} \LL^{-\langle \rho , \mm \rangle } 
& =
 \left  ( \frac{ [ \Pic ^0 ( C ) ]\LL ^{1-g}  }{ \LL - 1   } \right ) ^{ \left  |\AAA_U \right |}  
 \sum_{\substack{\mm ' \in \mathbf Z^{\AAA}  \\\mm ' \leqslant \mm }} \mathfrak b_{\mm ' } \LL^{ - \langle \rho   ,  \mm ' \rangle } \\
& \underset{\rho\text{-weight-lin.}}{\longrightarrow} \left  ( \frac{ [ \Pic ^0 ( C ) ]\LL ^{1-g}  }{ \LL - 1   } \right ) ^{ \left  |\AAA_U \right |}  \prod_{v\in C_0} F_v ( ( \LL ^{ - \rho_\alpha   } )_{\alpha_\in \AAA } , 0 ) 
\end{align*}
where $\mathfrak b_{\mm ' }$ is the $\mm '$-th coefficient of $ \prod_{v\in C_0} F_v ( \TT , 0 ) $
and 
\[
w\left  ( \tilde{\mathfrak a}^{0,\varepsilon}_{\mm} \LL^{-\langle \rho , \mm \rangle } \right  ) < - \delta_0 \langle  \varepsilon , \mm \rangle \text{ for all $\varepsilon \in \{ 0 , 1 \}^{\AAA_U} \setminus \{ \mathbf 0 \} $}
\] 
when $ \min_{\alpha \in \AAA} ( m_\alpha  ) \to \infty $ ; 
\item if $\xi \neq 0$
\[
w \left  ( \tilde{\mathfrak a}_\mm^\xi \LL^{-\langle \rho , \mm \rangle } \right ) < - \delta_\xi \left ( \langle \rho , \mm  \rangle_{\AAA\setminus \AAA_U^0 ( \xi )} + \langle \varepsilon  , \mm  \rangle_{ \AAA_U^0 ( \xi )}\right ) \text{ for all $\varepsilon \in \{ 0 , 1 \}^{\AAA_U^0 ( \xi ) }$}
\] 
when $\min ( m_\alpha )\to \infty$. 
\end{itemize}
\end{myptn} 

\subsubsection{Error term}
Now we study the contribution of $H(\TT , \xi) $, still assuming $\UUU = \XXX$. We rewrite this error term as follows, using the convenient notations $\mathbf g = ( g , ... , g ) $ and $\mathbf 1 = ( 1 , ... , 1 )$. In this paragraph $\xi$ is any character, including the trivial one; in this case $\AAA_U^0 ( 0 ) = \AAA_U$. 
\begin{align*}
H ( \TT , \xi ) 
& = 
\sum_{\varnothing \neq A\subset  \AAA_U^0 ( \xi ) } \prod_{\alpha \in \AAA_U^0 ( \xi ) \setminus A} \widetilde{Z^\text{Kap}_C } ( \LL^{\rho_\alpha -1} T_\alpha ) \prod_{\alpha \in A}  \left ( Z^\text{Kap}_C  ( \LL^{\rho_\alpha -1} T_\alpha ) - \widetilde{Z^\text{Kap}_C } ( \LL^{\rho_\alpha -1} T_\alpha ) \right ) \\
& = \sum_{ \substack{ \varnothing \neq A \subset \AAA_U^0 ( \xi )\\ \varepsilon \in \{0,1\}^{\AAA_U^0 ( \xi ) \setminus A} \times \{ 1 \}^A  }}  
[\Pic^0 ( C ) ]^{\left |\AAA_U^0 ( \xi ) \setminus A \right |}
\frac{\left (-\LL^{1-g}\right )^{ \left | \AAA_U^0 ( \xi ) \right | - |\varepsilon |  }}{(1-\LL)^{\left |\AAA_U^0 ( \xi ) \setminus A \right |}}  H_{A,\varepsilon} ( \TT , \xi ) 
\end{align*}
with for any $A \subset \AAA_U^0 ( \xi )$ non-empty and $\varepsilon \in \{0,1\}^{\AAA_U^0 ( \xi ) \setminus A} \times \{ 1 \}^A$
\[
H_{A,\varepsilon} ( \TT , \xi ) = 
\sum_{\mm\in \NN^{\AAA_U^0 ( \xi )}} \LL^{\langle \rho - \varepsilon   , \mm \rangle } \TT ^{\mm}
\prod_{\alpha \in  A} 
\left ( 
[ \Sym^{m_\alpha } C] - [\Pic^0 ( C ) ] \frac{\LL^{m_\alpha - g - 1 } -1 }{\LL - 1 } 
\right ) .
\]
Thus studying the contribution of $H(\TT , \xi )$ amounts to studying the $\mm$-th coefficient of the product 
\[
H_{A,\varepsilon} ( \TT , \xi ) \prod_{v\in C} F_v ( \TT , \xi ) 
\]
for every $A\subset \AAA_U^0 ( \xi )$ non-empty and $\varepsilon \in \{ 0 , 1 \}^{\AAA_U^0 ( \xi ) \setminus A} \times \{ 1 \}^A$. 
In what follows we fix such an $A$. 
We know that the term of multi-degree $\mm $ of $H_{A,\varepsilon} (\TT ) $
is zero whenever there is $\alpha \in A $ such that $m_\alpha \geqslant 2g - 1$.
So the $\mm$-th term of the product $H_{A,\varepsilon} ( \TT , \xi ) \prod_{v\in C} F_v ( \TT , \xi ) $ is equal to
\begin{align}
\mathfrak{e}_{\mm}^{A,\varepsilon} = \label{equ:error-term-general-curve-X=U}  
 \LL^{\langle \rho - \varepsilon  , \mm \rangle_{\AAA_U^0 ( \xi )}}
\sum_{\substack{ ( \nn , \nn ' )  \in \ZZ^\AAA \times \NN^\AAA \\ \nn + \nn ' = \mm \\  \nn_A ' \leqslant 2 ( \mathbf g_{ A} - \mathbf 1 ) }}   
\mathfrak{b}_{\nn } 
 \LL^{ - \langle \rho - \varepsilon  , \nn \rangle_{\AAA_U^0 ( \xi )} }
\prod_{\alpha \in   A} 
\left ( 
[ \Sym^{n_\alpha '  } C] - [\Pic^0 ( C ) ] \frac{\LL^{n_\alpha   '  - g - 1 } -1 }{\LL - 1 } 
\right )  
\end{align}
where $\mathfrak b_{\nn}$ is the $\nn$-th coefficient of $\prod_{v\in C} F_v ( \TT , \xi )$ and $\nn_A ' $ denotes the restriction to $A$ of $\nn ' \in \mathbf N^\AAA$. 
Since $A \neq \varnothing $, by Proposition \ref{proposition:convergence-residue-of-Zeta-0} (for the trivial character), Proposition \ref{proposition:convergence-residue-of-Zeta-xi} (for non-trivial characters) and Lemma \ref{lemma:negligible-terms} we get the existence of a $\delta  >0 $ (which only depends on the $d_\alpha ( \xi )$) such that for any $\varepsilon \in \{ 0 , 1 \}^{\AAA_U^0 ( \xi )\setminus A}\times \{ 1 \}^A $
\begin{align}\notag
& w \left ( \sum_{\substack{ ( \nn , \nn ' )  \in \ZZ^\AAA \times \NN^\AAA \\ \nn + \nn ' = \mm \\  \nn_A ' \leqslant 2 ( \mathbf g_{ A} - \mathbf 1 ) }}  
\mathfrak{b}_{\nn } 
 \LL^{ - \langle \rho , \nn \rangle}
\LL^{- \langle \varepsilon , \nn '  \rangle_{\AAA_U^0 ( \xi )} - \langle \rho , \nn '  \rangle_{\AAA \setminus \AAA_U^0 ( \xi )} }   \right )\\
&  < - \delta \left  (  \langle \varepsilon , \mm \rangle_{\AAA_U^0 ( \xi )} +  \langle \rho , \mm    \rangle_{\AAA \setminus \AAA_U^0 ( \xi )}    \right )  \label{equation:weight-control-H}
\end{align}
since for any $\nn \in \ZZ^\AAA$ one has the decomposition 
\[
\langle \rho  - \varepsilon , \nn \rangle_{\AAA_U^0 ( \xi ) } = \langle \rho , \nn \rangle - \langle \varepsilon , \nn \rangle_{\AAA_U^0 ( \xi )} - \langle \rho , \nn \rangle_{\AAA \setminus  \AAA_U^0 ( \xi )}. 
\]
Remark that the product over $A$
can only take a finite number of values, its weight is therefore bounded. 
Combining \eqref{equ:error-term-general-curve-X=U} and \eqref{equation:weight-control-H}, we conclude that there exists $\delta '  >0 $ such that 
\[
w \left (  \LL^{- \langle \rho , \mm \rangle }  \mathfrak{e}_{\mm}^{A,\varepsilon} \right )   < - \delta ' \left  (  \langle \varepsilon , \mm \rangle_{\AAA_U^0 ( \xi )} +  \langle \rho , \mm    \rangle \right )  
\]
whenever $ \langle \varepsilon , \mm \rangle_{\AAA_U^0 ( \xi )} +  \langle \rho , \mm    \rangle_{\AAA \setminus \AAA_U^0 ( \xi )}$ is large enough. 

%when $\min ( m_\alpha ) \to \infty $. 
%One may pay attention to the fact that here weight-linear convergence is actually essential to get the result. 

This analysis proves that the slight modification we performed on the Kapranov Zeta function of $C$ is painless  in the case $\UUU = \XXX$. It allows us to extend the result to any smooth projective irreducible curve of genus $g$. The control of the error terms coming from $H ( \TT , \xi )$ is given by the following proposition. 
\begin{myptn}[$\UUU = \XXX$, $C=C_0$]\label{propisition:result-any-curve-decomposition-H} Let $ \xi$ be any character. 
There exists a decomposition 
\[
H ( \TT , \xi ) \prod_{v\in C_0} F_v ( \TT , \xi )  =  \sum_{\substack{\varnothing \neq A \subset \AAA_U^0 ( \xi )  \\ \varepsilon \in \{0,1\}^{\AAA_U^0 ( \xi ) \setminus A} \times \{ 1 \}^A}}  \sum_{\mm \in \mathbf Z^\AAA}  \mathfrak{h}_\mm^{A,\varepsilon} \TT^\mm
\]
and a real number $\delta_\xi > 0$ (depending only on $\prod_{v\in C} F_v ( \TT , \xi ) $) such that for all non empty subset $A \subset \AAA_U^0 ( \xi ) $ and $\varepsilon \in \{0,1\}^{\AAA_U^0 ( \xi ) \setminus A} \times \{ 1 \}^A$ 
\[
w \left  (   \LL^{- \langle \rho , \mm \rangle }  \mathfrak{h}_{\mm}^{A,\varepsilon}   \right ) < - \delta_\xi  \left  (  \langle \varepsilon , \mm \rangle_{\AAA_U^0 ( \xi )} +  \langle \rho , \mm    \rangle \right )  
\]
when $\min ( m_\alpha ) \to \infty$. 
\end{myptn}

\subsubsection{$S$-integral points and the trivial character}
We now start treating the general case $\UUU \subset \XXX$. For clarity's sake we will keep using the coarse height zeta function $Z(\TT ) $ until it becomes necessary to switch to its refined version $\ZZZ ( \UU)$, that is, when we will integrate the Euler product over $S$ to the global one over $C$. 

% \paragraph{A short break: Clemens complexes.}
\smallskip 

In order to treat the case of $S$-integral points, as in \cite[\S 2.3]{chambert2016motivic} we introduce a simplicial complex encoding the intersecting data of the boundary's components. Let $Y$ be a smooth algebraic variety over a field $L$ and $\Delta $ a divisor on $Y$ with strict normal crossings. The Clemens complex $\Cl ( Y , \Delta ) $ is the simplicial complex whose vertices are the irreducible components $( \Delta _i )_{i\in I}$ of $\Delta$. It has an edge between $\Delta_i$ and $\Delta_j $ for $i\neq j$ if and only if $\Delta_i \cap \Delta_j \neq \varnothing$. It has a two dimensional face, given by the vertices $\Delta_i $, $\Delta_j$ and $\Delta_k$, for $i,j,k$ pairwise distinct, if and only if $\Delta_i \cap \Delta_j \cap \Delta_k \neq \varnothing$, and so on for higher dimensional faces: in general, a subset $J\subset I$ corresponds to a face of dimension $|J|-1$ of $\Cl ( Y , \Delta )$ if and only if $\Delta_J = \cap_{j\in J} \Delta_j \neq \varnothing$. 
Since we assumed that $Y$ is smooth and $\Delta$ has strict normal crossings, the intersections we are considering are smooth as well. 

Then a maximal face of $\Cl ( Y , \Delta ) $ is a simplex whose vertices are indexed by a subset $J\subset I$ such that $\Delta_J \neq \varnothing $ and $\Delta_J  \cap \Delta_k = \varnothing $ for any $k\in I \setminus J$. In particular, the dimension of $\Cl ( Y , \Delta ) $ is the maximal number of components of $\Delta$ with non-empty intersection minus one. 

The $L$-Clemens complex\footnote{This complex is called \textit{analytic Clemens complex} in \cite{chambert2016motivic}.} $\Cl_L ( Y , \Delta ) $ is defined in a similar way, by restriction to $L$-points. More precisely, $\Cl_L ( Y , \Delta ) $ is the subcomplex of $\Cl  ( Y , \Delta ) $ consisting of simplices $\Delta_J \in \Cl ( Y , \Delta ) $ such that $\Delta_J ( L ) \neq \varnothing $. The set of maximal faces of $\Cl _L ( Y , \Delta ) $ is $\Cl_L^\mathrm{max} ( Y , \Delta ) $.

% \paragraph{Back to our problem.}
\smallskip 

Let $s \in S$. 
By \cite[\S 6]{chambert2016motivic} and \cite[\S 6.3.6]{bilu2018motivic} we have 
\[
Z_s ^{\beta_s} ( \TT , 0 ) = \sum_{ A \subset \AAA} \LL^{\rho^\beta}  [\Delta_s ( A , \beta ) ] \LL^{-n+|A|} ( 1 - \LL^{-1})^{|A|} \prod_{\alpha\in A} \frac{\LL^{\rho_\alpha -1}T_\alpha}{1- \LL^{\rho_\alpha -1}T_\alpha} .
\]
Our goal here is to isolate the contribution to the poles of the boundary $D$.
To do so, following \cite{bilu2018motivic,chambert2016motivic} we associate to any pair $(A,\beta ) $ 
such that $\Delta_s ( A\cap \AAA_D , \beta ) \neq \varnothing $ 
a maximal subset $\mface_s $ of $\AAA_D$ such that 
$A\cap \AAA_D \subset \mface_s $ and $\Delta_s ( \mface_s , \beta ) \neq \varnothing $. 
For a general $A \subset \AAA$,  there is no canonical choice of such a maximal subset: we arbitrary choose such a map. This will not be a key issue, as we will see in Remarks \ref{remark:maximal-faces-canonical-choice} \& \ref{remark:maximal-faces-value-polynomial} below. 
We assume furthermore that $\Delta_s ( \mface_s , \beta ) $
has a $F_s$-point; this assumption is natural when one considers sections intersecting, above $s$, the divisors $\DDD_\alpha$ given by $\mface_s \subset \AAA_D$, see Remark \ref{remark:moduli-space-curves-max-faces}. 
\begin{mylemma}
Such an $\mface _s \subset \AAA_D$ corresponds to a maximal face of the analytic Clemens complex $\Cl_s ( X , D ) = \Cl_{F_s} ( X_{F_s} , D_{F_s} )$.
\end{mylemma}
\begin{proof}
 By definition, $\Delta_s ( A , \beta )$ is the set of points of the fibre $\XXX_s $ belonging exclusively to $\DDD_{\alpha } $ and $E_\beta $, that is to say 
 \[
 \Delta_s ( A , \beta ) = \left ( \bigcap_{\alpha \in A } \DDD_{\alpha , s } \cap E_\beta \right )  \setminus \left ( \bigcup_{\substack{\alpha \notin A \\ \beta ' \neq \beta} } \DDD_{\alpha , s } \cup  E_{\beta '}  \right ) .
 \]
For simplicity we assume that $\BBB_{1,s} = \{ \beta \}$. 
 
We argue by contradiction: assume that $\mface_s $ is not a maximal face of $\Cl_{F_s} (  X_{F_s} , D_{F_s} )$. It means that there exists a non-empty subset $\mface_s ' \subset \AAA_D $ which contains $\mface_s $ as a proper subset and such that 
\[
  \left  ( \bigcap_{\alpha \in \mface_s '  } \left (D_{\alpha }\right )_{F_s} \right  ) ( F_s )  \neq \varnothing .
\]
We can assume  furthermore that $\mface_s ' $ is maximal for this property. In other words $\mface_s ' $ is a maximal face of $\Cl_{F_s} (  X_{F_s} , D_{F_s} ) $ containing $\mface_s$ as a proper subface. 
Remark that since $X \setminus U = \cup_{\alpha \in \AAA_D} D_\alpha$ and $X_{F_s} ( F_s ) = X ( F_s) $, together with the maximality of $\mface_s ' $, then 
\begin{equation}\label{equation:F-points-of-intersections}
\bigcap_{\alpha \in \mface_s '} D_\alpha  \cap D_{\alpha ' } (F_s) = \varnothing 
\end{equation}
for every $\alpha '  \in \AAA \setminus \mface_s '$.

Our argument relies on the fact that $\XXX \to C$ is proper.  
Composing with the projection $X_{F_s} \to \XXX$ and applying the valuative criterion of properness, this $F_s$-point $x_s$ uniquely lifts to a $\mathcal O_s$-point $\widetilde{x_s}$  of $\XXX$. 
\begin{center}
\begin{tikzcd}
X_{F_s} \arrow[rr] & & \XXX \arrow[d,"\pi"]   \\
\Spec ( F_s ) \arrow[r] \arrow[u , "x_s" ] &  \Spec ( \mathcal O_s  ) \arrow[ur, dashed,  "\widetilde{x_s}" '] \arrow[r] & C  
\end{tikzcd}
\end{center}
Since for every $\alpha$, $\DDD_\alpha $ is the closure in $\XXX$ of $D_\alpha$,  $\widetilde{x_s}$ is actually a $\mathcal O_s$-point of $\bigcap_{\alpha \in \mface_s '  } \DDD_{\alpha } $.  
By reduction modulo the maximal ideal of $\mathcal O_s$, one gets a well-defined reduction $\overline{x_s}$ to the special fibre $\XXX_s $.
\begin{center}
\begin{tikzcd}
\XXX_s \arrow[r] \arrow[d] & \XXX_{\mathcal O_s} \arrow[d]\arrow[r] & \XXX \arrow[d] \\
\Spec ( k ) \arrow[r] \arrow[u, shift right=1.5ex, dashed, "\overline{x_s}" '] & \Spec ( \mathcal O_s ) \arrow[r] \arrow[u, shift right=1.5ex, dashed] \arrow[ur, dashed , "\widetilde{x_s}"] & C 
\end{tikzcd}
\end{center}
This very last point $\overline{x_s} $ is actually a point of $\bigcap_{\alpha \in \mface_s '  } \DDD_{\alpha , s }$ which does not belong to any of the $\DDD_{\alpha ' }$ for $\alpha ' \in \AAA \setminus \mface_s ' $: since all the intersections we are considering are smooth, we use Hensel's lemma here.  
This contradicts the maximality of $\mface_s  $ with respect to the hypothesis $\Delta_ s ( \mface_s , \beta ) \neq \varnothing$, since we have found a $\mface_s '$ strictly containing $\mface_s$ and such that $\Delta_s ( \mface_s ' , \beta )$ is non-empty. Thus $\mface_s $ is a maximal face of $\Cl_s ( X , D )$.
In the general case where the fibre above $s$ is not irreducible, then the above argument gives at least one irreducible component $\beta$ for which $\Delta_ s ( \mface_s ' , \beta ) \neq \varnothing $. The conclusion follows.
\end{proof}

\begin{myremark}\label{remark:maximal-faces-canonical-choice}
If $\mface _s$ is a maximal face of $\Cl_s ( X , D )$, then for every $A\subset \AAA_U$ and $\beta \in \BBB_{1,s}$ the pair $(A \cup \mface_s , \beta ) $ is sent to $\mface _s$. 
\end{myremark}

Collecting the terms with respect to every such a maximal face $\mface_s$, one gets 
\begin{align*}
Z_s^{\beta_s} ( \TT , 0 ) & = \sum_{\mface_s \in  \Cl_s^{\mathrm{max} }( X , D )} \sum_{\substack{A \subset \AAA\\ ( A , \beta_s ) \mapsto \mface_s }}   \LL^{\rho^\beta}  [\Delta_s ( A , \beta ) ] \LL^{-n+|A|} ( 1 - \LL^{-1})^{|A|} \\
& \times \prod_{\alpha\in A\setminus \mathscr \mface_s } \frac{\LL^{\rho_\alpha -1}T_\alpha}{1- \LL^{\rho_\alpha -1} T_\alpha}  \prod_{\alpha\in A \cap \mathscr \mface_s} \frac{\LL^{\rho_\alpha -1}T_\alpha}{1- \LL^{\rho_\alpha -1}T_\alpha} 
\\
%%%%
%%%%
& = \sum_{\mface_s \in  \Cl_s^{\mathrm{max} }( X , D )}  \frac{P_{\mface_s}^{\beta_s} ( \TT ) }{\prod_{\alpha \in \AAA_U} ( 1-\LL^{\rho_\alpha -1 } T_\alpha )} \prod_{\alpha \in \mface_s } \frac{1}{1-\LL^{\rho_\alpha -1 } T_\alpha } 
\end{align*}
where for all $\mface_s \in  \Cl_s^{\mathrm{max}} ( X , D )$
\begin{align*}
& P_{\mface_s}^{\beta_s} ( \TT ) =  \sum_{\substack{A \subset \AAA \\ ( A , \beta_s ) \mapsto \mface_s }} 
\LL^{\rho^\beta}  \left (  \prod_{\alpha\in A } \LL^{\rho_\alpha -1}T_\alpha \right )\\ 
 &  \times 
 [\Delta_s ( A , \beta_s ) ] \LL^{-n} ( \LL -1)^{|A|} 
\prod_{\alpha \in ( \AAA_U \setminus A) \cup ( \mface_s \setminus A)} ( 1- \LL^{\rho_\alpha -1}T_\alpha ) .
\end{align*}
\begin{myremark}\label{remark:maximal-faces-value-polynomial} 
In the expression of $ P_{\mface_s}^{\beta_s} ( \TT ) $, the polynomials $(1-\LL^{\rho_\alpha '  } T_\alpha )$ for $\alpha \in \mathscr \mface_s$ give vanishing factors at $T_\alpha = \LL^{-\rho_\alpha ' }$. In particular, for any $A \subset \AAA$ such that $A \cap \AAA_D \subset \mface_s$, we have 
\[ 
 \prod_{\alpha \in ( \AAA_U \setminus A) \cup ( \mface_s \setminus A)} ( 1- \LL^{\rho_\alpha -1}T_\alpha )  ( \LL^{-\rho_\alpha ' } ) \neq 0 
 \] if and only if $A \cap \AAA_D =\mface_s$, and thus using Remark \ref{remark:maximal-faces-canonical-choice}
 \begin{align*}
 P_{\mface_s } ^{\beta_s}\left  ( \left  ( \LL^{-\rho_\alpha ' } \right )_{\alpha \in \AAA} \right ) 
 & =    \sum_{\substack{A \subset \AAA \\ A \cap \AAA_D = \mface_s }} 
\LL^{\rho^\beta}   \LL^{- | A \cap \AAA_U |}
 [\Delta_s ( A , \beta_s ) ] \LL^{-n} ( \LL -1)^{|A|} 
\left  ( 1- \LL^{-1} \right )^{|\AAA_U \setminus A |} .  \\
%%%%
%%%%
 & = 
 ( \LL - 1 )^{| \mface_s |}
  ( 1 - \LL^{-1} )^{|\AAA_U| } \LL^{-n} 
  \LL^{\rho^{\beta_s}   } 
 \sum_{ A \subset \AAA_U }    [\Delta_s ( \mface_s \cup A , \beta_s ) ] . 
 \end{align*}
 
 By definition of $\Delta_s ( \mface_s \cup A , \beta )$ one has 
 \[
 \sum_{A\subset \AAA_U}  [\Delta_s ( \mface_s \cup A , \beta )] =  [  E_{\beta_s}^\circ \cap \DDD_{\mface_s } ] 
 \]
 where $E_{\beta}^\circ $ stands for  $E_{\beta} \setminus \cup_{\beta ' \neq \beta } E_{\beta ' }$ and $\DDD_{\mface_s } $ is the intersection $\cap_{\alpha \in \mface_s}  \DDD_{\alpha , s}$.
Finally the value of $P_{\mface_s}^{\beta_s}$ at $\left  ( \LL^{-\rho_\alpha ' } \right )_{\alpha \in \AAA} $ is 
\[
P_{\mface_s}^{\beta_s} \left  ( \left  ( \LL^{-\rho_\alpha ' } \right )_{\alpha \in \AAA} \right ) 
 =
  ( \LL - 1 )^{| \mface_s |}
  ( 1 - \LL^{-1} )^{|\AAA_U| } 
  \LL^{\rho^{\beta_s}   }  [  E_{\beta_s}^\circ \cap \DDD_{\mface_s } ] \LL^{-n}  . 
 \]
\end{myremark}
 
Now we are able to consider the whole Euler product for $\xi = 0$ and evaluate its coefficients. 
For any $\beta \in \prod_v \BU_v$, which we call a \textit{choice of vertical components} (recall that this product is actually finite, see \S \ref{subsection:back-to-our-setting} for the definition of $\BU_v$), let us refine Notation \ref{notation:F_v(T,0)} by setting for every $v\in C(k)$ 
\[
\FFF_v^\beta ( \UU , 0 )  = \prod_{\alpha \in \AAA_U } ( 1 - \LL^{\rho_\alpha - 1} U_\alpha )  \ZZZ_v^\beta ( \UU , 0 ) 
\]
where $\ZZZ_v^\beta ( \UU , 0 ) $ is the local factor of the refined zeta function $\ZZZ^\beta ( \UU , 0 ) $, see Definition \ref{def:refined-zeta-function}.
This local factor only depends on the indeterminates $\UU_0$ when $v \in C_0 ( k ) $ and on $( \UU_0, \UU_s ) $ when $v=s\in S$. 
Similarly we define $\PPP_{\mface_s}^{\beta} ( \UU_0 , \UU_s ) $ starting from $P_{\mface_s}^{\beta} ( \TT) $ for every $s\in S$
%where $\ZZZ_{v,A}^\beta (\UU_0  , 0 ) $ has been defined in Remark \ref{remark:dependance-of-local-factors} 
%and only depends on the indeterminates $( U_\alpha )_{\alpha \in A} $. 
and we write $\Cl_S^{\mathrm{max}} ( X , D ) $ for the product $ \prod_{s\in S} \Cl_s^{\mathrm{max}} ( X , D ) $. 
Then 
\begin{align*}
\ZZZ^\beta ( \UU , 0 ) & = \prod_{v\in C} \left ( \prod_{\alpha \in \AAA_U } ( 1 - \LL^{\rho_\alpha -1} U_\alpha  )^{-1}  \FFF_v^\beta ( \UU , 0 ) \right ) \\
%%%%
%%%%
&= \prod_{\alpha \in \AAA_U } Z_C ( \LL^{\rho_\alpha ' - 1 } U_\alpha )  \prod_{v\in C} \FFF_v^\beta ( \UU , 0 )\\ 
%%%%
%%%%
& =  \left ( \prod_{\alpha \in \AAA_U } Z_{C} \left ( \LL^{\rho_\alpha ' - 1 } U_\alpha \right )  \prod_{v\in C_0} \FFF_v^\beta ( \UU_0 , 0 )\right ) \\
&  \times  \prod_{s\in S}  \left (  \sum_{\mface_s \in  \Cl_s^{\mathrm{max}} ( X , D )}  \PPP_{\mface_s}^{\beta} ( \UU_0 , \UU_s ) \prod_{\alpha \in \mface_s } \frac{1}{1-\LL^{\rho_\alpha -1 } U_{\alpha  , s}  }   \right )\\
%%%%
%%%% 
& = \sum_{ \mface \in \Cl_S^{\mathrm{max}} ( X , D )} \ZZZ_\mface^\beta ( \UU , 0 )  
\end{align*}
where for every $\mface \in \Cl_S^{\mathrm{max}} ( X , D ) \subset  \mathcal P \left ( \AAA_D \right ) ^S $
\[
\ZZZ_\mface^\beta ( \UU , 0 )  =
\left ( \prod_{s\in S} \prod_{\alpha \in \mface_s} \frac{1}{1-\LL^{\rho_\alpha  - 1 } U_{\alpha , s} } \right ) 
\left (  \prod_{\alpha \in \AAA_U} Z_{C} ( \LL^{\rho_\alpha ' - 1 } U_\alpha ) \right ) 
\left ( \prod_{v\in C_0} \FFF_v ^\beta ( \UU_0 , 0  ) \right ) 
\left ( \prod_{s \in S} \PPP_{\mface_s}^\beta ( \UU_0  , \UU_s )  \right ) .
\]
This series only depends on the inderminates indexed by $\AAA_U$ or $\mface $.
Let us introduce the series 
\[
\GGG_\mface (\UU , 0 ) = 
 \prod_{\alpha \in \AAA_U} Z_C \left ( \LL^{\rho_\alpha ' - 1 } U_\alpha \right ) \prod_{s\in S} \prod_{\alpha \in \mface_s }  ( 1  - \LL_s^{\rho_\alpha ' } U_{\alpha , s}  )
\]
for all choice of maximal faces $\mface\in \Cl_S^{\mathrm{max}} ( X , D )$. 
These notations allow us to rewrite $\ZZZ_{\mface} ( \TT , 0 )$ as 
\[
\ZZZ_{\mface}^\beta ( \UU , 0 ) =  \GGG_\mface (\UU , 0 )  \prod_{v\in C_0} \FFF_v ^\beta ( \UU_0 , 0  ) \prod_{s \in S} \PPP_{\mface_s}^\beta ( \UU_0 , \UU_s ) . 
\]
We may decompose $\GGG_\mface (\UU , 0 )$ as 
\[
\GGG_\mface (\UU , 0 ) = \widetilde{\GGG_\mface} ( \UU , 0 )  + \HHH_\mface ( \UU , 0 )
\]
where 
\begin{align*}
\widetilde{\GGG_\mface} ( \UU , 0 ) & = \widetilde{G} ( \UU_0 , 0 ) \prod_{s\in S } \prod_{\alpha \in \mface_s} ( 1  - \LL_s^{\rho_\alpha ' } U_{\alpha , s}  )^{-1}\\
\HHH_\mface ( \UU , 0 ) & =  H ( \UU_0 , 0 ) \prod_{ s\in S } \prod_{\alpha \in \mface_s}  ( 1  - \LL_s^{\rho_\alpha ' } U_{\alpha , s}  )^{-1} 
\end{align*}
with $\widetilde{G} ( \TT , 0 )$ and $H ( \TT , 0 )$ defined in (\ref{eq-def:G-tilde(T,xi)}) and (\ref{eq-def:H(T,xi)}) 
(this definition is licit since $\widetilde{G} ( \TT , 0 )$ and $H ( \TT , 0 )$ 
only depends on the indeterminates indexed by $\AAA_U$). 
This leads to a decomposition 
\[
\ZZZ_\mface^\beta ( \UU , 0 )  = \widetilde{\ZZZ_\mface^\beta} ( \UU , 0 )  +  \HHH_\mface ( \UU , 0 )  \left ( \prod_{v\in C_0} \FFF_v^\beta ( \UU_0 , 0  ) \right ) \left ( \prod_{s \in S} \PPP_{\mface_s}^{\beta_s} ( \UU_0 , \UU_s ) \right )
\]
where
\[
\widetilde{\ZZZ_\mface^\beta} ( \UU , 0 ) = \widetilde{\GGG_\mface} ( \UU , 0 ) \left (   \prod_{v\in C_0} \FFF_v^\beta ( \UU_0 , 0  ) \right ) \left (  \prod_{s \in S} \PPP_{\mface_s}^{\beta_s} ( \UU_0 , \UU_s ) \right ) . 
\]
Now we are ready to adapt the $\UUU = \XXX$ case 
to the study of the coefficients of $\ZZZ_\mface^\beta ( \UU , 0 ) $ 
and $\widetilde{\ZZZ_\mface^\beta} ( \UU , 0 )  $ for a fixed $\mface \in \Cl_S^{\mathrm{max}} ( X , D )$. 
We are going to show that asymptotically the main contribution in this decomposition of $\ZZZ_\mface^\beta ( \UU , 0 )$ 
comes from $\widetilde{\ZZZ_\mface^\beta} ( \UU , 0 ) $ 
while the error term coming from $\widetilde{\ZZZ_\mface^\beta} ( \UU , 0 ) - \ZZZ_\mface^\beta ( \UU , 0 )$ 
is negligible. 

An expression of $\widetilde{\GGG_\mface} ( \UU , 0 )$ is given by 
\[
\widetilde{\GGG_\mface} ( \UU , 0 ) = 
\left  ( \frac{ [ \Pic ^0 ( C ) ] \LL^{1-g}    }{ \LL - 1  } \right ) ^{ |\AAA_U | }  \sum_{\varepsilon \in \{0,1\}^{\AAA_U}} (-1)^{|\varepsilon | }   \widetilde{\GGG_{\mface , \varepsilon}}(\UU , 0  )  \LL^{ |\varepsilon | (g-1)}
\]
where for all $\varepsilon \in \{0,1\}^{\AAA_U}$ 
\begin{align*}
\widetilde{\GGG_{\mface , \varepsilon }}(\UU , 0 ) &  = \prod_{\alpha\in \AAA_U } \left ( 1-\LL^{\rho_\alpha '  -\varepsilon_\alpha } U_\alpha \right )  \prod_{s\in S} \prod_{\alpha \in \mface_s} ( 1  - \LL_s^{\rho_\alpha ' } U_{\alpha ,s }  )^{-1}  \\
& = \sum_{ \substack{ \m_0\in \NN^{\AAA_U} \\ (\m_s) \in \prod_{s\in S}\NN^{\mface_s} }  }   \LL^{\langle \rho  ' - \varepsilon , \m_0 \rangle_{\AAA_U} + \sum_{s\in S} \langle \rho ' , \m_s \rangle_{\mface_s}} \UU^{\m_0 + ( \m_s )_{s\in S}} .
\end{align*}
The coefficient of order $\mm \in \NN^{\AAA_U} \times \NN^{\mface} \subset \NN^{\AAA_U} \times \NN^{\AAA_D \times S}$ of $\widetilde{\ZZZ_\mface^\beta} ( \UU  , 0 ) $
is equal to
\begin{equation}\label{equation:main-term-for-beta-coeff}
\sum_{ \substack{ \nn , \nn ' \in \NN^{\AAA_U } \times \NN^\mface \\ \nn + \nn' =  \mm }} \mathfrak{c}_\mface^\beta  ( {\nn } )  \LL^{\langle \rho ' - \varepsilon , \nn ' \rangle _{\AAA_U \sqcup \mface}} 
\end{equation}
where $\mathfrak{c}_\mface^\beta  ( {\nn } )  $ is the $\nn $-th coefficient (with $\nn \in \NN^{\AAA_U} \times \NN^{\mface}$) of 
\[
\prod_{v\in C_0} \FFF_v^\beta ( \UU_0 , 0  ) \prod_{s\in S} P_{\mface_s}^{\beta_s}  ( \UU_0 , \UU_s ) 
\] 
and $\AAA_U \sqcup \mface$ is the disjoint union of $\AAA_U$ and $\mface$ in $\AAA_U \sqcup \AAA_D^S $. 
After normalising by $\LL^{\langle \rho ' , \mm \rangle_{\AAA_U \sqcup \mface}}$, the coefficient
(\ref{equation:main-term-for-beta-coeff}) for $\varepsilon = \mathbf 0$ is the $\mm ' $-th partial sum of the motivic Euler product 
\[ \left ( \prod_{v\in C_0} \FFF_v^\beta ( \UU_0 , 0 )\right )  (  ( \LL^{-\rho_\alpha ' } )_{\alpha \in \AAA_U} )  \times  \prod_{s\in S} \left ( P_{\mface_s}^{\beta_s} ( \UU_0, \UU_s )  (  ( \LL^{-\rho_\alpha ' } )_{\alpha \in \AAA_U}  ,  ( \LL^{-\rho_\alpha ' } )_{\alpha \in \AAA_D}   )  \right )
\]
to which one can apply Proposition \ref{proposition:convergence-residue-of-Zeta-0}.  
If $\varepsilon \neq \mathbf 0 $,
the only difference is again the factor $\LL^{\langle \varepsilon , \mm - \mm ' \rangle}$, which is controlled with Lemma \ref{lemma:negligible-terms} as in the case $\UUU = \XXX$. Since by Lemma \ref{lemma:weight-convergence-product-stability} we know that weight-linear convergence is stable by finite product, we deduce the following proposition. 

\begin{myptn}\label{proposition:general-case-Z-tilde-xi=0}
For all choice of maximal faces $\mface \in\Cl_S^{\mathrm{max}} ( X , D ) $ and vertical components $\beta \in \prod_v \BU_v$,
there exists a decomposition 
\[
\widetilde{\ZZZ_\mface^\beta} ( \UU , 0 ) = 
\sum_{\mm \in \NN ^{\AAA_U} \times \NN^\mface } 
\tilde{\mathfrak a} ^\beta _{\mm}
 \UU^{\mm}  
=
\sum_{\varepsilon \in \{ 0 , 1 \}^{\AAA_U } }   \sum_{\mm \in \NN ^{\AAA_U} \times \NN^\mface } 
\tilde{\mathfrak a} _{\mm} ^{\beta , \varepsilon } 
 \UU^{\mm}  
\]
as well as a real number $\delta > 0$ such that in $\widetilde{\mathscr M_k } $
\[
\tilde{\mathfrak a} _{\mm} ^{\beta , \mathbf 0 }   \LL^{ - {\langle \rho ' , \mm \rangle}_{\AAA_U \sqcup\mface}  } 
\]
converges $\rho'$-weight-linearly to the non-zero effective element
\begin{align*}
&  \left  ( \frac{ [ \Pic ^0 ( C ) ] \LL^{1-g}    }{ \LL - 1  } \right ) ^{ \rg (\Pic (U)) } 
\prod_{v\in C_1 } \left ( 
 (1-\LL^{-1})^{\rg (\Pic (U))} \frac{[\UUU_v]}{\LL^{n}} 
 \right) 
\times \prod_{v\in C_0 \setminus C_1}   \left  ( (1-\LL^{-1})^{\rg (\Pic (U)) }  \LL^{\rho^{\beta_v}  } \frac{[ E_{\beta_v}^\circ ]}{\LL^n }  \right ) \\
&  \times \prod_{s\in S}  \left ( (1  - \LL^{-1})^{|\mface_s |+ \rg ( \Pic ( U )) } \LL^{\rho^{\beta_s}  } \frac{ [  E_{\beta_s}^\circ \cap \DDD_{\mface_s } ]\LL^{|\mface_s |}  }{\LL^n} \right )
\end{align*}
and for every $\varepsilon \in \{ 0 , 1 \}^{\AAA_U} \setminus \{ \mathbf 0 \}$ 
\[
w \left  ( \tilde{\mathfrak a} _{\mm} ^{\beta , \varepsilon } 
  \LL^{ - \langle \rho ' , \mm \rangle_{\AAA_U \sqcup \mface} }  \right ) < - \delta \langle \varepsilon , \mm \rangle_{\AAA_U}   
\] 
when $\min ( m_\alpha ) \to \infty$.

\end{myptn}
Now we can adapt the case $\UUU = \XXX$ to control the contribution of $\HHH_\mface ( \UU ,0)$. One immediately gets a decomposition 
\[
\HHH_\mface ( \UU ,0)
= \sum_{ \substack{ \varnothing \neq A \subset \AAA_U \\ \varepsilon \in \{0,1\}^{\AAA_U \setminus A} \times \{ 1 \}^A }} 
[\Pic^0 ( C ) ]^{|\AAA_U \setminus A |}
\frac{\left (-\LL^{1-g}\right )^{ | \AAA_U  | - |\varepsilon |  }}{(1-\LL)^{|\AAA_U \setminus A |}}  \HHH_\mface  ^{A,\varepsilon} ( \UU , \xi ) 
\]
where for every $A \subset \AAA_U$ non-empty and $\varepsilon \in \{0,1\}^{\AAA_U \setminus A} \times \{ 1 \}^A  $
\[
\HHH_{\mface} ^{A,\varepsilon} ( \UU , 0 )  = \sum_{\mm \in \NN^{\AAA_U } \times \NN ^\mface }  \LL^{\langle \rho' - \varepsilon   , \mm \rangle_{\AAA_U \sqcup \mface }} 
\prod_{\alpha \in  A} 
\left ( 
[ \Sym^{m_\alpha } C] - [\Pic^0 ( C ) ] \frac{\LL^{m_\alpha - g - 1 } -1 }{\LL - 1 } 
\right ) \UU ^{\mm} . 
\]
Adapting the argument from the $\UUU = \XXX$ case, one deduces the following proposition. 
\begin{myptn}[$\UUU\subset \XXX$, $C_0 \subset C$, $\xi = 0$]  \label{proposition:general-case-H-A}
For all choice of maximal faces ${\mface \in\Cl_S^{\mathrm{max}} ( X , D )} $, there is a decomposition 
\[
\HHH_\mface ( \UU , 0 ) \left ( \prod_{v\in C_0} \FFF_v^\beta ( \UU_0 , 0  ) \right ) \left ( \prod_{s \in S} \PPP_{\mface_s}^{\beta_s} ( \UU_0 , \UU_s ) \right )=  \sum_{\substack{\varnothing \neq A \subset \AAA_U   \\ \varepsilon \in \{0,1\}^{\AAA_U \setminus A} \times \{ 1 \}^A }}  \sum_{\mm \in \NN ^{\AAA_U} \times \NN^\mface}  \left ( \mathfrak{h}_\mface^{\beta ,  A,\varepsilon} \right )_\mm \UU^\mm
\]
and a real number $\delta  > 0$
% (depending only on $\prod_{v\in C} F_v ( \TT , 0 ) $)
such that for all non-empty subset $A \subset \AAA_U $ and $\varepsilon \in \{0,1\}^{\AAA_U \setminus A} \times \{ 1 \}^A $ 
\[
w \left  (   \left ( \mathfrak{h}_\mface^{\beta ,  A,\varepsilon} \right )_\mm  \LL^{- \langle \rho  ' , \mm \rangle_{\AAA_U \sqcup \mface } }  \right ) < - \delta \left  (  \langle \varepsilon , \mm \rangle_{\AAA_U} +  \langle \rho ' , \mm    \rangle_{\AAA_U \sqcup \mface } \right )  
\]
when $\min ( m_\alpha ) \to \infty$. 
\end{myptn}

\subsubsection{$S$-integral points and non-trivial characters}

Concerning places $s\in S$, 
recall that the local factor of $Z^\beta ( \TT ) $ is given by the formula 
\[
Z_s^\beta ( \TT , \xi ) = \sum_{A\subset \AAA  }
 \LL^{\rho^\beta} 
\int_{\Omega_s ( A ,\beta )} \prod_{\alpha \in A} \left ( \LL^{\rho_\alpha } T_\alpha \right ) ^{\ord_s ( x_\alpha ) } \ee ( \langle x , \xi \rangle ) \mathrm{d}x  . 
\]
Let $\Cl_s ( X , D )_0$ be the subcomplex of $\Cl_s ( X , D )$ 
where we only keep vertices $\alpha \in \AAA$ such that $d_\alpha ( \xi ) = 0$. 
Then Proposition 5.3.1 of \cite{chambert2016motivic} can be rewritten as follows. 
\begin{myptn}\label{proposition:Zs(T-xi)-without-indeterminacies}
Assume first that $f_{\xi}$ extends to a regular map $\XXX_{F_s} \to \PP^1_{F_s}$.
Then there is a family $(P_{\mface_s^0}^{\beta_s} )_{\mface_s^0 \in \Cl^{\mathrm{max}} ( X , D )_0}$ of polynomials with coefficients in $\ExpM_k$ such that 
\[
Z_s^\beta ( \TT , \xi_s ) = \sum_{\mface_s ^0 \in \Cl_s^{\mathrm{max}} ( X , D )_0} \frac{P_{\mface_s^0 }^{\beta_s} ( \TT , \xi_s ) }{\prod_{\alpha \in \AAA_U^0 ( \xi ) \sqcup \mface_s ^0 } ( 1 - \LL^{\rho_\alpha  ' } T_\alpha  ) } . 
\]
\end{myptn} 
If $f_\xi$ does not extend to a regular map $\XXX_{F_s} \to \PP^1_{F_s}$ 
then the argument of Chambert-Loir and Loeser in \cite{chambert2016motivic} consists in using  a resolution of indeterminacies, that is a proper birational morphism $\varpi  : \YYY_{\mathscr O_s} \to \XXX_{\mathscr O_s}$ such that $\varpi ^* f $ extends to a regular map from $\YYY_{\mathscr O_s}$ to $\PP^1_{\mathscr O_v}$. Moreover one can assume that the generic fibre $\YYY_{F_s} \to \XXX_{F_s}$ is invariant under the action of $G_{F_v}$, and this can be done locally uniformly with respect to $\xi$ \cite[Lemma 6.5.1]{chambert2016motivic}.  
Let $\DDD_{\alpha , s} '$ be the strict transform of $\DDD_{\alpha , s}$ through $\varphi $ and $(\EEE_\gamma ) _{\gamma \in \Gamma }$ the set of horizontal exceptional divisors of $\varpi$. There exists a family $(m_{\gamma , \alpha })$ of nonnegative integers such that  
\[
\pi^* (\mathscr D_{\alpha , s})_{F_s} =( \mathscr D_{\alpha , s} ')_{F_s} + \sum_{\gamma \in \Gamma} m_{\gamma , \alpha }(  \mathscr E_{\gamma })_{F_s}
\]
as well as positive integers $\nu_\gamma $ such that 
\[
K_{\YYY_{F_s} / \XXX_{F_s}} = \sum_{\gamma \in \Gamma} \left ( \nu_\gamma -  1 \right ) (  \mathscr E_{\gamma })_{F_s}
\]
and $v$  a bounded constructible function on $\LLL ( \YYY_{\OOO_s} ) $ such that 
\[
\ord_{K_{\YYY_{\OOO_s} / \XXX_{\OOO_s} } } = \sum_{\gamma\in\Gamma} ( \nu_\gamma - 1 ) \ord_{\EEE_\gamma} + v . 
\]
Chambert-Loir and Loeser then use the change of variable formula \cite[Theorem 13.2.2]{cluckers2010constructible} 
and obtain an integral over $\LLL ( \YYY )$, which in our case gives 
\begin{align*}
Z_s ^\beta ( \TT , \xi_s ) & =  \sum_{A\subset \AAA } 
\LL^{\rho^\beta} \times
\\
& \times \sum_{\substack{ \pp \in \ZZ^\AAA \\  \text{(finite sum)}}} \LL^{\langle \rho ' , \pp \rangle} \TT^\pp \int_{W_\pp \cap \varpi^* \Omega_s ( A , \beta ) } 
\LL^{\langle \rho ' , \mathbf{ord}_{\mathscr D ' }  (y ) \rangle} 
\TT^{\mathbf{ord}_{\mathscr D ' }  (y )} \times \\ 
& \times \prod_{\gamma \in \Gamma } \left ( \LL^{\nu_\gamma - 1} \LL^{\langle \rho ' , \mm_\gamma \rangle} \TT^{\mm_\gamma}  \right )^{\mathbf{ord}_{\mathscr E_\gamma  }  (y )} \LL^{v(y)} \ee ( \varpi^* f_{\xi_s} ( y ) ) \dd y 
\end{align*}
where $(W_\pp ) $ is a finite partition of $\mathscr L ( \YYY ) $.

Then one applies Proposition \ref{proposition:Zs(T-xi)-without-indeterminacies} to $\varpi ^* f_{\xi_v}$ and to the integral over $W_\pp \cap \varpi^* \Omega_v ( A , \beta )$ to get a similar result: $Z_s ( \TT , \xi_s )$ is rational with denominators given by products of polynomials of the form $(1- \LL^{ \langle \rho ' , \aA \rangle }  \TT^{\aA }) $ for some $\aA \in \NN^\AAA$.  
One has to justify that this procedure does not change the set $\AAA$ and the relevant faces of Clemens complex. The argument is given in the very last paragraph of \cite{chambert2016motivic} and in \cite[\S 3.4]{chambert2012integral}: the $F_s$-Clemens complex $\Cl_{F_s} ( Y , Y \setminus \cup_{\alpha \in \AAA_U} ( \DDD_{\alpha, s} ')_{F_s} )$ has vertices coming from $X$, corresponding to the strict transforms $\DDD_\alpha ' $ for $\alpha \in \AAA_D$, and vertices corresponding to the divisors $\EEE_\gamma$ contracted by $\varpi$. The divisor of $\varpi^* f_\xi$ on $\YYY_{F_s}$ takes the form
\[
\Xi_\xi ' - \sum_{\alpha \in \AAA} d_\alpha ( \xi ) (  \DDD_{\alpha , s} ')_{F_s} - \sum_{\gamma \in \Gamma} e_\gamma (\EEE_\gamma )_{F_s}  
\]  
where $\Xi_\xi '$ is the strict transform of the hyperplane $\langle \xi , x \rangle = 0$. Since $Y_{F_s}$ is an equivariant compactification of $G_{F_s}$, 
the integers $e_\gamma$ are all non-negative. 
Furthermore
$ K_{\YYY_{F_s} / \XXX_{F_s}} $ is a linear combination of the $ ( \mathscr E_{\gamma })_{F_s} $ with non-negative coefficients. Consequently, looking back on the formula giving $Z_s^\beta ( \TT , \xi_s) $ above, one remarks that only the vertices coming from $X_{F_s}$ will contribute to the poles. 

Replacing $\xi = 0$ by $\xi \neq 0$, $\AAA_U $ by $\AAA_U^0 ( \xi )$ and $\Cl_s ( X , D )$ by $\Cl_s ( X , D )_0$ everywhere in the previous paragraph, 
including in the definitions of $\widetilde{\ZZZ_\mface^\beta}$, $\widetilde{\GGG_\mface}$ and $\widetilde{\HHH_\mface}$, 
one gets the following proposition. 
Since $\Cl_s^{\mathrm{max}} ( X , D )_0$ is a proper subset of $\Cl_s^{\mathrm{max}} ( X , D )$ for every $s\in S$ \cite[Lemma 3.5.4]{chambert2012integral},
this proposition implies that non-trivial characters do not contribute asymptotically, included in the case $\UUU \neq \XXX$. 
\begin{myptn} \label{proposition:general-case-xi-neq-zero} Let $\xi$ be a non-trivial character. 
For all $\mface^0 \in\Cl_S^{\mathrm{max}} ( X , D )_0 \subset \mathcal P ( \AAA_D ) ^S $,
there exist decompositions
\[
\widetilde{\ZZZ_{\mface^0}} ( \UU , \xi ) = 
\sum_{\mm \in \NN ^{\AAA_U} \times \NN^{\mface^0} } 
\left ( \tilde{\mathfrak a} _{\mface^0}^{\beta , \xi } \right )_\mm
 \UU^{\mm}  
=
\sum_{\varepsilon \in \{ 0 , 1 \}^{\AAA_U^0 ( \xi )  } }   \sum_{\mm \in \NN ^{\AAA_U} \times \NN^{\mface^0}  } 
\left ( \tilde{\mathfrak a} _{\mface^0}^{\beta , \xi , \varepsilon } \right )_\mm
 \UU^{\mm}  
\]
\[
\HHH_{\mface^0} ( \UU , \xi ) =  \sum_{\substack{\varnothing \neq A \subset \AAA_U^0  ( \xi )    \\ \varepsilon \in \{0,1\}^{\AAA_U^0 ( \xi ) \setminus A} \times \{ 1 \}^A }}  \sum_{\mm \in \NN ^{\AAA_U} \times \NN^{\mface^0} } \left ( \mathfrak{h}_{\mface^0}^{\beta ,  \xi , A,\varepsilon} \right )_\mm \UU^\mm
\]
as well as a real number $\delta_\xi > 0$ 
such that for all $\varepsilon \in \{ 0 , 1 \}^{\AAA_U^0 ( \xi ) } $
\[
w \left  ( \left ( \tilde{\mathfrak a} _{\mface^0}^{\beta ,\xi , \varepsilon } \right )_\mm \LL^{ - \langle \rho ' , \mm \rangle_{\AAA_U \sqcup \mface^0 }  }   \right ) < - \delta_\xi  \left  (  \langle \varepsilon , \mm \rangle_{\AAA_U^0 ( \xi )} +  \langle \rho ', \mm    \rangle \right )  
\] 
and for all $\varnothing \neq A \subset \AAA_U $ 
and $\varepsilon \in \{0,1\}^{\AAA_U^0  ( \xi ) \setminus A} \times \{ 1 \}^A $ 
\[
w \left  ( \left ( \mathfrak{h}_{\mface^0}^{\beta , \xi ,  A,\varepsilon} \right )_\mm \LL^{- \langle \rho ' , \mm \rangle_{\AAA_U \sqcup \mface^0} }    \right ) < - \delta \left  (  \langle \varepsilon , \mm \rangle_{\AAA_U^0  ( \xi ) } +  \langle \rho ' , \mm    \rangle \right )  
\]
when $\min ( m_\alpha ) \to \infty$.
%whenever $ \langle \varepsilon , \mm \rangle_{\AAA_U^0 ( \xi )} +  \langle \rho , \mm    \rangle_{\AAA \setminus \AAA_U^0 ( \xi )}$ is large enough.  

\end{myptn}
%%% SUBSECTION %%%

\subsection{Summation over all the characters and convergence} In this subsection we perform the final step of our proof: we have to permute a motivic summation with taking a limit. 
\label{subsection:uniform-convergence}

\subsubsection{Expression of $\left [   M_{\nn , \m_S}^\beta \right ]$ in $\mathscr M_k$}
The $\mm$-th coefficient of the  Zeta function $\ZZZ^\beta ( \TT ) $ 
is by definition the class $\left [  M_{\nn , \m_S}^\beta \right ]$ of the moduli space of section we are interested in. 
Recall that we have a decomposition 
\[
\ZZZ^\beta ( \UU ) = \LL^{n(1-g)} \left (  \ZZZ ^\beta (\UU, 0  ) + \sum_{\xi \in V \setminus \{ 0\} } \ZZZ^\beta ( \UU ,  \xi  ) \right ).
\]
where $V=L(\tilde E)^n$ is the $n$-th power of a finite dimensional Riemann-Roch space. 
The class of $M_\nn$ in $\mathscr M_k$ is given by 
\[
\left [   M_{\nn , \m_S}^\beta \right ] = \LL^{(1-g)n} \sum_{\xi \in V } \sum_{ \m \in S^{\nn_U^\beta} ( C )  } \theta_{\nn^\beta}^* \mathscr F ( \mathbf 1_{H(\m + \m_S , \beta )} ) (\xi ) 
\]
where the two sums here are motivic sums (given by the projections to $S^{\nn_U^\beta} ( C)$ and then to $k$), 
$\tilde{E}$ is the divisor defined by equation (\ref{def:divisor-E-tilde}) and 
\[
\theta_{\nn^\beta } : V \times S^{\nn^\beta } C \longrightarrow\AAA_{\nn^\beta} ( \nu - s , \nu - s ' , N , 0 )
\]
is the morphism introduced in  \S \ref{paragraph:poisson-families}. 
By definition, 
\[
 \sum_{ \m \in S^{\nn_U^\beta} ( C ) }   \theta_{\nn ^\beta }^* \mathscr F ( \mathbf 1_{H(\m + \m_S , \beta )} ) ( \xi ) \in \ExpM_V 
 \]
is the coefficient of degree $( \nn_U , \m_S )$ of 
\[
\prod_{v\in C} \ZZZ_v^{\beta_v} ( \UU , \xi ) . 
\] 
\subsubsection{Uniform convergence}

By Remark \ref{remark:uniform-convergence-finite-stratum}, we already know the existence of a finite constructible partition of $V \setminus \{ 0 \} $ over which the weight-linear convergence of $\prod_{v\in C_1} F_v ( \TT , \xi )$ with respect to $\rho$ is uniform in $\xi$. 
Concerning places in $S$, Proposition \ref{proposition:Zs(T-xi)-without-indeterminacies} and \cite[Lemma 6.5.1]{chambert2016motivic}
allows us to resolve indeterminacies of $f_\xi$ uniformly on a partition 
(which is actually finite for the same reason). 
Taking a partition refining both previous partitions, we get a finite partition $P$ of $V \setminus \{ 0\}$ over which $\xi \mapsto \AAA_U^0 ( \xi ) $ and $\xi \mapsto \AAA_U^1 ( \xi )$ are constant, as well as the corresponding degrees for $\alpha \in \AAA_D$,
which means that $\rho'$-weight-linear convergence is uniform over such a partition. This partition does not depend on a choice of $\beta$. 
This provides a decomposition 
\begin{align*}
\left [   M_{\nn , \m_S}^\beta \right ] &=\LL^{(1-g)n}  \sum_{\m \in S^{\nn_U^\beta} ( C ) }   \theta_{\nn ^\beta }^* \mathscr F ( \mathbf 1_{H(\m + \m_S , \beta )} ) ( 0 ) \\
 & + \LL ^{(1 -g)n}  \sum_{ \substack{ P\\ \text{stratum }} }     \sum_{\xi \in P }   \sum_{\m \in S^{\nn_U^\beta} ( C ) }   \theta_{\nn ^\beta }^* \mathscr F ( \mathbf 1_{H(\m + \m_S, \beta )} ) ( \xi ) .
\end{align*}
All the results of \S 4.2 hold uniformly on each stratum $ P$, which means that 
we can pass from the decomposition of $\prod_{v\in C} \ZZZ_v^{\beta_v} ( \UU , \xi ) $ given by Proposition \ref{proposition:general-case-xi-neq-zero} to a decomposition of 
\[
\sum_{\xi \in V} \prod_{v\in C} \ZZZ_v^{\beta_v} ( \UU , \xi )  
\]
from which one deduces a decomposition of its coefficient
\[
\sum_{\xi \in P }   \sum_{\m \in S^{\nn_U^\beta} ( C ) }   \theta_{\nn ^\beta }^* \mathscr F ( \mathbf 1_{H(\m + \m_S, \beta )} ) ( \xi ) 
\]
for every $(\nn_U , \m_S) \in \NN^{\AAA_U} \times \NN^{\AAA_D \times S}$. Summing over the finite set of stratum $P$ and adding the term coming from the trivial character, we get a decomposition of $\left [   M_{\nn , \m_S}^\beta \right ]$. 

\subsubsection{Final result}
We combine the uniform convergence argument above and the results of Propositions \ref{proposition:general-case-Z-tilde-xi=0}, \ref{proposition:general-case-H-A} and \ref{proposition:general-case-xi-neq-zero}, fixing a choice of vertical components $\beta \in \prod_v \BU_v$.  Recall that the limit we obtained does not depend on the choice of the maps $(A, \beta ) \mapsto \mface_s$ for every $s\in S$. 

\begin{myptn}\label{ptn:final-result}
Let $\mface  \in \Cl_S^{\mathrm{max}} ( X , D )  $ be any choice of maximal faces over the places $s\in S$ and $\beta \in \prod_v \BU_v$ a choice of vertical components. When $\min ( m_{\alpha , s} ) $ becomes arbitrary large, 
the coefficients of 
\[ 
\ZZZ_\mface^\beta ( \UU )  = \LL^{(1-g)n} \sum_{\xi \in V } \ZZZ_\mface^\beta ( \UU , \xi ) , 
\]
normalised by $\LL^{\langle \rho ' , \mm \rangle}$,
converges in $\widehat{\mathscr M_k } $ to a non-zero effective element of $\widehat{\mathscr M_k}$ which can be written as the following motivic Euler product:
\begin{align*}
& \LL^{(1-g)n  }   
\left  ( \frac{ [ \Pic ^0 ( C ) ] \LL^{1-g}    }{ \LL - 1  } \right ) ^{ \rg (\Pic (U)) }\\  
& \times\prod_{v\in C_1 } \left ( 
 (1-\LL^{-1})^{\rg (\Pic (U))} \frac{[\UUU_v]}{\LL^{n}} 
 \right) 
\times \prod_{v\in C_0 \setminus C_1}   \left  ( (1-\LL^{-1})^{\rg (\Pic (U)) }  \LL^{\rho^{\beta_v}  } \frac{[ E_{\beta_v} ^\circ ]}{\LL^n }  \right ) \\
&  \times \prod_{s\in S}  \left (
  ( 1 - \LL^{-1} )^{\rg (\Pic (U)) + |\mface_s| } 
  \LL^{\rho^{\beta_s}   } \frac{  [  E_{\beta_s}^\circ \cap \DDD_{\mface_s } ] \LL ^{| \mface_s |} }{\LL^n} \right ).
\end{align*} 
Furthermore, there exists a finite partition $( P )$ of $V\setminus \{ 0 \}$ such that the normalised coefficients of $\ZZZ_{\mface} ^ \beta (\UU , \xi )$ tends to zero weight-linear-uniformly with respect to $\rho$ on each stratum $P$. 
\end{myptn}

\begin{myremark}\label{remark:interpretation-local-term-volume-integral-points}
Over places $v\in C_0$,
as we pointed out in Remark \ref{remark:interpretation-local-term-volume}, the local term $\LL^{\rho^{\beta_v} - \langle \rho , \ee^{\beta_v} \rangle } \frac{[ E_{\beta_v} ^\circ ]}{\LL^n } $ can be interpreted as the motivic integral  
\[
\int_{G(F_v , \beta_v )} \LL^{-(g, \LLL_{\rho ' } ) } | \omega_X | =
\int_{\LLL ( \XXX_v , E_{\beta_v}^\circ ) } \LL^{-\ord_{\LLL_{\rho ' }} ( x  ) } \LL^{- \ord_{\omega_{X}} ( x  )  } .
\]
This remark can be adapted for places $s\in S$ if one replaces $E_{\beta_v}^\circ $ by $E_{\beta_v}^\circ  \cap \DDD_{\mface_s}$ and add a controling term. Let $\epsilon$ be a non-negative integer, $\epsilon_\alpha = \epsilon $ if $\alpha \in \mface_s$ and zero otherwise, and 
\[
\LLL_{\rho ' + \epsilon} = \LLL_{\rho ' } + \sum_{\alpha \in \AAA } \epsilon_\alpha \LLL_\alpha .
\]
Then one considers for $\epsilon \geqslant 1$
\begin{align*}
& \int_{\LLL ( \XXX_s , E_{\beta_s}^\circ \cap \DDD_{\mface_s}) } \LL^{-\ord_{\LLL_{\rho ' + \epsilon }} ( x  )  } \LL^{- \ord_{\omega_{X}} ( x  )  } \\ 
& =  \int_{\LLL ( \XXX_s , E_{\beta_s}^\circ \cap \DDD_{\mface_s} ) }  \LL^{-\sum_{\alpha \in \AAA_U \sqcup \mface_s} ( \rho_\alpha ' + \epsilon_\alpha ) \left ( \ord_{\DDD_\alpha} ( x ) + e_\alpha^{\beta_s} \right )} \LL^{\rho^{\beta_s} + \sum_{\alpha \in \AAA_U \sqcup \mface_s} \rho_\alpha  \ord_{\DDD_\alpha} ( x )} \\
 & = \LL^{\rho^{\beta_s} - \langle \rho ' - \epsilon , \mathbf e^\beta \rangle} \int_{\LLL ( \XXX_s , E_{\beta_s}^\circ \cap \DDD_{\mface_s} ) } \LL^{ \sum_{\alpha \in \mface_s} ( 1 - \epsilon_\alpha ) \ord_{\DDD_\alpha} (x  ) }  \\
 & = \LL^{\rho^{\beta_s} - \langle \rho ' - \epsilon , \mathbf e^\beta \rangle} \sum_{A\subset \AAA_U } \int_{ \Omega_s ( \mface_s \cup A , \beta )} \LL^{ \sum_{\alpha \in \mface_s}(1 - \epsilon_\alpha ) \ord_{\DDD_\alpha} (x  ) } .
\end{align*}
This last family of integrals can be computing using the isomorphism 
\[
 \Theta : 
\Delta_v ( A , \beta ) \times \LLL ( \mathbf A^1 , 0 ) ^A \times \LLL ( \mathbf A^1 , 0 )^{n-|A|}  \longrightarrow  \Omega_v ( A , \beta )
\]
introduced at the beginning of \S \ref{section:compactification-of-additive-groups}, together with the motivic volumes computed in \S \ref{subsubsection:motivic-volume}:
\begin{align*}
& \int_{ \Omega_s ( A \cup \mface_s  , \beta )} \LL^{ \sum_{\alpha \in \mface_s} ( 1 - \epsilon_\alpha ) \ord_{\DDD_\alpha} (x  ) }  \\
& = \sum_{\mm \in \NN_{>0}^{A\cup \mface_s}} \int_{\substack{ \Delta_s ( A  \cup \mface_s, \beta ) \times \LLL ( \mathbf A^1 , 0 ) ^{A \cup \mface_s} \times \LLL ( \mathbf A^1 , 0 )^{n-|A|-|\mface_s|} \\ \ord_s (x_\alpha ) = m_\alpha }  } \LL^{|(1 - \epsilon) \mm_{\mface_s} | }  \\
& = \sum_{\mm \in \NN_{>0}^{A\cup \mface_s}} [ \Delta_s ( A \cup \mface_s , \beta ) ] \LL^{ | ( 1 - \epsilon ) \mm_{\mface_s} | } ( 1 - \LL^{-1} )^{|A\cup \mface_s|} \LL^{-|\mm |} \LL^{-n+|A|+|\mface_s|} \\
& = [ \Delta_s ( A \cup \mface_s  , \beta ) ] \left ( \frac{\LL^{-\epsilon}}{1-\LL^{-\epsilon}} \right )^{|\mface_s |} \left ( \frac{\LL^{-1}}{1-\LL^{-1}} \right )^{|A|}  ( 1 - \LL^{-1} )^{|A\cup \mface_s|} \LL^{-n+|A|+|\mface_s|} \\
& = \left ( \frac{\LL^{-\epsilon}}{1-\LL^{-\epsilon}} \right )^{|\mface_s |}  ( 1 - \LL^{-1} )^{|\mface_s |} [ \Delta_s ( A \cup \mface_s , \beta ) ] \LL^{|\mface_s| - n } . 
\end{align*} 
This last equality shows that it makes sense to consider the value at $\epsilon = 0$ of 
\[
( 1 - \LL^{-\varepsilon})^{|\mface_s|} 
\int_{\LLL ( \XXX_s , E_{\beta_s}^\circ \cap \DDD_{\mface_s}) } \LL^{-\ord_{\LLL_{\rho ' + \epsilon }} ( x  )  } \LL^{- \ord_{\omega_{X}} ( x  )  } .
\] 
Since $[E_{\beta_v}^\circ  \cap \DDD_{\mface_s}] = \sum_{A\subset \AAA_U }  [ \Delta_s ( A \cup \mface_s , \beta ) ] $, this residue is precisely the local term 
\[
\LL^{\rho^{\beta_s} - \langle \rho ' , \mathbf e^\beta \rangle}  ( 1 - \LL^{-1} )^{|\mface_s |} [E_{\beta_v}^\circ  \cap \DDD_{\mface_s}]  \LL^{|\mface_s| - n } . 
\] 
\end{myremark}

\begin{myremark}\label{remark:moduli-space-curves-max-faces} 
A geometric interpretation of Proposition \ref{ptn:final-result} is as follows. 

For any ${ \mface \in \Cl_S^{\mathrm{max}} ( X , D ) }$, let $M_{\nn , \m_S}^{\beta , \mface }\subset  M_{\nn , \m_S}^\beta$ be the constructible moduli subspace of sections $\sigma \in M_{\nn , \m_S}^\beta$ such that 
\[
\forall s \in S \quad  \forall \alpha \in \mface_s \quad ( \sigma ( \eta_C ) , \DDD_\alpha )_s > 0   . 
\]
% Remark that if $ M_{\nn, \beta ,  \mface} \neq \varnothing $ then 
% $\nn^\beta \in \NN^{\AAA_\mface}$ and 
% $\langle \rho ' , \nn \rangle_{\AAA_\mface} = \langle \rho ' , \nn^\sigma \rangle - \langle \rho ' , \ee^\beta \rangle_{\AAA\setminus \AAA_\mface}$. 
Remark that by definition of the Clemens complexes, $M_{\nn,  \m_S }^{\beta , \mface} $ is non-empty if and only if the support of $\m_s = \left ( ( \sigma ( \eta_C ) , \DDD_\alpha )_s \right )_{\alpha \in \AAA_D} \in \NN^{\AAA_D}$ is \textit{exactly} $\mface_s$ for all $s\in S$. The maximal number of $\mathscr D_\alpha$ a section can intersect above each point of $S$ is encoded in $\Cl_S^{\mathrm{max}} ( X , D ) $. 
This also means that when the multidegree of a section tends to infinity, there is an unique way to associate a maximal face to the section and 
$M_{\nn , \m_S}^{\beta , \mface }$ is actually the \textit{right} moduli space to consider. We showed that $[M_{\nn , \m_S}^{\beta , \mface }]\LL^{-\langle \nn_U^\beta + \m_S , \rho ' \rangle }$ converges to the constant of Proposition \ref{ptn:final-result}. Since $[M_{\nn , \m_S}^{\beta , \mface }]$ coincides with the coefficient of sufficiently high degree of the refined height zeta function $\ZZZ^\beta_\mface ( \UU ) $, this is the following statement and main result.
\end{myremark}

% Above each place we can choose the maximal face associated to $A\subset \AAA$ and $\beta $ to be $\mface_s$ whenever $A \cap \AAA_D \subset  \mface_s $. If we do so, then the motivic function associated to $[M_{\nn, \beta, \mface} ] $ is actually $Z_{\mface} ( \TT , \beta )$, and the previous result can be interpreted as the asymptotic behaviour of  $[M_{\nn, \beta, \mface} ] \LL^{-\langle \rho ' , \nn^\beta \rangle_{\AAA_\mface} }$, from which one deduces the one of $[M_{\nn, \mface} ] \LL^{-\langle \rho ' , \nn \rangle _{\AAA_\mface}}$.  

\begin{mythm}\label{thm:final-result}
When the multidegree $(\nn_U^\beta , \m_S ) \in \NN^{\AAA_U} \times \NN^{\mface}  $ goes arbitrary far from the boundaries of the corresponding cone, the normalised class
\[
[M_{\nn , \m_S}^{\beta , \mface }] \LL^{- \langle  \rho ', \nn \rangle } 
\]
converges, in the completion of $\mathscr M_k $ for the weight topology, to the motivic Euler product
\begin{align*}
& \LL^{(1-g)n  }   
\left  ( \frac{ [ \Pic ^0 ( C ) ] \LL^{1-g}    }{ \LL - 1  } \right ) ^{ \rg (\Pic (U)) }\\  
& \times\prod_{v\in C_1 } \left ( 
 (1-\LL^{-1})^{\rg (\Pic (U))} \frac{[\UUU_v]}{\LL^{n}} 
 \right) 
\times \prod_{v\in C_0 \setminus C_1}   \left  ( (1-\LL^{-1})^{\rg (\Pic (U)) }  \LL^{\rho^{\beta_v} - \langle  \rho ' , \ee^{\beta_v} \rangle  } \frac{[ E_{\beta_v} ^\circ ]}{\LL^n }  \right ) \\
&  \times \prod_{s\in S}  \left (
  ( 1 - \LL^{-1} )^{\rg (\Pic (U)) + |\mface_s| } 
  \LL^{\rho^{\beta_s}  - \langle  \rho ' , \ee^{\beta_s} \rangle  } \frac{  [  E_{\beta_s}^\circ \cap \DDD_{\mface_s } ] \LL ^{| \mface_s |} }{\LL^n} \right ).
\end{align*} 
\end{mythm}

\section*{Acknowledgements}
I am grateful to Emmanuel Peyre for his help, reading and useful comments throughout the drafting process of this article. I am also very indebted to Margaret Bilu for all the constructions and properties that I used in this paper and which are due to her, especially those concerning the motivic Euler product, as well as for enlightening discussions and remarks on an earlier version of this work. 
I thank the anonymous referee for his/her remarks and suggestions that helped me to enhance the clarity of the exposition. 

%\bibliographystyle{siam}
%\bibliography{compact-ev_biblio}

% \printbibliography

\end{document}